\documentclass{amsart}[12pt]
\usepackage{graphicx} % Required for inserting images
\usepackage[dvipsnames]{xcolor}
\usepackage[inline]{enumitem}
\setlist[enumerate]{leftmargin=*,widest=0}

\usepackage[normalem]{ulem} %we need normalem to use emph
\usepackage{mathrsfs}
\usepackage{amsthm}
\usepackage{amsmath}
\usepackage{amssymb}
\usepackage{mathtools}
\usepackage{braket}
\usepackage{overpic}
\usepackage{tcolorbox} 
\usepackage{tikz-cd}
\usepackage{adjustbox}

\usepackage{array}
\usepackage{url}
\usepackage[colorlinks,citecolor=blue,linkcolor=red!80!black]{hyperref}
\usepackage[nameinlink]{cleveref}

\DeclareEmphSequence{\bfseries} %bold in emph braces
\theoremstyle{plain}
\newtheorem{thm}{Theorem}[section]

\newtheorem*{lem*}{Lemma}
\newtheorem{prop}[thm]{Proposition}

\newtheorem*{cor*}{Corollary}

\newtheorem*{claim*}{Claim}
\newtheorem*{thm*}{Theorem}
\newtheorem{question}[thm]{Question}

\theoremstyle{definition}
\newtheorem{definition}[thm]{Definition}

\theoremstyle{remark}

\newtheorem{obs*}[thm]{Observation}

\newtheorem{rmk*}[thm]{Remark}

\theoremstyle{plain}

\theoremstyle{definition}
\newtheorem{DEF}[thm]{Definition}

\theoremstyle{plain}
\newtheorem{theorem}[thm]{Theorem}

\newtheorem{lemma}[thm]{Lemma}
\newtheorem*{lemma*}{Lemma}
\newtheorem{proposition}[thm]{Proposition}
\newtheorem{corollary}[thm]{Corollary}
\newtheorem*{corollary*}{Corollary}

\newtheorem*{theorem*}{Theorem}

\newtheoremstyle{named}{}{}{\itshape}{}{\bfseries}{.}{.5em}{\thmnote{#3}}
\theoremstyle{named}
\newtheorem*{namedtheorem}{Theorem}

\theoremstyle{definition}

\theoremstyle{remark}
\newtheorem*{remark}{Remark}
\newtheorem{remark*}[thm]{Remark}

\DeclareMathOperator{\Aut}{Aut}
\DeclareMathOperator{\Out}{Out}

\DeclareMathOperator{\Map}{Map}

\DeclareMathOperator{\Homeo}{Homeo}

\newcommand{\Asterisk}{\mathop{\scalebox{1.5}{\raisebox{-0.2ex}{$\ast$}}}}

\renewcommand{\a}{a}
\renewcommand{\b}{b}
\renewcommand{\c}{c}
\renewcommand{\d}{d}

\newcommand{\F}{\mathbb{F}}

\newcommand{\R}{\mathbb{R}}

\newcommand{\lk}{\mathrm{link}}

\DeclareMathOperator{\Ends}{Ends}

\newcommand{\ie}{\textit{i.e.}}

\newcommand{\G}{\Gamma}

\newcommand{\sph}{\mathcal{S}}
\DeclareMathOperator{\nbd}{nbd}

\title{Automorphisms of the sphere complex of an infinite graph}
\author[Hill, Kopreski, Rechkin, Shaji, \& Udall]{Thomas Hill, Michael C.\ Kopreski, Rebecca Rechkin, George
Shaji, and Brian Udall}

\begin{document}

%%%%%%%%%%%%%%%%%%%%%%%%%%%%%%%%%%%%% 
\begin{abstract}
  For a locally finite, connected graph $\Gamma$, let
$\operatorname{Map}(\Gamma)$ denote the group of proper homotopy equivalences of
$\Gamma$ up to proper homotopy.  
Excluding
sporadic cases, we show
$\operatorname{Aut}(\mathcal{S}(M_\Gamma)) \cong \operatorname{Map}(\Gamma)$, where $\mathcal{S}(M_\Gamma)$ is
the sphere complex of the doubled handlebody $M_\Gamma$ associated
to $\Gamma$.  We also construct an exhaustion of $\mathcal{S}(M_\Gamma)$
by finite strongly rigid sets when $\Gamma$ has finite rank and
finitely many rays, and an
appropriate generalization otherwise. 
\end{abstract}
  
\maketitle

\section{Introduction}
Let $\Gamma$ be a locally finite, connected graph, 
and let $\Map(\Gamma)$ denote its \emph{mapping class group},
defined by Algom-Kfir--Bestvina 
to be the group of proper homotopy equivalences of
$\Gamma$ up to proper homotopy. 
Let $M_\Gamma$ denote the doubled handlebody associated to $\Gamma$ (see \Cref{def:DoubHandle0}).
We prove in that the sphere complex $\sph(M_\Gamma)$ 
(see \Cref{def:sphcpx}) satisfies an Ivanov-type theorem.   

\begin{theorem} \label{thm:IvanovRigidity0}
    Let $\Gamma,\Gamma'$ 
    be two locally finite connected graphs. 
  Suppose $f: \sph(M_{\G})\to \sph(M_{\G'})$ is an
isomorphism. Then $f$ is induced by a diffeomorphism $h:M_{\G}
\to M_{\G'}$. In particular, 
\begin{enumerate*}[label=(\roman*)]
  \item $\G$ and $\G'$ are proper homotopy equivalent and
  \item when $\Gamma$ is not a graph of rank $r$ with $s$ rays such that $2r+s<4$ or $(r,s)\in \{(0,4), (2,0)\}$,  $\Aut(\sph(M_{\G}))\cong
    \Map(\G)$ as topological groups. 
\end{enumerate*}
\end{theorem}

Our proof adapts an observation from 
Bavard--Dowdall--Rafi
\cite{bavard2020isomorphisms} that the link of a sphere system 
$\sigma \subset \sph(M_\Gamma)$ is isomorphic to the join of the
$\sph(M_i)$ for components $M_i \subset M_\Gamma \setminus \sigma$. 
For a maximal
sphere system $\sigma$, we obtain an isomorphism of dual graphs
$\Delta_\sigma \to \Delta_{f\sigma}$ from which we construct
the diffeomorphism $h$.

\medskip
In addition, we generalize the results of Bering--Leininger
\cite{bering2024finite} to $\sph(M_\Gamma)$.

\begin{theorem}\label{thm:geom_rigidity}
  Suppose $\Gamma$ has rank $r$ and $s$ rays such that $6 \leq 2r +
  s < \infty$.  Then there exists an exhaustion of $\sph(M_\Gamma)$
  by finite strongly rigid subcomplexes.  
\end{theorem} 

\noindent
From \Cref{thm:geom_rigidity} and techniques from
\cite{bavard2020isomorphisms}, we obtain another proof of
\Cref{thm:IvanovRigidity0}. Finally, 
we extend \Cref{thm:geom_rigidity} to the infinite-type setting 
using techniques from the first proof of 
\Cref{thm:IvanovRigidity0}. We say a graph $\Gamma$ is of
\textbf{infinite-type} if it is not proper homotopy equivalent to a
finite rank graph with finitely many ends.

\begin{theorem} \label{thm:sr_exhaustion}
  For $\Gamma$ 
  infinite-type, there exists an exhaustion of $\sph(M_\Gamma)$
  by topologically locally finite subcomplexes 
  that are strongly rigid over maximal maps.
  \end{theorem}

  A subcomplex $X\subset \sph(M_\Gamma)$ is 
\emph{strongly rigid (\textit{resp.}\ over maximal maps)} 
if any locally
injective simplicial map $X \to \sph(M_{n,s})$ 
(\textit{resp.}\ preserving
the maximality of sphere systems) extends uniquely to an
automorphism of
$\sph(M_\Gamma)$. The subcomplex $X$ is \emph{topologically locally
finite} if every compact set $K \subset M_\Gamma$ intersects
finitely many vertices of $X$.

\subsection{Motivation}

Recall that $\Out(\F_n)$ is the group of outer automorphisms of the
free group $\F_n$, defined as $\Out(\F_n) \coloneq 
\operatorname{Aut}(\F_n) / \operatorname{Inn}(\F_n)$.  From a
topological perspective, $\Out(\F_n)$ can be thought of as
$$\Out(\F_n) =\set{ \text{homotopy equivalences } \Gamma \to
\Gamma}/\:\text{homotopy}$$ where $\Gamma$ is a finite graph of rank
$n$ (see \cite[Proposition 1B.9]{hatcher2002algebraic}).  There is a
rich dictionary between mapping class groups of surfaces, $\Map(S)$,
and $\Out(\F_n)$ (see \cite{bestvina2019notes}).  Analogies between
the two include:

\smallskip
\begin{center}
\begin{tabular}{ r c l  } 
  $\Map(S)$ & $\longleftrightarrow$ &
  $\Out(\F_n)$  \\ 
  $\operatorname{Teich}(S)$ &
  $\longleftrightarrow$& Outer space,
 $\operatorname{CV}_n$\\ 
 curve complex & $\longleftrightarrow$ &
 $\left\lbrace\phantom{\parbox{0cm}{a\\b\\c}}\right.$\parbox{3.5cm}{sphere complex \\
   free factor complex \\
   free splitting complex}
 \begin{tabular}{l}
    \end{tabular}
\end{tabular}
\end{center}
\smallskip

Within the last decade, there has been a surge of interest in
\emph{big} mapping class groups of surfaces, that is,
$\mathrm{Map}(S)$ for a surface $S$ whose fundamental group is
not finitely generated. Motivated by the parallels between graphs
and surfaces, a natural question is: what is the big version of
$\mathrm{Out}(\mathbb{F}_n)$? Generalizing the interpretation of
$\mathrm{Out}(\mathbb{F}_n)$ as homotopy equivalences
of a graph up to homotopy, \cite{AB2021} propose the definition of
the mapping class group of a locally finite, infinite graph $\Gamma$
to be:
$$\Map(\Gamma) := \{\text{proper homotopy equivalences } \Gamma
\to \Gamma\}/\text{proper homotopy}.$$ The group $\Map(\Gamma)$ is
sometimes referred to as ``big $\Out(\F_n)$.''  When $\Gamma$ is a
finite graph of rank $n$, these definitions coincide:
$\mathrm{Out}(\mathbb{F}_n) \cong \Map(\Gamma)$.
\cite{AB2021,domat2023coarse,domat2023generating,
udall2024spherecomplexlocallyfinite} have demonstrated that
$\Map(\Gamma)$ exhibits many similarities with big mapping class
groups of surfaces. 

The curve complex $\mathcal{C}(S)$ is one of the most important
tools for studying mapping class groups of surfaces. A celebrated
theorem of Ivanov states that $\Aut(\mathcal{C}(S)) \cong
\Map^{\pm}(S)$, illustrating an underlying connection between the
curve complex and the surface mapping class group
\cite{ivanov1997automorphisms, luo1999automorphisms}. Ivanov's
theorem has inspired a broader metaconjecture: ``every object
naturally associated to a surface $S$ and possessing a sufficiently
rich structure has $\Map^{\pm}(S)$ as its group of automorphisms''
\cite{ivanov2006problems}. There have been many subsequent results
supporting this metaconjecture (see \cite{brendle2019normal} for
further discussion).  

When $S$ is an infinite-type surface, 
the curve complex is geometrically uninteresting 
(it has diameter 2); however, it is sufficiently rich
\textit{combinatorially} that
Ivanov's theorem still holds 
%$\mathrm{Aut}(\mathcal{C}(S)) \cong \Map^\pm(S)$
\cite{hernandez2018isomorphisms, bavard2020isomorphisms}. 
The aim of this project is to introduce a complex that plays a
similar role in the setting of big $\Out(\F_n)$. Specifically, we
will show that the sphere complex of a 3-manifold associated to a
locally finite graph $\Gamma$, denoted by
$\sph(M_\Gamma)$, satisfies an analog of Ivanov's theorem (see
\Cref{thm:IvanovRigidity0}).  When $\Gamma$ has rank zero,
$\Map(\Gamma) \cong \Homeo(\Ends(\Gamma))$ and 
\Cref{thm:IvanovRigidity0}
coincides with a recent result of Branman--Lyman, who show the
homeomorphism group of a Stone space is isomorphic to the
automorphism group of its complex of cuts \cite[Theorem
A]{branman2024complex}. Our proofs provide an alternate perspective
on this result.

Bavard--Dowdall--Rafi \cite{bavard2020isomorphisms} ultimately 
apply Ivanov's theorem 
to prove algebraic rigidity of the mapping class groups of
infinite-type surfaces, showing that 
two infinite-type surfaces $S$ and $S'$ are homeomorphic if and only
if $\Map^\pm(S)\cong\Map^\pm(S')$.  In contrast, algebraic rigidity
fails in general for locally finite infinite graphs, though all
known examples are when both graphs are trees. For example, if
$\Gamma$ has end space a Cantor set, and $\Gamma'$ has end space a
Cantor set with one extra isolated point, then $\Map(\G)\cong
\Map(\G')$ (see \Cref{thm:ADMQ}). A natural next question asks 
what conditions are needed to guarantee algebraic
rigidity in the big $\Out(\mathbb{F}_n)$ setting.  

\subsection{Outline of paper.}  
In \Cref{sec:Prelim}, we define the sphere complex of the doubled
handlebody associated to a locally finite infinite graph $\Gamma$,
discuss some general tools for studying spheres in 3-manifolds, and
recall relevant results from \cite{bering2024finite}.  In
\Cref{sec:geom_rigid}, we prove \Cref{thm:IvanovRigidity0} by first
showing that sphere graph isomorphisms induce isomorphisms of the
dual graphs of pants decompositions of the doubled handlebody
$M_\Gamma$.  These graph isomorphisms define diffeomorphisms
of $M_\Gamma$, which are compatible with pants decompositions that
differ by a flip move. 
In \Cref{sec:finite_rigidity} we generalize the results of
\cite{bering2024finite} to finite-type doubled handlebodies with
$S^2$ boundaries by proving \Cref{thm:geom_rigidity}, with the aim
of giving an alternative proof of \Cref{thm:IvanovRigidity0} in
\Cref{sec:mainthm} in the style of \cite{bavard2020isomorphisms}.
In \Cref{sec:loc_strongly_rigid} we prove \Cref{thm:sr_exhaustion},
generalizing a finite-type result of \cite{bering2024finite}.
Finally, in \Cref{sec:sporadic}, we consider the existence of finite
rigid sets in the low complexity cases not covered by
\Cref{thm:geom_rigidity}. 

\subsection*{Acknowledgements} 

The authors express their gratitude to Sanghoon Kwak, Vivian He, and
Hannah Hoganson
for their insightful conversations, which initially sparked our
interest in this project. We also extend our appreciation to George
Domat and Chris Leininger for valuable discussions, 
and special thanks
to Mladen Bestvina, whose thoughtful input greatly contributed to
the project's development.  We would also like to thank the
organizers of the 2023 Wasatch Topology Conference and 2024 Log
Cabin Workshop for providing a collaborative environment that helped
shape the direction of this work.

The first, second, and third authors acknowledge support from RTG
DMS-1840190. The third author also acknowledges support from NSF
grant DMS–2046889. 
The second and fourth authors acknowledge support from NSF
grant DMS-2304774. The fifth author acknowledges support from NSF
grant DMS-1745670.

\section{Preliminaries}\label{sec:Prelim}

Let $\Gamma$ be a locally finite graph.  A proper map $f : \Gamma
\to \Gamma$ is a \emph{proper homotopy
equivalence} if there exists a proper map $g: \G \to
\G$ such that $fg$ and $gf$ are properly homotopic to the identity.

We denote by $\Ends(\Gamma)$ the (Freudenthal) \emph{end space} of
$\Gamma$. We recall that given a compact
exhaustion $K_1 \subset K_2 \subset \dots$ of $\Gamma$, with maps
$\pi_0(\Gamma \setminus K_i) \to \pi_0(\Gamma \setminus K_{i-1})$
induced by inclusion, $\Ends(\Gamma)$ is the inverse limit 
$\varprojlim_i \:  \pi_0 (\Gamma \setminus K_i)$.
Like for surfaces, ends of graphs come in two flavors: those
accumulated by loops (also called unstable) and those that are not
accumulated by loops (also called stable ends).  Heuristically, 
if you always see loops as you move out toward an end, then it is
accumulated by loops.  The (possibly empty) 
subset of all ends accumulated by loops
is a closed subset of $\Ends(\Gamma)$, denoted $\Ends_\ell(\Gamma)$.

The rank and end space classify locally finite, infinite graphs up
to proper homotopy equivalence. 

\begin{theorem}[\cite{ayala1990proper}]\label{thm:ADMQ} 
Let $\G$ and $\G'$ be two locally finite, infinite graphs.  Then
$\Gamma$ is \emph{properly homotopy equivalent} to $ \Gamma'$ if and
only if $$(\mathrm{rk}(\Gamma), \Ends(\Gamma), 
\Ends_\ell(\Gamma)) \cong
(\mathrm{rk}(\Gamma'), \Ends(\Gamma'), \Ends_\ell(\Gamma')).$$ 
By $``\cong$'' we mean that $\operatorname{rk}(\Gamma) =
\operatorname{rk}(\Gamma')$ 
and there exists a homeomorphism of pairs $f :
(\Ends(\Gamma),\Ends_\ell(\Gamma)) \to
(\Ends(\Gamma),\Ends_\ell(\Gamma))$.
The tuple $(\operatorname{rk}(\Gamma), \Ends(\Gamma), 
 \Ends_\ell(\Gamma))$ is called the \emph{charateristic triple.} 
\end{theorem}

Let $\text{PHE}(\Gamma)$ denote the group of proper homotopy
equivalences of $\Gamma$, equipped with the compact-open topology.

\begin{DEF}
    The \emph{mapping class group} of $\G$, denoted
    $\Map(\G)$, is defined as
\begin{equation*}
    \Map(\G)\coloneq\text{PHE}(\G)/\text{proper homotopy}.
\end{equation*}
\end{DEF}

\noindent
$\Map(\Gamma)$ is a topological group with the quotient topology.
Not every homotopy equivalence of $\Gamma$ that is proper is a
proper homotopy equivalence as defined above (see
\cite[Example 4.1]{AB2021}).

We now construct the doubled handlebody associated to $\Gamma$,
whose study will be the main focus of this paper.  
\begin{definition}\label{def:DoubHandle0}
  Let $N_\Gamma$ be
the $3$-manifold with $0$-handles and $1$-handles glued according to
the vertices and edges, respectively, in $\Gamma$.  The
\emph{doubled handlebody $M_\Gamma$ associated to $\Gamma$} 
is the double of $N_\Gamma$, obtained by gluing two copies of
$N_\Gamma$ along $\partial N_\Gamma$.
\end{definition}

One can think of $N_{\G}$ as a regular
neighborhood of the image of a proper embedding of $\G$ in $\R^3$
(see \Cref{fig:DoubHandle}).  When $\Gamma$ is finite of rank $n$,
$M_\Gamma \cong \mathop{\#}_n (S^2 \times S^1)$, 
the connect sum of $n$ copies of $S^2
\times S^1$.  Let $M_{n,s}$ denote $\mathop{\#}_n (S^2 \times S^1)$
with $s$ disjoint open balls removed; note that $M_{n,s}\setminus
\partial M_{n,s} \cong M_\Gamma$ for $\Gamma$ rank $n$ with $s$
rays. 

\begin{figure}[ht!]
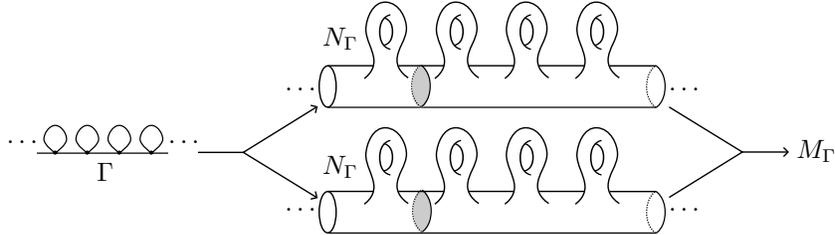

    \centering
    \begin{overpic}[width=10cm]{figs/DoubHandle.png}
        \put(-4,12){\dots}
        \put(17.5,12){\dots}
        \put(8,7){$\Gamma$}
        \put(38,25){$N_\Gamma$}
        \put(38,8){$N_\Gamma$}
        \put(33,19){\dots}
        \put(33,3){\dots}
        \put(101,10){$M_\Gamma$}
        \put(84,19){\dots}
        \put(84,3){\dots}
    \end{overpic}
    \caption{The construction of $M_\Gamma$ for the 
    ladder graph. 
  Disk cross-sections in distinct copies of $N_\Gamma$ 
form hemispheres of a 2-sphere in $M_\Gamma$.  } 
    \label{fig:DoubHandle}
\end{figure}

As usual, we define the \emph{mapping
class group} of $M_{\G}$ to be 
\begin{equation*}
    \text{Map}(M_{\G})=\text{Diff}^+(M_{\G})/\text{isotopy}
\end{equation*}
where $\text{Diff}^+(M_{\G})$ denotes the orientation preserving
diffeomorphisms of $M_{\G}$.  

\begin{remark}
  We will generally work in the smooth
category: submanifolds are assumed to be smoothly embedded, and homotopies and isotopies are smooth.
\end{remark}

\subsection{Spheres in $3$-manifolds}\label{sec:spheres}

The \emph{sphere complex} on $M_\Gamma$ is analogous to the curve
complex on surfaces.  An embedded sphere in a $3$-manifold $M$ is
\emph{essential} if it does not bound a ball and is not peripheral,
\textit{i.e.}\ 
not isotopic into a small neighborhood of a boundary sphere or
puncture of $M$.  

\begin{DEF}
\label{def:sphcpx}
The \emph{sphere complex} associated with a 3-manifold $M$ 
is denoted by $\sph(M)$.  It is a simplicial complex with
\begin{itemize}
    \item \textbf{vertices} corresponding to isotopy classes of essential embedded 2-spheres in $M$; 
    \item \textbf{$k$-cells} spanned by vertices $\a_0,
      \dots, \a_k$ if these spheres can be isotoped to be pairwise
      disjoint.  
\end{itemize}
\end{DEF}

We equip the automorphism group $\Aut(\sph(M_{\Gamma}))$ of 
$\sph(M_{\G})$ with the \emph{permutation topology}. This topology
is defined by the subbasis $\{U_a\}$ at the identity, where $a$
ranges over the vertices of $\sph(M_{\G})$ and $$U_a=\{\phi \in
\Aut(\sph(M_{\G})) \ | \ \phi(a)=a\}.$$
\par 
Note that $\text{Map}(M_{\G})$ acts naturally on $\sph(M_{\G})$, as
$\text{Diff}^+(M_{\G})$ acts on the collection of essential embedded
spheres of $M_{\G}$, preserving disjointness between pairs of
spheres. 
We have the following
from \cite{udall2024spherecomplexlocallyfinite}, which is
the first step in proving \Cref{thm:IvanovRigidity0}:

\begin{theorem}[{\cite[Theorem
  1.1]{udall2024spherecomplexlocallyfinite}}]\label{thm:GraphMCGAction}
  There is a surjective continuous map $\Psi:\mathrm{Map}(M_{\G})\to
  \mathrm{Map}(\G)$ with $\ker(\Psi)$ contained in the kernel $K$ 
  of the
  action of $\mathrm{Map}(M_{\G})$ on $\sph(M_{\G})$, hence inducing
  an action of $\Map(\Gamma)$ on $\sph(M_{\Gamma})$.  If $\Gamma$
  is not a graph with rank $r$ and $s$ rays such that $2r + s < 4$
  or $(r,s) \in \{(0,4),(2,0)\}$, then 
  $K =\ker(\Psi)$ and the induced action is faithful.
\end{theorem}

In the next two sections, we present tools used 
to characterize the links of simplices in
$\sph(M_\Gamma)$.  In particular, we give a method for constructing
intersecting spheres and describe certain full subcomplexes of
links. 

\subsubsection{Intersecting spheres}\label{sec:int_spheres}

Henceforth, let $M$ be an oriented $3$-manifold.  
A submanifold $N$ is \emph{essential} if it is not null-homotopic or peripheral, and we say two submanifolds \emph{intersect essentially} if they cannot be made disjoint up to
homotopy.  
Given transverse oriented
submanifolds $S, T \subset M$ of complementary dimension, the \emph{signed intersection number} of $S$ and $T$ is the sum 
\[\hat \iota(S,T) \coloneq \sum_{x \in S \cap T} \epsilon(x)\]
where
$\epsilon(x)= \pm 1$ according to the orientation induced by $S, T$
and $M$. The signed intersection number $\hat\iota$ is invariant up to homotopy in the following sense: 

\begin{lemma}\label{lem:int_invar}
  Let $S \subset M$ be an oriented submanifold and $T$ 
  an oriented
  manifold such that $\dim S + \dim T = \dim M$.  Let $\Psi : I
  \times T \to M$ be a homotopy transverse to $S$ such that 
  $\Psi^{-1}(S)$ is compact and $\Psi(I \times \partial T) \cap S 
  = \Psi(I \times T)\cap \partial S = \varnothing$. Let $\psi_t(x) =
  \Psi(t,x)$.  If 
  $\psi_0, \psi_1$ are embeddings transverse to $S$, then $\hat
  \iota(\psi_0(T),S) = \hat\iota(\psi_1(T),S)$.
\end{lemma}

\begin{proof}  
  Fix the usual orientation for $I$ and 
  endow $I \times T$ with the product orientation. 
  $Z = \Psi^{-1}(S)$ is
  then an oriented compact $1$-manifold with (oriented) 
  boundary $- \psi_0^{-1}(S) \sqcup
  \psi_1^{-1}(S)$. 
  In particular,
  $\hat\iota(\psi_1^{-1}(S),Z) - \hat\iota(\psi_0^{-1}(S),Z) =
  \hat\iota(\partial Z,Z) = 0$, from which the claim follows.
\end{proof}

By \cite{laudenbach_sur_1973}, if two
essential spheres in $M$ are homotopic, 
then they are isotopic, and this
isotopy extends to an ambient isotopy of $M\setminus \partial M$.  
In particular, if two spheres 
are disjoint up to homotopy, then there exists an isotopy of
\textit{one}
sphere realizing their disjointness. Likewise, if a sphere and an
arc are disjoint up to homotopy, then we may homotope the arc (rel
ends) to be disjoint.

\begin{corollary}\label{cor:int_arc}
  Let $\a \subset M\setminus \partial M$ be 
  an embedded non-peripheral sphere and $\gamma \subset M$ 
  a transverse simple arc.
  If $\hat\iota(\a,\gamma) \neq 0$, then $\a, \gamma$ are essential
  and intersect essentially. 
\end{corollary}

\begin{proof}
  By the above, if $\a,
  \gamma$ do not intersect essentially, then there exists 
  some arc $\gamma'$ homotopic
  to $\gamma$ (rel ends) and disjoint from $\a$. 
  In fact, $\gamma'$ is then properly homotopic to $\gamma$, 
  and up to a small
  homotopy leaving $\hat \iota$ unchanged we assume this
  homotopy is transverse to $\a$. As $\a$ is compact
  and disjoint from $\partial M$, we satisfy the
  hypotheses of \Cref{lem:int_invar}, and  $\hat\iota(\a,\gamma) =
  \hat\iota(\a,\gamma') = 0$. 
\end{proof}

\begin{lemma}\label{lem:int_pop}
  Let $\a \subset M\setminus \partial M$ 
  be a non-peripheral sphere and 
  let $\gamma$ be a simple arc between distinct punctures.
  Let $\b \cong S^2$
  be the boundary of a regular neighborhood of $\gamma$.  
  If $\gamma$ intersects $\a$ essentially, then $\b$ must also intersect $\a$ essentially.
\end{lemma}

\begin{proof}
  We prove the contrapositive.  Suppose that $\a,\b$ do not
  intersect essentially. Then $\a$ is isotopic to a sphere $\a'$
  disjoint from $\b$.  Let $M' \sqcup M'' = M \setminus \b$; $M'$ is
  disjoint from $\gamma$ and $M''$ is a thrice-punctured $3$-sphere.
  If $\a' \subset M'$ then $\a$ does not intersect $\gamma$
  essentially.  If $\a' \subset M''$, then $\a'$  is peripheral in
  $M''$; since $\a$ is not peripheral in $M$, $\a'$ is homotopic 
  to $\b$ which is disjoint from $\gamma$.
\end{proof}

\subsubsection{Full subcomplexes of
$\sph(M_\Gamma)$}\label{sec:full_subcpx}

A \emph{sphere system} is a (possibly infinite) collection of
distinct pairwise disjoint essential spheres.
We show that the sphere complex of
a submanifold obtained by cutting along a sphere system 
is a full subcomplex of $\sph(M_\Gamma)$:  

\begin{proposition}\label{prop:full_subcpx}
  Let $Z \subset M_\Gamma \setminus \sigma$ 
  be a complementary component of a
  sphere system $\sigma$.  Then $Z \hookrightarrow M_\Gamma$ 
  induces an injective full simplicial map 
  $\sph(Z) \hookrightarrow \sph(M_\Gamma)$.
\end{proposition}

The proof of \Cref{prop:full_subcpx} will make use of the following
lemma and two results of \cite{hatcher}. 

\begin{lemma}\label{lem:pi2_inj}
  $Z \hookrightarrow M_\Gamma$ is $\pi_2$-injective.  In particular,
  any essential sphere in $Z$ is essential in $M_\Gamma$.
\end{lemma}

\begin{proof}
  Let $p : \tilde M_\Gamma \to M_\Gamma$ be the universal covering of 
  $M_\Gamma$ and fix
  $\tilde Z \subset \tilde M_\Gamma$ a connected component of $p^{-1}(Z)$;
  observe $\tilde Z$ is a complementary component of the sphere
  system $\tilde \sigma = p^{-1}(\sigma)$.  
  We show the inclusion $\tilde Z
  \hookrightarrow \tilde M_\Gamma$ is $\pi_2$-injective and
  preserves essentiality, which suffices: two spheres are homotopic
  if and only if they have a pair of homotopic lifts, and
  neighborhoods of punctures lift homeomorphically. 
  
  An application of van Kampen's theorem shows that $Z$
  is $\pi_1$-injective, hence $\tilde Z$ is simply connected.  
  The pair $(\tilde M_\Gamma, \tilde
Z)$ induces the exact sequence
\[
H_3(\tilde M_\Gamma, \tilde Z) \to H_2(\tilde Z)
\xrightarrow{\iota_\natural} H_2(\tilde
M_\Gamma)\:;
\]
since each component of $\tilde M_\Gamma \setminus \tilde Z$ is
non-compact, $H_3(\tilde M_\Gamma,\tilde Z) \cong H_3(\tilde
M_\Gamma \setminus \tilde Z, \partial \tilde Z) = 0$.  Hence
$\iota_\natural$, which is induced by the inclusion $\iota : \tilde
Z \hookrightarrow \tilde M_\Gamma$, is injective.  Since
$\tilde Z, \tilde M_\Gamma$ are simply connected, the  
Hurewicz (natural) isomorphism implies that 
$\iota_*: \pi_2(\tilde Z) \to \pi_2(\tilde M_\Gamma)$ is 
likewise injective.   
It follows that two spheres $a,a'\subset
\tilde Z$ are  homotopic in $\tilde M_\Gamma$ only if they are
likewise in $\tilde Z$, and in particular $a \subset \tilde Z$ is
null-homotopic in $\tilde M_\Gamma$ only if it is likewise in 
$\tilde Z$.  Extending along paths to a basepoint 
$z_0 \in \tilde Z$, we
obtain homotopic based spheres $\bar a,\bar a' \subset \tilde
M_\Gamma$ also homotopic to $a,a'$ respectively.  Since $\tilde
M_\Gamma$ is simply connected, $\bar a,\bar a'$ can be chosen to be
homotopic to $a,a'$ via homotopies in $\tilde Z$, 
and $\pi_2$-injectivity 
implies that $\bar a,\bar a'$ are homotopic in $\tilde Z$ as well.

It remains to show that if $\a \subset \tilde Z$ is peripheral in 
$\tilde M_\Gamma$, then it is null-homotopic or peripheral
in $\tilde Z$.  Suppose
that $\a$ is peripheral to an (isolated) 
puncture $e \in \Ends(\tilde M_\Gamma)$. If $e \in
\Ends(\tilde Z)$, then $\a$ is peripheral in $\tilde Z$. Else, let
$M^+$ be obtained by replacing a neighborhood of $e$ disjoint from
$\tilde Z$ with a ball and
$\sigma^+ \subset M^+$ by removing components (in fact, at most one)
of $\tilde \sigma$ peripheral in $M^+$. Then 
$\tilde Z \hookrightarrow M^+$ is isotopic to a 
component of $M^+ \setminus \sigma^+$, hence $\pi_2$-injective 
by the
above: since $\a$ is null-homotopic in $M^+$, it is also 
null-homotopic in $\tilde Z$.
\end{proof}

We briefly review Hatcher normal form. 
Given a maximal sphere system $\sigma \subset M \cong M_{n,s}$ and a
transverse sphere $\a \subset M$, $\a$ is in \emph{normal form} with
$\sigma$ if either $\a$ is equal to a component of $\sigma$ or if (the closure of) each component of $\a \setminus \sigma$ 
meets each sphere in $\sigma$ at most once and no component
is homotopic rel boundary to a disk in $\sigma$. The intersection of $\a$ with the closure of a component of $M\setminus \sigma$ is called a \emph{piece} of $\a$.

\begin{proposition}[{\cite[Proposition
  1.1]{hatcher}}]\label{prop:hnf} Let $\a \subset M_{n,s}$ be an
  essential sphere and $\sigma \subset M_{n,s}$ a maximal sphere
  system.  Then $\a$ is homotopic to a sphere $\a'$ in normal
  position with $\sigma$, and if $\a \cap \sigma_0 = \varnothing$
  for some subsystem $\sigma_0 \subset \sigma$, then there exists a
  homotopy disjoint from $\sigma_0$.  
\end{proposition}

\begin{remark}
  Our statement is slightly stronger than that in \cite{hatcher},
  but immediate from the construction in its proof.
\end{remark}

\begin{proposition}[{\cite[Proposition 1.2]{hatcher}}]\label{prop:hnf_equiv}
  Let $\a,\a' \subset M_{n,s}$ be essential homotopic spheres in
  normal form with a maximal sphere system $\sigma \subset M_{n,s}$.
  Then $\a,\a'$ are homotopic via a homotopy which restricts to an
  isotopy on intersections with each sphere in $\sigma$. 
\end{proposition}

\begin{proof}[Proof of \Cref{prop:full_subcpx}]
 The map is well defined by \Cref{lem:pi2_inj}: 
  a sphere is essential in $Z$ only if 
  it is essential in $M_\Gamma$. 
  To show injectivity, we verify that two
  spheres in $Z$ 
  are isotopic in $M_\Gamma$ only if they likewise are in
  $Z$ (\textit{n.b.}\ $\pi_2$-injectivity is insufficient since
  $M_\Gamma$ is not simply connected). To show fullness, we prove
  that two spheres intersect essentially in $Z$ only if they
  likewise do in $M_\Gamma$; the converse is immediate, implying 
  that the map is simplicial.  

  First suppose that $\a,\b \subset M_\Gamma \setminus \sigma$ 
  are essential spheres homotopic in $M_\Gamma$, and
  fix some $N \cong M_{n,s} \subset M_\Gamma$ containing the image
  of the homotopy such that the components of $\partial N$  are
  essential spheres disjoint from $\sigma$.  
  Let $\sigma' \supset \{\b\} \cup \sigma$ be a
  maximal sphere system on $N$. By \Cref{prop:hnf}, $\a$
  is homotopic to some $\a'$ in normal form with $\sigma'$
  disjointly from $\sigma$; since $\b$ is likewise in normal form
  with $\sigma'$, disjoint from $\sigma$, and homotopic with $\a$
  (hence $\a'$), $\a',\b$ are homotopic disjoint from $\sigma$ by
  \Cref{prop:hnf_equiv}. Then, $\a,\b$ are likewise
  homotopic disjointly from $\sigma$.

  If $\a,\b \subset Z$ are essential spheres homotopic in
  $M_\Gamma$, then they are disjoint from $\sigma$ and thus
  homotopic in $Z$ by the above.
  If instead $\a,\b \subset Z$ are essential spheres that 
  do not intersect essentially
  in $M_\Gamma$, then fix an embedded sphere $\a'$ homotopic to $\a$
  in $M_\Gamma$ and simultaneously 
  disjoint from $\b, \sigma$.  Again by the
  above, $\a,\a'$ are homotopic in $Z$ 
  and hence $\a,\b$ do not intersect essentially in $Z$.
\end{proof}

\begin{remark*}\label{rmk:full_in_link} 
  The image of $\sph(Z)$ in
  \Cref{prop:full_subcpx} is exactly the subcomplex in
  $\sph(M_\Gamma)$ spanned by spheres in $\lk(\sigma) \subset
  \sph(M_\Gamma)$ contained in $Z$. 
  We show that a sphere $\a \in \lk(\sigma)$ contained in
  $Z$ is essential in $Z$, which suffices. If $\a\subset Z$ 
  is not essential
  in $Z$, then $\a$ bounds a ball in $Z$ thus
  likewise in $M_\Gamma$, or $\a$ is peripheral in $Z$ hence  
  either $\a \in \sigma$ or it is peripheral in $M_\Gamma$: in all
  cases, $\a \notin \lk(\sigma)$.
\end{remark*}

\begin{corollary}\label{cor:link_join}
  Let $Z_i\subset M_\Gamma \setminus \sigma$ denote the
  complementary components of a sphere system $\sigma$.  Then
  $\lk(\sigma)$ is isomorphic to the join 
  $\Asterisk_i\sph(Z_i)$.
\end{corollary}

\begin{proof}
  By \Cref{rmk:full_in_link}, it suffices to show that the
  $\sph(Z_i)$ have pairwise disjoint images in $\lk(\sigma)$.
  Suppose that $\a \in \lk(\sigma)$ 
  is in the image of $\sph(Z_i)$ and $\sph(Z_j)$, hence 
  we may fix homotopic representatives $\alpha \subset Z_i$ and 
  $\alpha' \subset Z_j$. 
  As in the proof of \Cref{prop:full_subcpx} we obtain a homotopy
  between $\alpha,\alpha'$ disjoint from $\sigma$, hence
  $\alpha,\alpha'$ lie in the same complementary component of
  $\sigma$ and $Z_i = Z_j$. 
\end{proof}
\subsubsection{Pants decompositions and $M_{0,s}$}\label{sec:rigid0}

We now shift our focus and introduce the tools we will need to prove
\Cref{thm:geom_rigidity}. 
The proof relies on a generalization
of the results of Bering--Leininger in \cite{bering2024finite},
which we will discuss in \Cref{sec:finite_rigidity}. 
We state here the relevant
preliminary definitions and results from \cite{bering2024finite}.

\begin{definition}\label{def:Pants} Any manifold homeomorphic to
  $M_{0,3}$ is called a \emph{pair of pants}. A maximal 
  sphere system $P\subset \sph(M_{n,s})^{(0)}$ is called a
  \emph{pants decomposition}. Fixing an open regular neighborhood
  $\nbd(P) \supset P$, each component of $M_{n,s}\setminus
  \nbd(P)$ is homeomorphic to a pair of pants.  
\end{definition}

\begin{definition}
    Suppose $P$ is a pants decomposition of $M_{n,s}$. Two spheres $\a,\b\in P$ are \emph{adjacent} in $P$ if they are two of the boundary spheres of some pair of pants component of $M_{n,s}\setminus P$. A sphere $\a\in P$ is \emph{self-adjacent} in $P$ if it bounds two cuffs of a single pair of pants in $M_{n,s}\setminus P$.
\end{definition}

\begin{definition}
Let $P_\a$ be a pants decomposition containing a sphere $\a$ that
is not self-adjacent.  Consider the connected component of
$M_{n,s}\setminus \nbd( {P_a\setminus \a})$ containing $a$. This is
homeomorphic to $M_{0,4}$, and $\sph(M_{0,4}) = \set{a, a', a''}$.
There are pants decompositions $P_{a'} = P_{a} \setminus \set{a}
\cup a'$ and similarly $P_{a''}$.  A change in pants decomposition
from $P_a \mapsto P_{a'}$ or $P_a \mapsto P_{a''}$ is called a
\emph{flip move.}  See \Cref{fig:Xdet}.
\end{definition} 

  \begin{definition}[{\cite[Definition
    2.2]{bering2024finite}}]\label{def:XDetectable} Let $X\subset
    \sph(M_{n,s})$ be a subcomplex. Two spheres $\a, \a' \in X^{(0)}$
    which intersect essentially have $X$-\emph{detectable
    intersection} if there are pants decompositions $P_\a, P_{\a}'
    \subset X^{(0)}$ such that $P_\a,P_{\a'}$ differ by a \emph{flip
    move} $\a \to \a'$, \ie\ 
    for which $\a\in P_\a$, $\a'\in P_{\a'}$ and
    $P_\a\setminus \{\a\}= P_{\a'} \setminus \{\a'\}$.
    
\end{definition}

Since $P_\a, P_{\a'}$ differ by a flip move 
$\a \to \a'$, 
the component of $M_{n,s} \setminus \nbd(P_\a
\setminus \{\a\})$ containing $\a \cup \a'$ deformation retracts to
$\a \cup \a'$ and is homeomorphic to $M_{0,4}$ (see \Cref{fig:Xdet}). 
If $\a,\a'$ has an intersection $X$-detectable, then
they intersect essentially: we apply
\Cref{prop:full_subcpx} to $M_{0,4}\setminus \partial
M_{0,4} \hookrightarrow M_{n,s} \setminus \partial M_{n,s}$.  
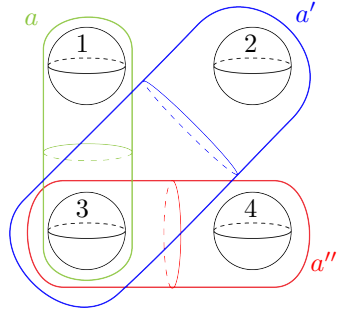
\begin{figure}[ht!]
    \centering
    \begin{overpic}[width=4cm]{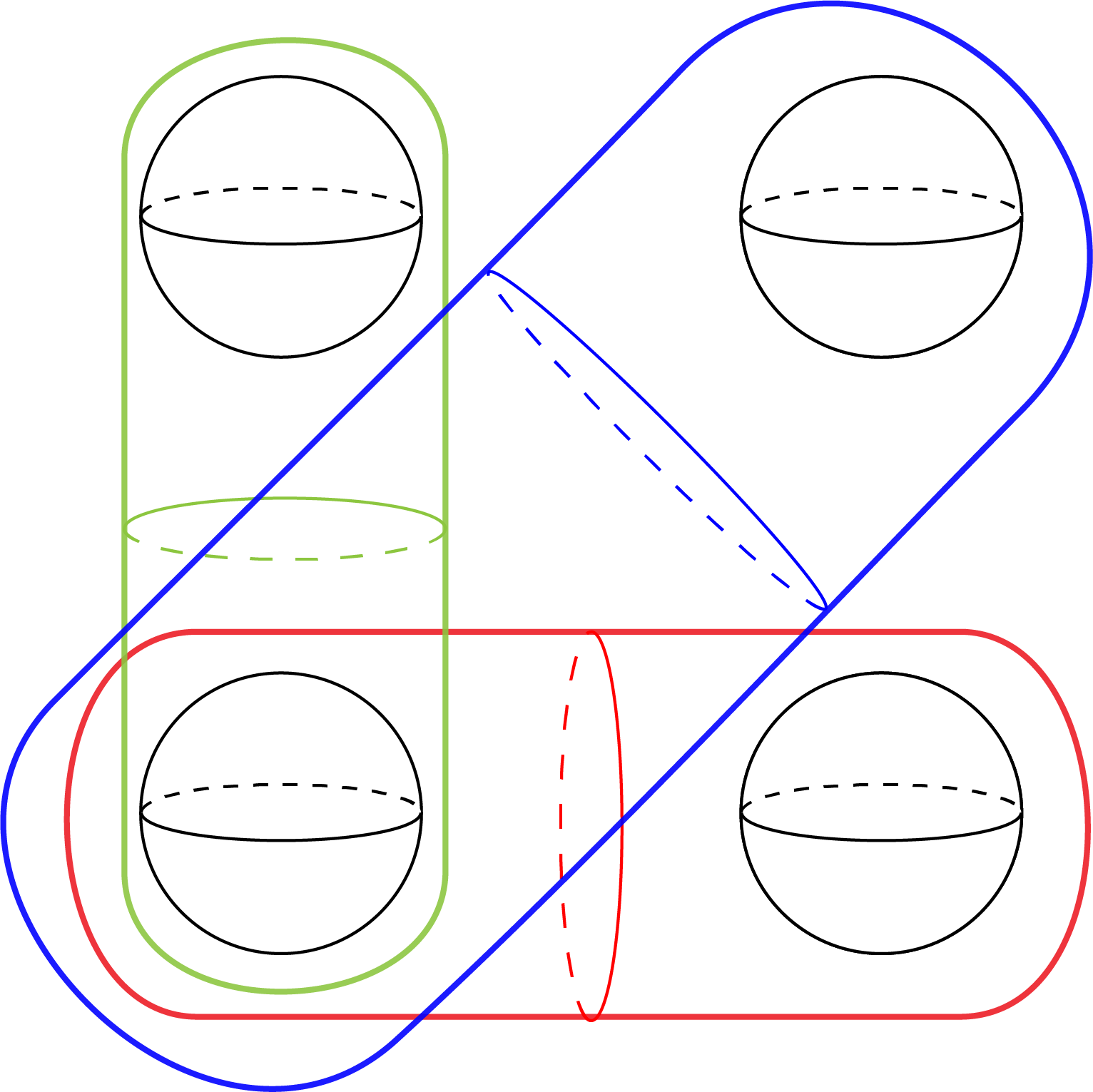}
        \put(100,12){\color{red}{$\a''$}}
        \put(95,95){\color{blue}{$\a'$}}
        \put(5,94){\color{LimeGreen}{$\a$}}
        \put(22,85){1}
        \put(78,85){2}
        \put(22,30){3}
        \put(78,30){4}
    \end{overpic}
    \caption{The three spheres $a, a',$ and $a''$ in $M_{0,4}$ differ by a flip move.  $\partial
    M_{0,4}$ is labeled $1$--$4$.}
    
\label{fig:Xdet}
\end{figure}

We will consider locally injective simplicial maps $f:X \to
\sph(M_{n,s})$, where $X$ is a subcomplex of $\sph(M_{n,s})$. The
map $f$ is \emph{simplicial} if it sends vertices to vertices and
edges to edges, and \emph{locally injective} if it is injective
when restricted to the stars of vertices in $X$.

\begin{lemma}[{\cite[Lemma 8]{bering2024finite}}]\label{lem:XdetInt}
      \label{lem:locallyInjXDetectable} 
      Let
      $X\subset \sph(M_{n,s})$ be a subcomplex, and suppose $f:X \to
      \sph(M_{n,s})$ is a locally injective simplicial map. If $\a,\a'
      \in X^{(0)}$ have $X$-detectable intersection, then $f(\a)$ and
      $f(\a')$ have $f(X)$-detectable intersection.
\end{lemma}

The following two results discuss the sphere graph of $M_{0,s}$. 

\begin{lemma}[{\cite[Lemma 9]{bering2024finite}}]
  \label{lem:GenusZeroSphere}
    A sphere $\a\in \sph(M_{0,s})^{(0)}$ is determined by the
  partition of $\partial M_{0,s}$ induced by the connected
components of $M_{0,s}\setminus \a$.  
\end{lemma}

\begin{lemma}[{\cite[Corollary 11]{bering2024finite}}]
  \label{lem:GenusZeroRigidity}
    If $N\cong M_{0,s} \cong N'$ and $f:\sph(N) \to \sph(N')$ is a
  simplicial automorphism, then there is a homeomorphism $h: N \to
N'$ so that $h(\a)=f(\a)$ for every $\a\in \sph(N)$.  
\end{lemma}

The following lemma provides a useful topological criterion to distinguish spheres combinatorially.

\begin{lemma}[{\cite[Lemma 12]{bering2024finite}}]
  \label{lem:ThreeSphereConfiguration}
    Suppose $\a, \b, \c \in \sph(M_{n,s})$ are such that $\b$ and $\c$
    are disjoint and $\a$ essentially intersects $\b$ and $\c$, each
    in one circle, such that the boundary components of 
    $\overline{\nbd}(\a \cup \b \cup \c)$ are either essential
    or peripheral.
    Let $\b'\in \sph(\nbd(\a\cup \b))$ be the
    unique sphere in $\nbd(\a\cup \b)$ distinct from $\a$ and $\b$ and
    $\c'\in \sph(\nbd(\a\cup \c))$ the unique sphere in $\nbd(\a\cup \b)$
    distinct from $\a$ and $\c$. Then $\b'$ and $\c'$
    intersect, and both intersect $\b$ and $\c$.  
\end{lemma}

We include a useful technical result, which is equivalent to the
connectedness of the pants graph for $M_{n,s}$.  In particular, this
proposition replaces the use of Outer space in the 
proof of Proposition 22 in \cite{bering2024finite}. 

\begin{proposition}\label{prop:FlipMoveConnectivity}
  Any two pants decompositions in $M_{n,s}$ differ 
  by a finite sequence of flip moves.
\end{proposition}

\begin{proof}
  Fix $M_{n,s}$ with minimal dimension $k = \dim
  \sph(M_{n,s})$ such that the proposition
  does not hold: $k > 0$, else $\sph(M_{n,s})$ is empty or $M_{n,s}
  $ is $M_{1,1}$ or $M_{0,4}$ and the claim is immediate.
  Let $\mathcal{S}'$ denote the barycentric
  subdivision of $\sph(M_{n,s})$ and let $\mathcal{S}''$ denote the
  barycentric subdivision of $\sph(M_{n,s})^{(k-2)}$.  It is
  equivalent to show that the subcomblex 
  $\mathcal{P} = \mathcal{S}'
  \setminus \mathcal{S}''$ is connected: in particular, any flip
  move corresponds to moving between (the barycenters of) two
  $k$-simplices in $\sph(M_{n,s})$ 
  via a $(k-1)$-simplex. We claim that 
  if two $k$-simplices in $\sph(M_{n,s})$ share a face 
  then their barycenters are in the same component of
  $\mathcal{P}$, which suffices.
  In particular any simplex in $\sph(M_{n,s})$ is the face of a
  $k$-simplex, 
  hence any path in $\mathcal{S}'$ gives a
  sequence of $k$-simplices $P_i$ with $P_i \cap P_{i+1} \neq
  \varnothing$. Since $k > 0$, $\sph(M_{n,s})$ and $\mathcal{S}'$
  are connected, hence by the claim $\mathcal{P}$ is connected.

  Suppose that $P,P'$ are $k$-simplices with non-empty
  intersection $\sigma = P \cap P'$.  
  Let $N_i$ denote the non-pants 
  components of $M_{n,s}\setminus \sigma$.
  Since $\sigma$ is non-empty and 
  each $\sph(N_i)$ is a full subcomplex
  of $\sph(M_{n,s})$, $\dim(\sph(N_i)) < k$.
  By minimality we may choose a sequence
  of flip moves in $N_i$ between $P \cap N_i$ and $P' \cap N_i$ for
  each $i$; concatenating these sequences gives a path
  between $P,P'$ in $\mathcal{P}$.
\end{proof}

\subsection{Edge isomorphisms and rigidity}\label{sec:edge_isoms}
In this subsection, we state a version of Whitney's
graph isomorphism theorem  
which will be useful in \Cref{sec:geom_rigid}.  Let
$\Delta, \Delta'$ denote graphs, possibly with loops and multiple edges between vertices.

\begin{definition}\label{def:edge_isom}
  
  A bijection 
  $\psi : E(\Delta) \to E(\Delta')$ is an \emph{edge isomorphism} if
  for all $e,e' \in E(\Delta)$, 
   there is an isomorphism  of subgraphs
  $e \cup e' \to \psi(e) \cup \psi(e')$ inducing $\psi|_{\{e,e'\}}$.
\end{definition}

\begin{definition}\label{def:edge_pair}
  Given a map $\psi : E(\Delta) \to E(\Delta')$ and graphs 
  $G,G'$, $\psi$ has a \emph{$G,G'$-pair} if, for some $\eta
  \subset E(\Delta)$, the subgraphs 
  $\bigcup_{e \in \eta} e, \bigcup_{e \in \eta} \psi(e)$ are
  isomorphic to $G, G'$, in any order.
\end{definition}

Let $K_3$ denote the $3$-clique and $K_{1,3}$ the $3$-star.  The following theorem is immediate from 
{\cite[\textit{e.g.}\ Corollary 2.2]{gardner}}.

\begin{theorem}[Gardner]\label{thm:whitney_multi}
  Suppose $\Delta,\Delta'$ are finite.  
  An edge isomorphism $\psi : E(\Delta) \to E(\Delta')$ is induced
  by an isomorphism $\Delta \to \Delta'$ if and only if $\psi$ does
  not have a $K_3,K_{1,3}$-pair.
\end{theorem}

\begin{remark}
  If $\Delta$ is connected and 
  $|\Delta| \neq 2$, then a unique graph isomorphism induces $\psi$.
  In particular, if $\psi'$ were another such then $\psi^{-1}\psi'$
  is the identity on edges, which implies identity unless $\Delta$
  consists of two vertices with edges only between them.
\end{remark}

\begin{corollary}\label{cor:whitney_infinite}
  \Cref{thm:whitney_multi} likewise holds if $\Delta,\Delta'$
  are infinite and $\Delta$ is connected with $|\Delta| \neq 2$.
\end{corollary}

\begin{proof}
  Fix a compact exhaustion  of $\Delta$ by connected subgraphs
  $\Delta_i$ of order $n_i \neq 2$.  
  Let $\Delta_i'$ denote the subgraph $\bigcup_{e \in
  E(\Delta_i)} \psi(e)$; note that $\Delta_i'$ is likewise a compact
  exhaustion of $\Delta'$, since every $e' \in E(\Delta')$ is
  the image of an edge $e \in E(\Delta_i)$ for some $i$.
  Then $\psi$ restricts to edge isomorphisms $\psi_i :
  E(\Delta_i) \to
  E(\Delta_i')$; if any has a $K_3,K_{1,3}$-pair then so
  does $\psi$, hence it follows that  $\psi_i$ is induced by a
  unique isomorphism $\tilde \psi_i : \Delta_i \to \Delta_i'$.
  By uniqueness, these maps form a direct system, the direct 
  limit of which is an isomorphism 
  $\tilde\psi:\Delta \to \Delta'$ inducing $\psi$.  
\end{proof}

As in the remark, we observe that the $\tilde\psi$ above 
is likewise unique.

\section{A proof of the main theorem}\label{sec:geom_rigid} 

We prove \Cref{thm:IvanovRigidity0}.  We first show 
that a sphere graph isomorphism induces an isomorphism between 
the dual graphs of a pants decomposition and its image, which we 
will use to define the desired diffeomorphism of $M_{\Gamma}$. 

\subsection{Dual graphs of pants decompositions}
\label{sec:pants_dual}

Henceforth, let $f : \sph(M_\Gamma) \xrightarrow{\sim}
\sph(M_{\Gamma'})$ be an isomorphism of sphere graphs.
Given a maximal sphere system $\sigma \subset \sph(M_\Gamma)$,
$f\sigma$ is likewise maximal and $\sigma,f\sigma$ 
define pants decompositions
of $M_\Gamma, M_{\Gamma'}$ respectively.  Let $\Delta_\sigma$,
$\Delta_{f\sigma}$ denote the dual graphs to these pants
decompositions. 
For a sphere $\a\in \sigma$, we denote by $e_\a$ the
dual edge in $\Delta_\sigma$. 
Thus $E(\Delta_\sigma) = \{e_t : t \in \sigma\}$ and
the restriction $f|_\sigma$
defines a bijection $f|_\sigma : E(\Delta_\sigma) \to
E(\Delta_{f\sigma})$. 

\break 

\begin{lemma}\label{lem:dual_edge_isom}
  $f|_\sigma$ is an edge isomorphism.
\end{lemma}

\begin{proof}
For any $e_\a, e_\b \in E(\Delta_\sigma)$, 
we must show there is an isomorphism on subgraphs 
$e_\a \cup e_\b \to e_{f\a} \cup e_{f\b}$ 
for which $e_\a \mapsto e_{f\a}$ and $e_{\b} \mapsto e_{f\b}$. 
%Let $a,b\in \sigma$. 
By examining complementary components and applying
\Cref{cor:link_join} 
we obtain the following: 

\begin{enumerate}[label=(\roman*)]
  \item $e_\a$ is a loop in $\Delta_\sigma$ if and only
    if $\lk(\sigma \setminus \a) = \{\a\} \cong
    \sph(M_{1,1})$.  Otherwise,
    $\lk(\sigma \setminus \a) = \{\a,\a',\a''\}
    \cong \sph(M_{0,4})$ for
    some distinct $\a',\a''$.

  \item $e_\a, e_\b$ have common incident
    vertices (\textit{i.e.}\ $a,b$ are adjacent) in 
    $\Delta_\sigma$ if
    and only if $\lk(\sigma \setminus \{\a,\b\}) \cong
    \sph(M_{1,2})$ or $\sph(M_{0,5})$. 

  \item If $e_\a, e_\b$ are adjacent in $\Delta_\sigma$
    and $e_\a$ or $e_\b$ is a loop, 
    then they are incident on exactly one
    common vertex.  If $e_\a,e_\b$ are not loops and 
    incident on exactly one
    common vertex, then $\lk(\sigma \setminus
    \{\a,\b\}) = \sph(M_{0,5})$.  If they are incident on two
    common vertices (\textit{i.e.}\ they form a bigon), then $\lk(\sigma \setminus \{\a,\b\}) =
    \sph(M_{1,2})$.
\end{enumerate}

These properties are sufficient to determine $e_a \cup e_b$ up to
order preserving isomorphism.  Moreover, since (i)-(iii) are specified
combinatorially (because links are preserved by isomorphisms), they are preserved by $f$: if any hold for $e_a,
e_b$, then likewise do they for $e_{fa}, e_{fb}$, which suffices.
\end{proof}

\begin{proposition}\label{prop:edge_to_dual}
  $f|_\sigma$ is induced by an isomorphism 
  $\Delta_\sigma \to
  \Delta_{f\sigma}$.  
\end{proposition}

\begin{proof}
  By \Cref{thm:whitney_multi} and \Cref{cor:whitney_infinite}, 
  it suffices that $f|_\sigma$ does not have a $K_3,K_{1,3}$-pair.
  Suppose that $f|_\sigma$ has such a pair on 
  $e_\a,e_\b,e_\c \in E(\Delta_\sigma)$.  If
  $e_\a \cup e_\b \cup e_\c \cong K_3$ then $e_{f\a} \cup
  e_{f\b} \cup e_{f\c} \cong K_{1,3}$. It follows that $\lk(\sigma
  \setminus \{\a,\b,\c\}) \cong \sph(M_{1,3})$ and
  $\lk(f\sigma \setminus \{f\a,f\b,f\c\}) \cong 
  \sph(M_{0,6}) \not\cong \sph(M_{1,3})$, a contradiction. See \Cref{fig:K13K3}.  
  An analogous argument
  applies exchanging $K_3$ and $K_{1,3}$.
  \end{proof}

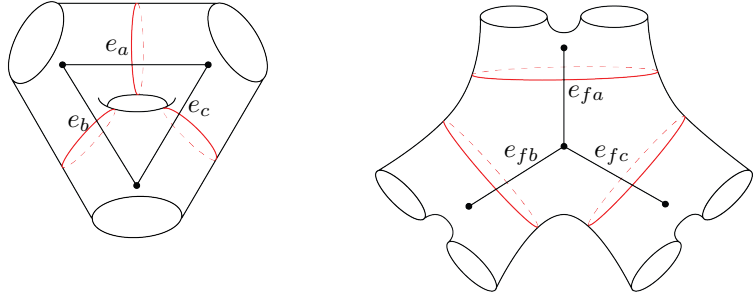
\begin{figure}[ht!]
      \centering
      \begin{overpic}[width=10cm]{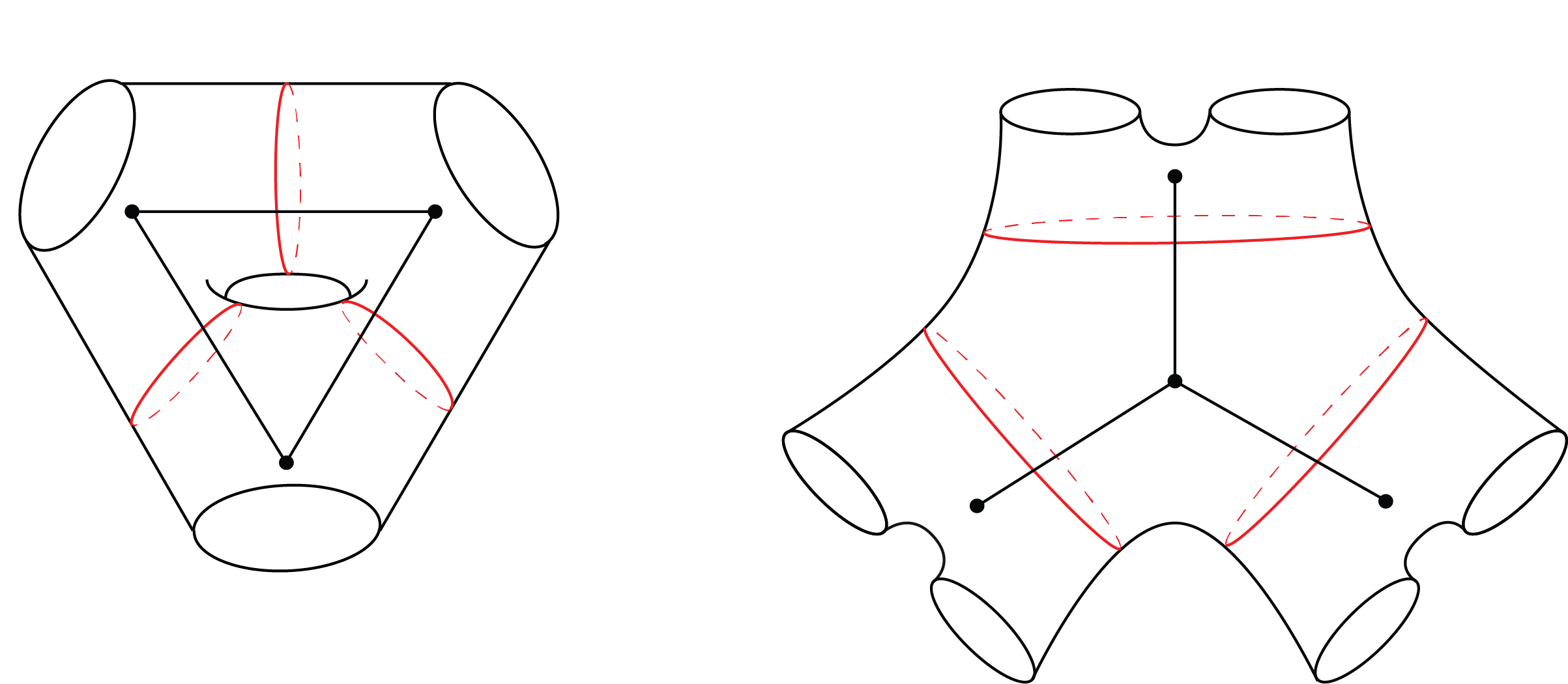}
        \put(14,32){$e_\a$}
        \put(9,22){$e_\b$}
        \put(25,24){$e_\c$}
        \put(75.5,26){$e_{f\a}$}
        \put(67,18){$e_{f\b}$}
        \put(79,18){$e_{f\c}$}
      \end{overpic}
      \caption{The $M_{1,3}$ and $M_{0,6}$ components in $M_\Gamma
      \setminus (\sigma \setminus \{\a,\b,\c\})$ and
      $M_{\Gamma'}
      \setminus (f\sigma \setminus \{f\a,f\b,f\c\})$ respectively.
    The figure shows one-half of the doubled handlebodies.}
    \label{fig:K13K3}
  \end{figure}

\subsection{Constructing a 
diffeomorphism}\label{sec:diffeo}

We are interested in diffeomorphisms $M_\Gamma \to M_{\Gamma'}$
that agree with $f$ locally about $\sigma$.

\begin{definition}\label{def:compatible}
  For a maximal sphere system $\sigma \subset \sph(M_\Gamma)$, let
  \[X_\sigma \coloneq \bigcup_{\a \in \sigma}
  \lk(\sigma \setminus \a)\:.\]
\end{definition}

We note that $\sigma \subset X_\sigma$.
Let $f_\sigma : \Delta_\sigma \to \Delta_{f\sigma}$ denote an
isomorphism inducing $f|_\sigma$. 
Realizing $M_\Gamma,M_{\Gamma'}$ as the gluing
of pairs of pants along spheres in $\sigma,f\sigma$ respectively,
$f_\sigma$ is induced by  
a diffeomorphism $h_0 : M_\Gamma \to M_{\Gamma'}$ defined
up to isotopy and precomposition by 
diffeomorphisms fixing pointwise
$\Delta_\sigma$.

\begin{definition}\label{def:almost_peripheral}
  A separating 
  sphere $\a \subset M_\Gamma$ is \emph{almost peripheral} if it
  bounds a pair of pants or a $M_{1,1}$.
\end{definition}

For an almost peripheral sphere $\a
\in \sigma$, 
$\lk(\sigma
\setminus \a) = \{\a,\a',\a''\}$ and the half-twist fixes $\a$ and
exchanges $\a',\a''$: obtain $h$ from $h_0$ by precomposing
by half-twists on almost peripheral $\a \in \sigma$ such that
$h$ agrees with $f$ on $\lk(\sigma \setminus \a)$.  

\begin{lemma}\label{lem:compatibility_sigma}
  $h$ agrees with $f$ on $X_\sigma$.
\end{lemma}

\begin{proof}
  By construction $h_*$ agrees with $f$ on $\sigma$ and $\lk(\sigma
  \setminus \a)$ for $\a \in \sigma$ almost peripheral.  If
  $\a$ is a loop in $\Delta_\sigma$, then $\lk(\sigma \setminus
  \a) = \{\a\}\subset \sigma$. Thus it 
  remains to show that $h_*$ agrees with $f$ on $\lk(\sigma
  \setminus \a)$ when $\a$ is not a loop or 
  almost peripheral.

  Let $\lk(\sigma \setminus \a) = \{\a,\a',\a''\}$
  and $\sigma' = (\sigma \setminus \a) \cup \a'$ and
  $\sigma'' = (\sigma \setminus \a) \cup \a''$.
  Let $M \cong M_{0,4}\setminus \partial M_{0,4}$ denote the
  complementary 
  component of $\sigma \setminus \a$
  containing $\a$, hence $M' = h(M)$ is the complementary
  component of $f\sigma \setminus f\a$ containing $f\a$.  
  Since $\a$ is not almost peripheral,
  there exist $\b,\b' \in \sigma$ which are separated by
  $\a$ in $M$ (possibly $\b = \b'$). 
  Without loss of generality, $\a'$ 
  separates $\b,\b'$ in $M$ and $\a''$ does not, hence
  likewise $h\a$ 
  and $h\a'$ separate $h\b,h\b'$ in $M'$ and
  $h\a''$ does not. 
  It follows that $\lk(\sigma \setminus
  \{\b,\b'\}) \cong \lk(\sigma' \setminus \{\b,\b'\})
  \not\cong \lk(\sigma'' \setminus \{\b,\b'\})$ and 
  $\lk(h\sigma \setminus
  \{h\b,h\b'\}) \cong \lk(h\sigma' \setminus
  \{h\b,h\b'\})
  \not\cong \lk(h\sigma'' \setminus \{h\b,h\b'\})$. See \Cref{fig:FourLinks}. 
  
  We have that $h\a = f\a$, $h\b = f\b$, and
  $h\b'=f\b'$.  Suppose that $h\a' = f\a''$ and
  $h\a'' = f\a'$.  Then $h\sigma'' = f\sigma'$ (and
  \textit{vice versa}), hence $\lk(h\sigma'' \setminus
  \{h\b,h\b'\}) = \lk(f\sigma' \setminus
  \{f\b,f\b'\}) \cong \lk(\sigma' \setminus \{\b,\b'\})
  \cong \lk(\sigma \setminus \{\b,\b'\}) \cong \lk(h\sigma
  \setminus \{h\b,h\b'\})$, a contradiction with the above. Hence 
  $h\a' = f\a'$ and $h\a'' = f\a''$ as required.
\end{proof}

\begin{figure}[h!]
    \centering
    \begin{tabular}{>{\centering\arraybackslash}m{0.473\textwidth} >{\centering\arraybackslash}m{0.473\textwidth}}
        \hline \\[0.2cm]
       \center
       \begin{overpic}[scale=0.3]{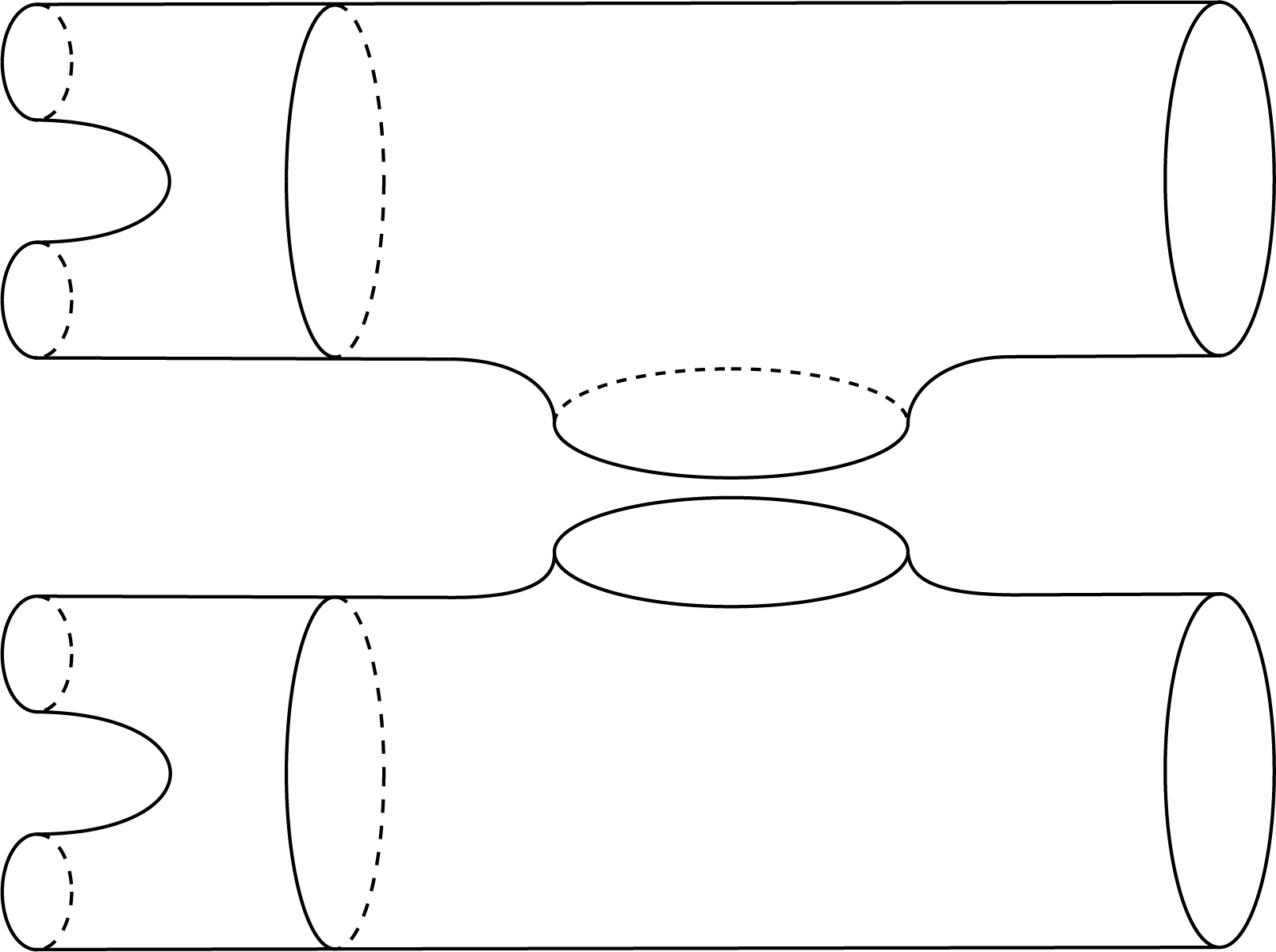}
           \put(32,57){$b$}
           \put(32,10){$b'$}
           \put(70,35){$a$}
       \end{overpic}
       & 
        \begin{overpic}[scale=0.3]{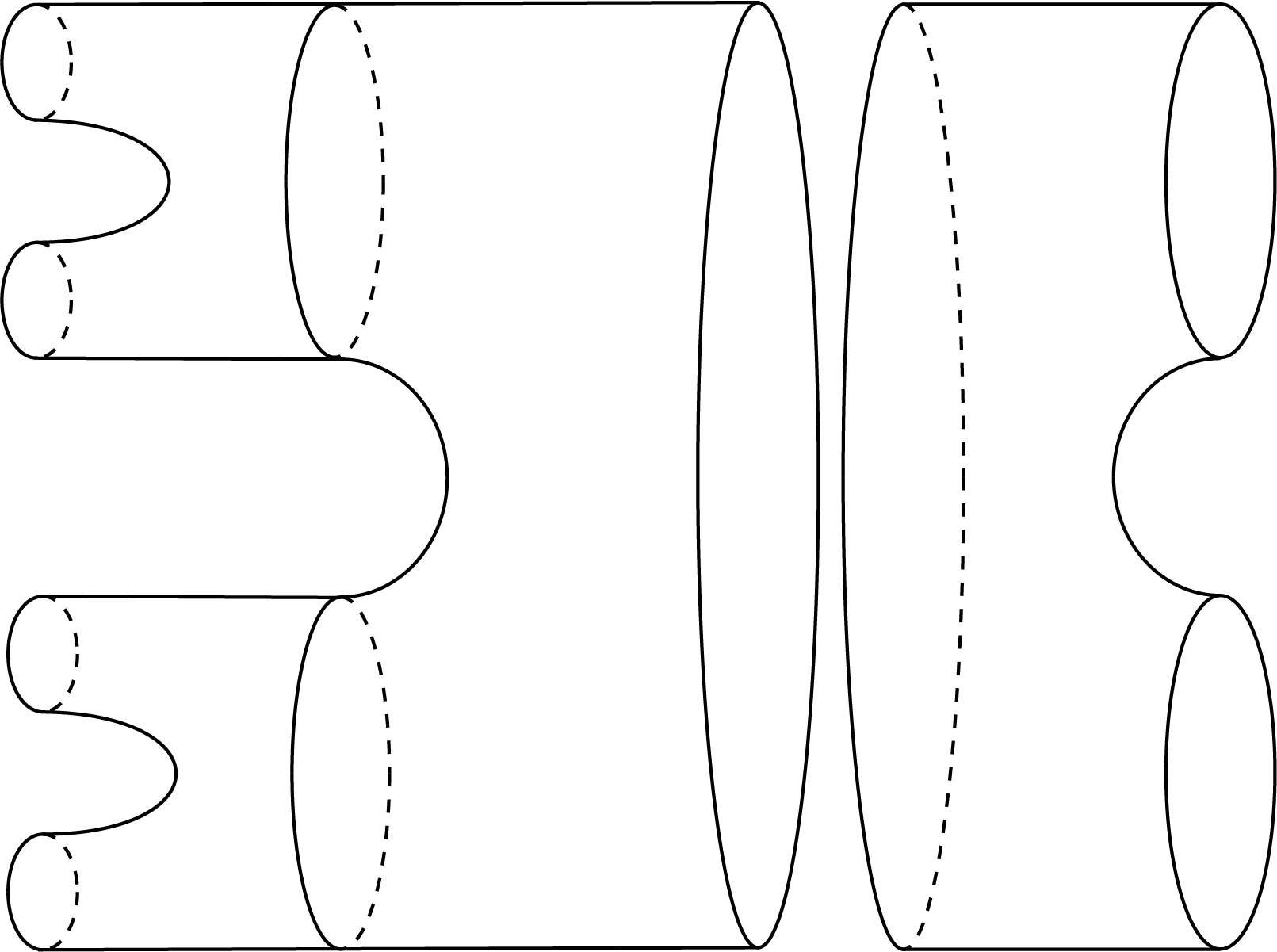}
           \put(32,57){$b$}
           \put(32,10){$b'$}
           \put(62.5,73){$a''$}
       \end{overpic}\\[1cm] 
       $\color{Maroon}\lk(\sigma\setminus\{b,b'\})\cong S(M_{0,4})\ast S(M_{0,4})$ &  $\color{Maroon}\lk(\sigma''\setminus\{b,b'\})\cong S(M_{0,5})$ \\[0.4cm]
       \center  \begin{overpic}[scale=0.3]{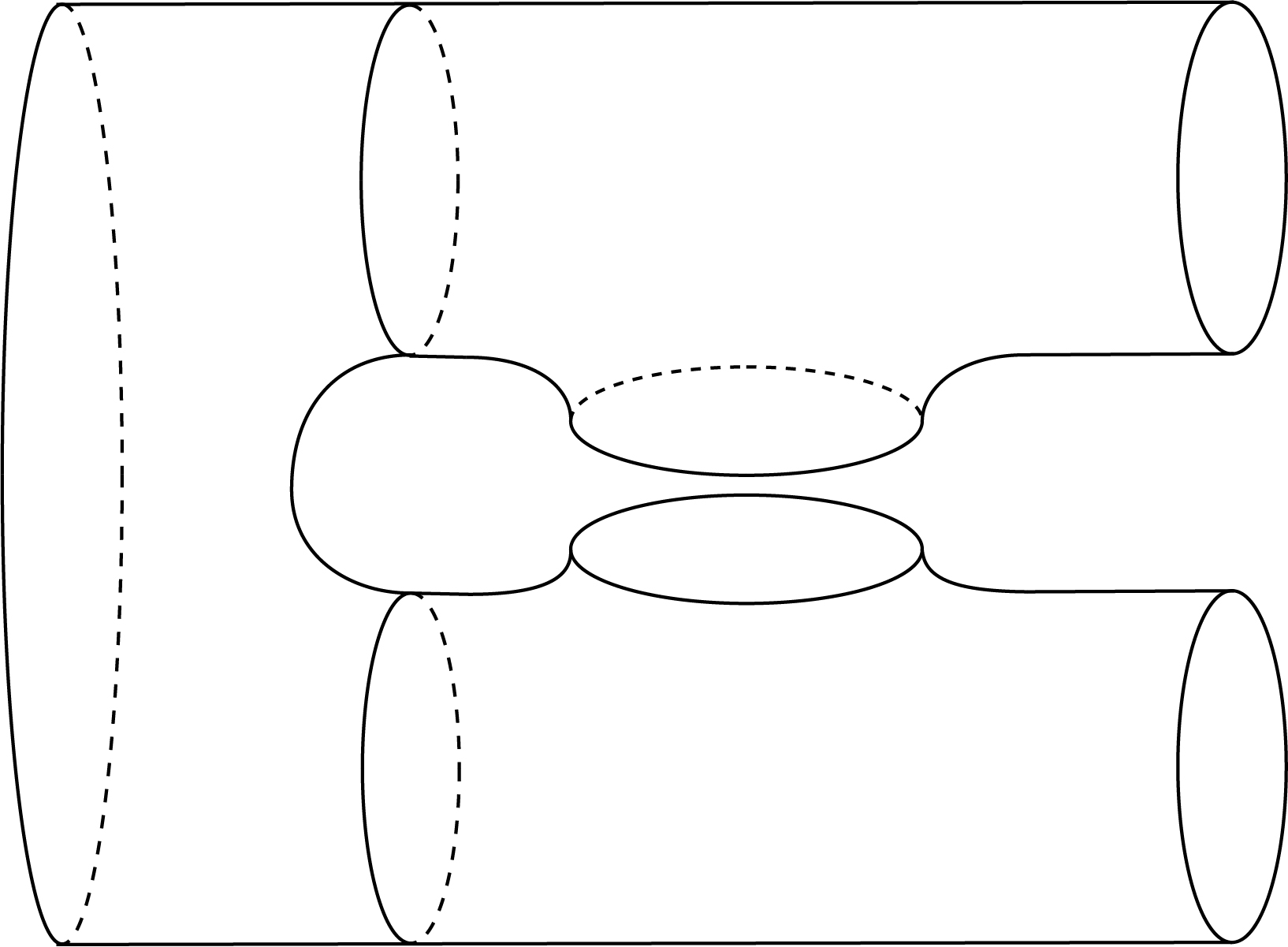}
           \put(37,57){$b$}
           \put(37,10){$b'$}
           \put(70,34){$a$}
       \end{overpic}&  \begin{overpic}[scale=0.3]{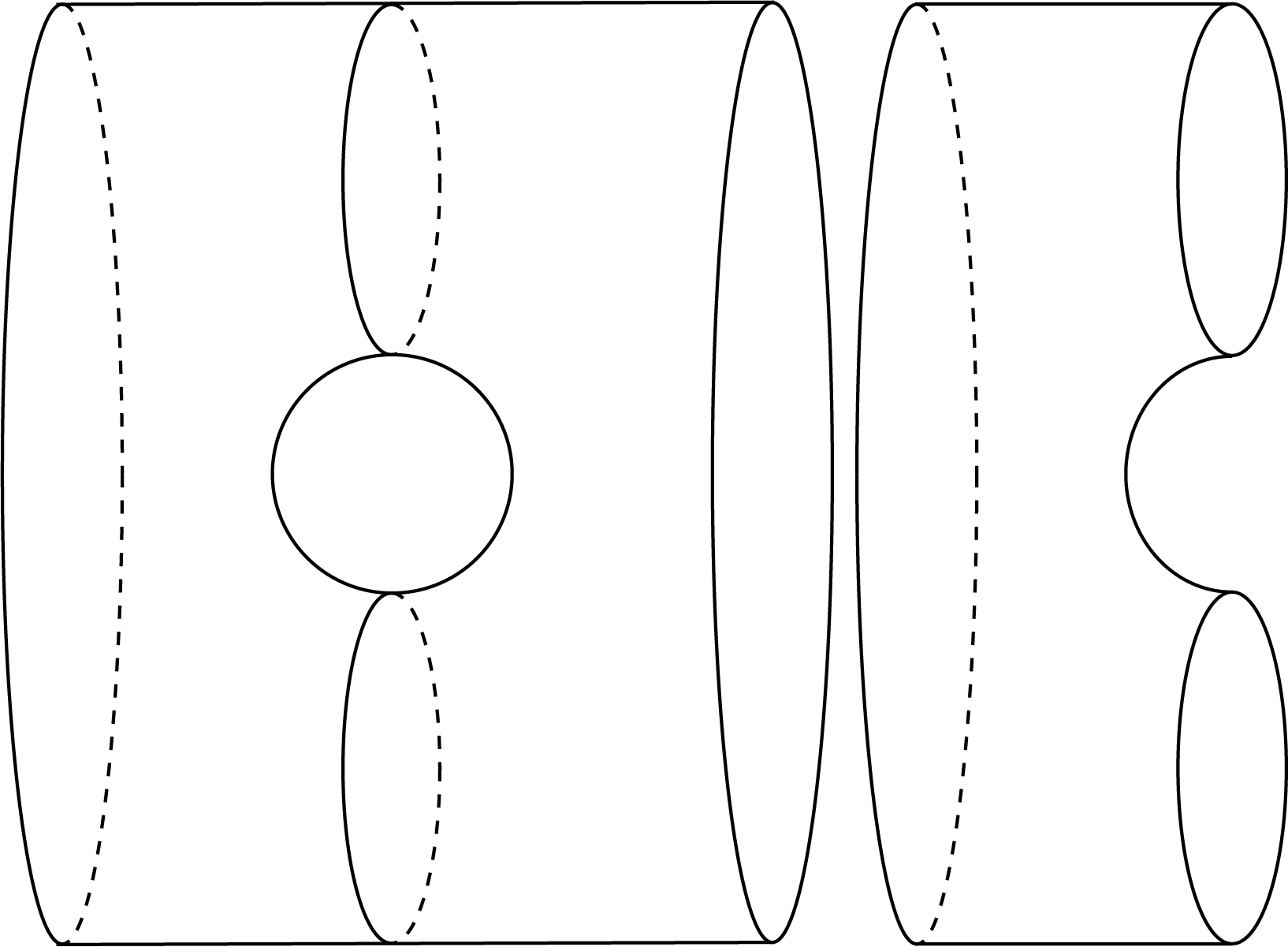}
           \put(35,57){$b$}
           \put(35,10){$b'$}
           \put(62.5,73){$a''$}
       \end{overpic} \\[1.1cm] 
       \center$\color{Maroon}\lk(\sigma\setminus\{b,b'\})\cong S(M_{0,5})$ &  $\color{Maroon}\lk(\sigma''\setminus\{b,b'\})\cong S(M_{1,2})$ \\[0.4cm]
        \center  \begin{overpic}[scale=0.3]{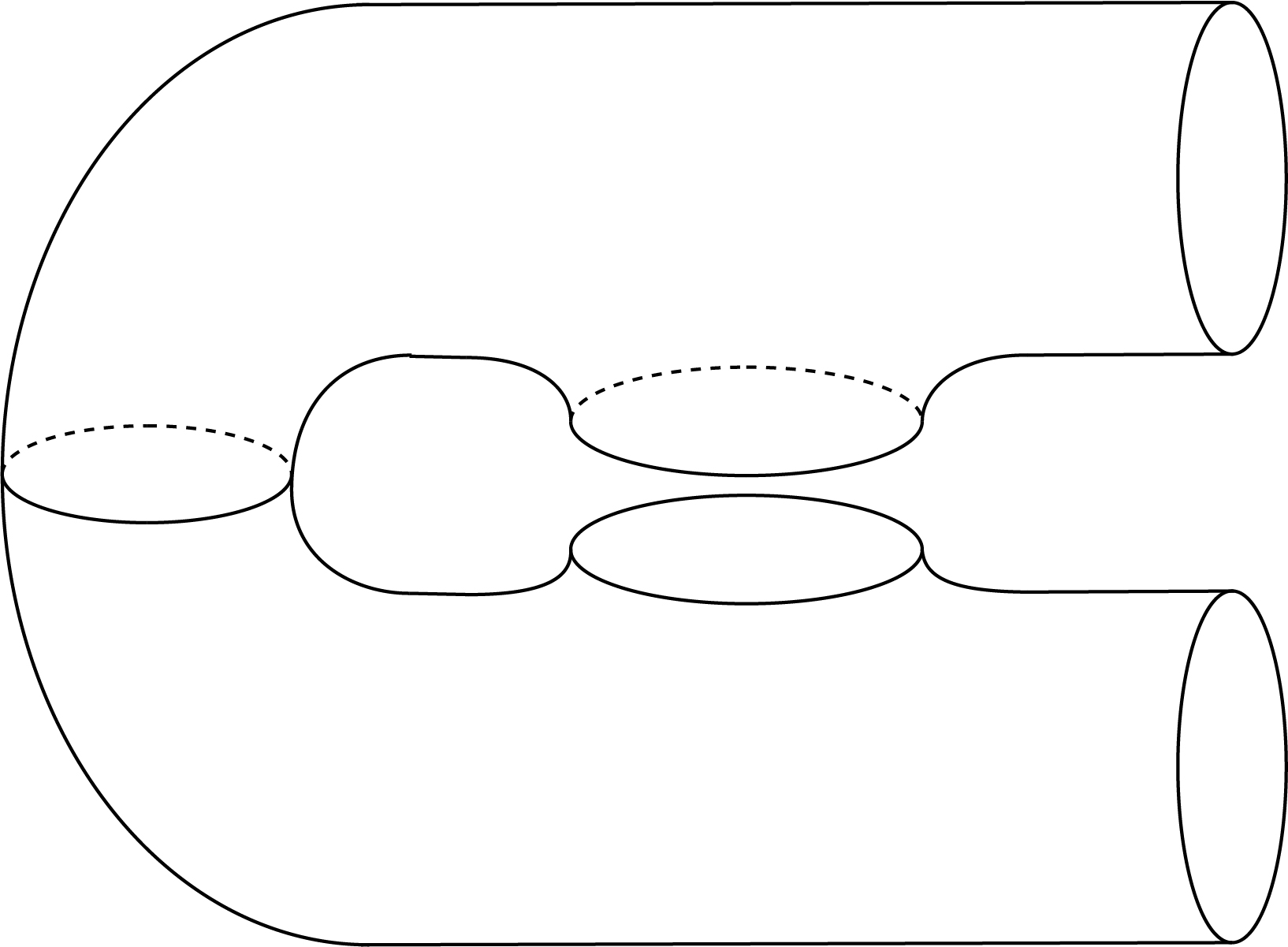}
           \put(4,42.5){$b = b'$}
           \put(70,34){$a$}
       \end{overpic}&  \begin{overpic}[scale=0.3]{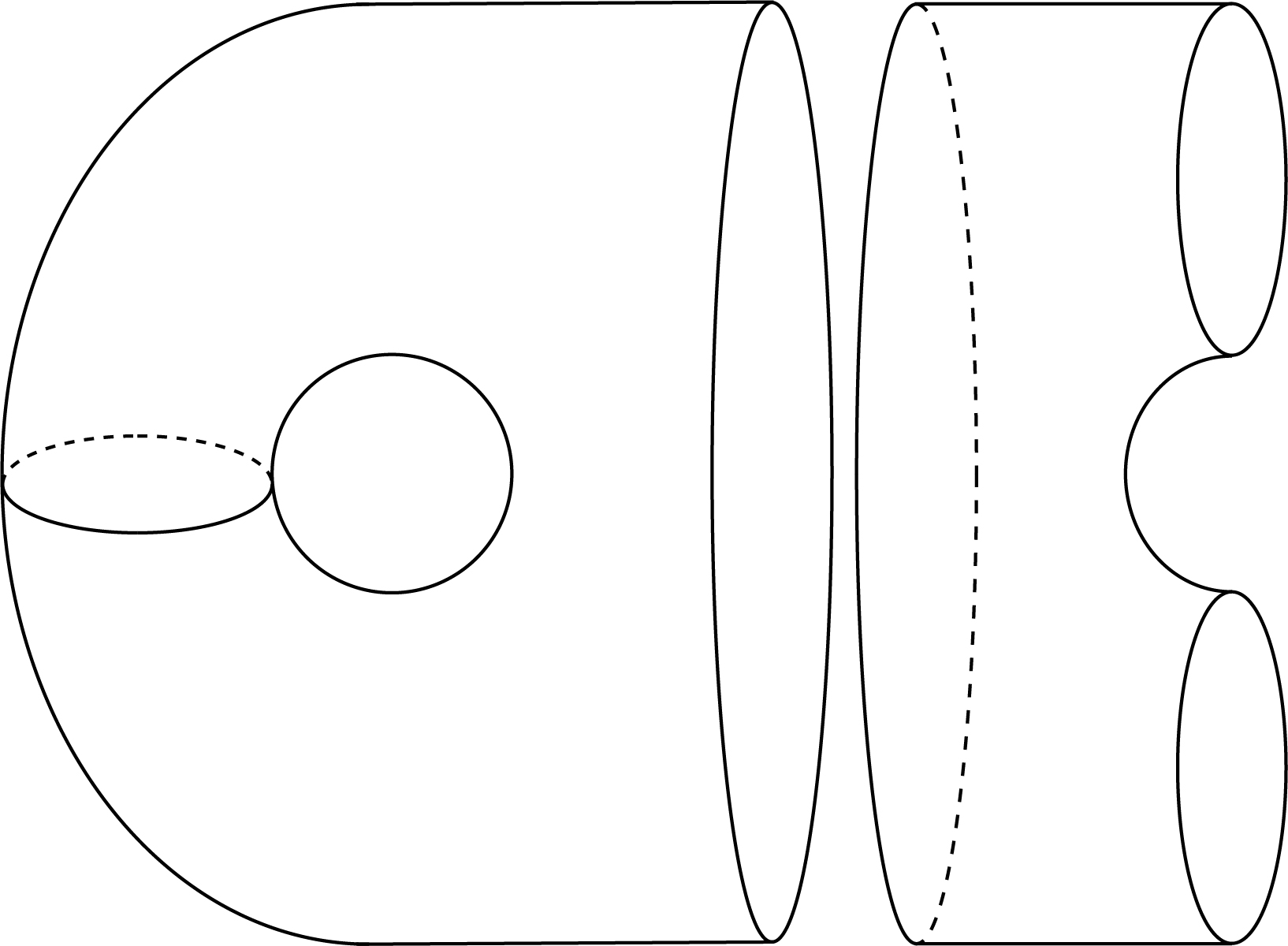}
           \put(2,41.5){$b = b'$}
           \put(62.5,73){$a''$}
       \end{overpic} \\[.4cm] 
       \center$\color{Maroon}\lk(\sigma\setminus\{b,b'\})\cong S(M_{0,4})$ &  $\color{Maroon}\lk(\sigma''\setminus\{b,b'\})\cong S(M_{1,1})$ \\[0.4cm]
       \hline
    \end{tabular}    
    \caption{Various links in the
    proof of \Cref{lem:compatibility_sigma}. Row 1 shows when $e_a$ and $e_{b}$ are not
    adjacent, Row 2 when
$e_a$ and $e_{b}$ are incident on a single common vertex, and 
Row 3 when $b = b'$.
Because $a$ and $a'$ both separate $b$ and $b'$ in $M$, the left column is identical replacing $a$ with $a'$, $\sigma$ with $\sigma'$, and $b$ with $b'$.}
\label{fig:FourLinks}
\end{figure}
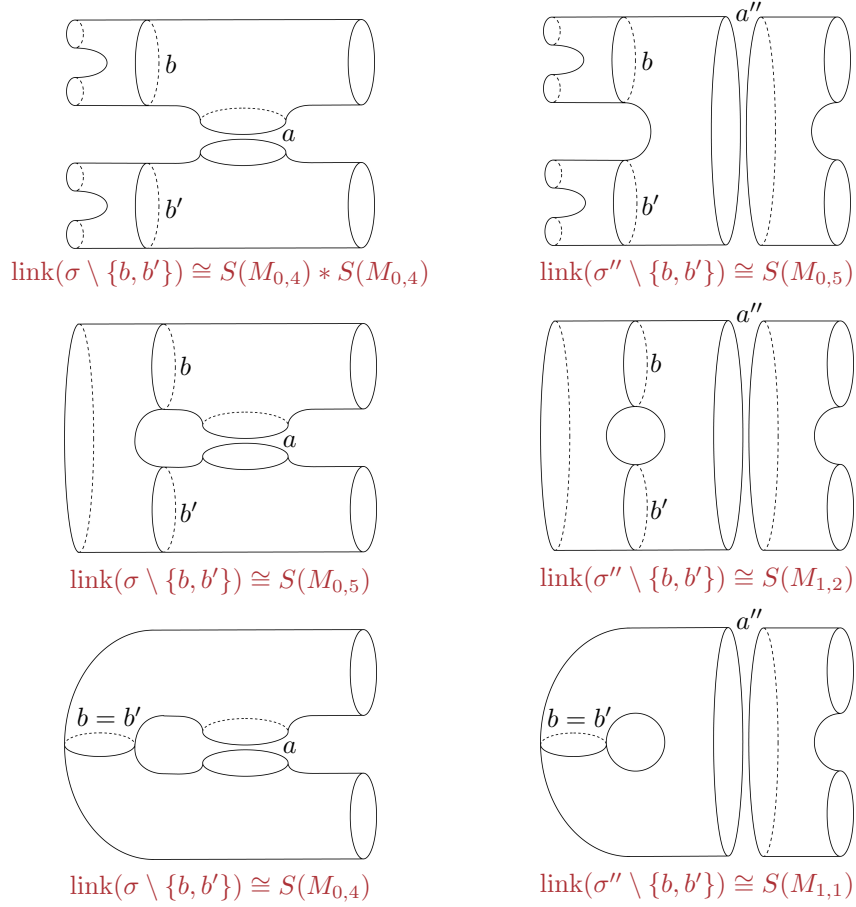
 
\begin{remark*}\label{rmk:compatibility_sigma} 
  The map $f$ need not be an isomorphism for
  \Cref{lem:compatibility_sigma} to hold.  Suppose that $f : Z
  \subset \sph(M_\Gamma) \to \sph(M_{\Gamma'})$ such that $X_\sigma
  \subset Z$, $f|_\sigma$ is the edge map of an isomorphism
  $\Delta_\sigma \xrightarrow{\sim} \Delta_{f\sigma}$, and $h :
  M_\Gamma \to M_{\Gamma'}$ is the diffeomorphism constructed as
  above.  If $f$ extends to an isomorphism $\lk(\sigma \setminus
  \{\a,\b,\c\}) \xrightarrow{\sim} \lk(f\sigma \setminus
  \{f\a,f\b,f\c\})$ for any $\a,\b,\c \in \sigma$ such that
  $e_\a \cup e_\b \cup e_\c$ is connected in $\Delta_\sigma$, then
  $h$ induces $f$ on $X_\sigma$: in particular, in the proof above 
  $\lk(\sigma' \setminus \{b,b'\}) \subset \lk(\sigma \setminus
  \{a,b,b'\})$ and $\lk(f\sigma' \setminus \{f\b,f\b'\}) \cong
  \lk(\sigma' \setminus \{\b,\b'\})$, hence the argument 
  applies without modification.
  We will use this fact in \Cref{sec:loc_finite_infinite}.
\end{remark*}

\begin{lemma}\label{lem:compatibility_flip}
  Suppose that $\rho,\rho'$ are maximal sphere systems 
  that differ by a flip move.
  If a diffeomorphism 
  $g : M_\Gamma \to M_{\Gamma'}$ agrees 
  with $f$ on $X_\rho$, then it likewise agrees with $f$ on  
  $X_{\rho'}$.
\end{lemma}

\begin{proof} 
  Suppose that $\rho'$ is obtained from $\rho$ 
  by a flip move $\a \mapsto \a'$, and let $\rho_0 \subset
  \rho \cap \rho'$ denote the set of spheres in $\rho$ 
  adjacent to $\a$
  (equivalently, spheres in $\rho'$ adjacent to $\a'$).  
  We note that $\lk(\rho \setminus
  \a) = \lk(\rho' \setminus \a')$, $\rho \triangle \rho'
  = \{\a,\a'\}\subset \lk(\rho\setminus \a)$, and for
  $\gamma \in (\rho \cap \rho')\setminus \rho_0$, $\lk(\rho \setminus
  \gamma) = \lk(\rho' \setminus \gamma)$. 

  It follows that  
  $X_{\rho'} \setminus X_\rho \subset \bigcup_{\b \in \rho_0}
  \lk(\rho' \setminus \b)$, thus it suffices that $g_*$
  agrees with $f_*$ on $\lk(\rho' \setminus \b)$ for $\b \in
  \rho_0$. 
  If $\b$ is a loop in $\Delta_{\rho'}$, 
  then $\lk(\rho' \setminus \b) = \{\b\} \subset \rho$.
  Otherwise, let $M$ be the complementary component of 
  $\rho \setminus \{\a,\b\} = \rho' \setminus 
  \{\a',\b\}$ containing $\a,\a'$, and $\b$; 
  $M' = g(M)$ is then the 
  complementary component of $f\rho \setminus \{f\a,f\b\}$
  containing $f\a,f\a'$ and $f\b$.  Let $\lk(\rho
  \setminus \b) = \{\b,\b',\b''\}$ and $\lk(\rho'
  \setminus \b) = \{\b,\b^\dag,\b^\ddag\}$, all spheres
  of which are contained in $M$.  We consider two cases (see \Cref{fig:FlipMoveAgrees}):

  \begin{enumerate}[label=\textit{(\roman*)}]
  \item \textit{$e_\a$ and $e_\b$ are incident on two vertices}  
    ($M \cong M_{1,2} \setminus \partial M_{1,2}$). Then $\b$ is
    non-separating in $M$, hence exactly one of
  $\b^\dag,\b^\ddag$ is separating in $M$, say $\b^\dag$.
  Thus $\lk(\b^\dag \cup (\rho' \setminus \{\a',\b\}))
  \cong \sph(M_{1,1})$ and 
  $\lk(\b^\ddag \cup (\rho' \setminus \{\a',\b\})) \cong
  \sph(M_{0,4})$. 
  The same holds for the $f$ images of these
  sets, and since only $g\b^\dag$ is separating in $M'$, likewise
  for their $g_*$ images.  Hence $g_*\b^\dag = f\b^\dag$ 
  and $g_*\b^\ddag = f\b^\ddag$.

\item \textit{$e_{\a'}$ and $e_\b$ are incident on only one vertex}
  ($M \cong M_{0,5}\setminus \partial M_{0,5}$). 
  Exactly one of $\b^\dag,\b^\ddag$ is disjoint from $\b'$, say
  $\b^\dag$.  Then likewise $g\b^\dag$ is disjoint from $g\b'$ and
  $g\b^\ddag$ is not.  Since
  $f$ preserves disjointness and $g_*\b' = f\b'$, $g_*\b^\dag
  = f\b^\dag$ and $g_*\b^\ddag = f\b^\ddag$. 
  \qedhere
\end{enumerate}
  \end{proof}

  \begin{figure}[ht!]
    \centering
    \begin{tabular}{>{\centering\arraybackslash}m{0.473\textwidth} >{\centering\arraybackslash}m{0.473\textwidth}}
        \centering
    \begin{overpic}[width=4.7cm]{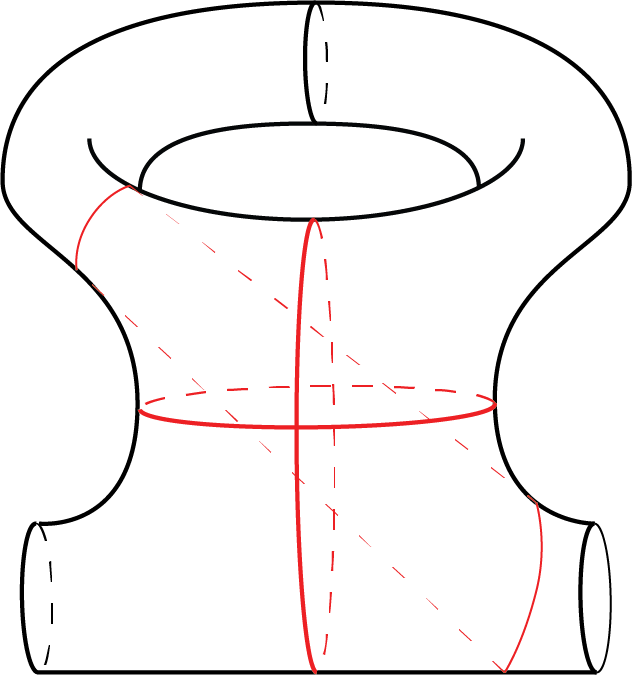}
        \put(39,88){$a'$}
        \put(10,53){\color{red}{$b$}}
        \put(75,40){\color{red}{$b^\dag$}}
        \put(47,-6){\color{red}{$b^\ddag$}}
    \end{overpic}
         &
       \begin{overpic}[width=5.3cm]{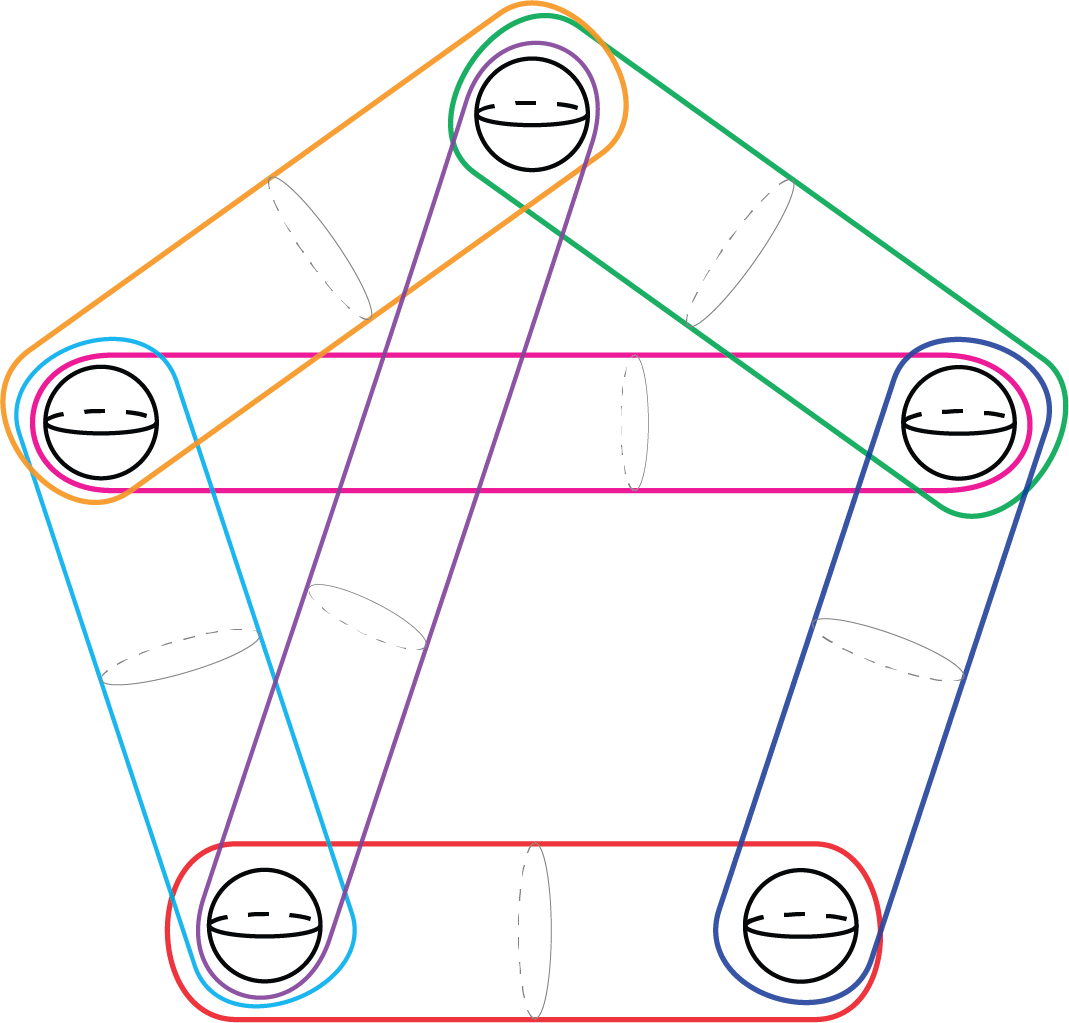}
        \put(25,83){\color{orange}{$b$}}
        \put(75,83){\color{teal}{$b^\dag$}}
        \put(6,24){\color{cyan}{$b'$}}
        \put(50,-5){\color{red}{$a'$}}
        \put(90,24){\color{blue}{$a$}}
        \put(61,43){\color{magenta}{$b^\ddag$}}
        \put(40,30){\color{violet}{$b''$}}
        \end{overpic} \\[1.5cm] 
    \end{tabular}
    \vspace{5pt}
    \caption{On the left is $M \cong M_{1,2} \setminus \partial M_{1,2}$ illustrating part (i). 
  On the right is $M \cong M_{0,5}$ with the spheres in part (ii) drawn in.}
  \label{fig:FlipMoveAgrees}
\end{figure}
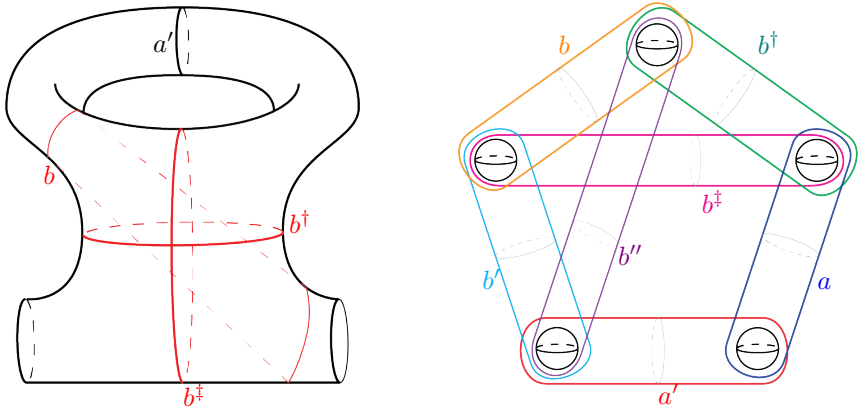

\begin{corollary}\label{cor:ManifoldMCGActionSurjective}
    If a diffeomorphism $g$ agrees with
    $f$ on $X_\sigma$, then $g_* = f$. 
\end{corollary}

\begin{proof}
  Fix a compact exhaustion $M_i$ of $M_\Gamma$ such that $\partial
  M_i \subset \sigma$. For any $\b \in \sph(M_\Gamma)$, choose
  $i$ such that $\b \subset M_i$; let 
  $\bar \sigma = \sigma \cap M_i\setminus \partial M_i$ and fix 
  $\bar \sigma'$ a maximal sphere system in $M_i$ containing
  $\b$.  By \Cref{prop:FlipMoveConnectivity} there is a finite 
  sequence of maximal sphere systems in $M_i$ between $\bar
  \sigma$ and $\bar \sigma'$ by successive flip moves, 
  which then extends to such a sequence between $\sigma$ and
  $\sigma' = \bar\sigma' \cup (\sigma \setminus \bar\sigma) 
  \ni \b$ of maximal sphere systems in $M_\Gamma$.
  Inductively applying 
  \Cref{lem:compatibility_flip} shows that $g$ 
  agrees with $f$ on $X_{\sigma'}$, and in particular $g_*(\b) =
  f\b$. As this is true for any $\b$, it follows that $g_*=f$.
\end{proof}

\begin{remark*}\label{rmk:uniqueness} 
  It follows that $g_*$ is uniquely determined by its restriction to
  $X_\sigma$ among the class of isomorphisms induced by
  diffeomorphisms $M_\Gamma \to M_{\Gamma'}$. 
  In particular, if $g'$ is another
  diffeomorphism that
  agrees with $g_*$ over $X_\sigma$, then $g'_* = g_*$. 
\end{remark*}

We now restate and prove the main theorem of this paper: 

\begin{namedtheorem}[\Cref{thm:IvanovRigidity0}] Let
  $\Gamma,\Gamma'$ be two locally finite connected graphs.  Suppose
  $f: \sph(M_{\G})\to \sph(M_{\G'})$ is an isomorphism. Then $f$ is
  induced by a diffeomorphism $h:M_{\G} \to M_{\G'}$. In particular,
  \begin{enumerate*}[label=(\roman*)] \item $\G$ and $\G'$ are
    proper homotopy equivalent and \item when $\Gamma$ is not a
    graph of rank $r$ with $s$ rays such that $2r+s<4$ or $(r,s)\in
    \{(0,4), (2,0)\}$,  $\Aut(\sph(M_{\G}))\cong \Map(\G)$ as
  topological groups.  \end{enumerate*} \end{namedtheorem}
  \begin{proof} 
    From \Cref{lem:compatibility_sigma} and
    \Cref{cor:ManifoldMCGActionSurjective} we obtain that 
    $f$ is induced by a diffeomorphism, and in particular the action
    $\xi : \Map(M_\Gamma) \to \Aut(\sph(M_\Gamma))$ is surjective. 
    Thus $M_{\G}$ and $M_{\G'}$ are diffeomorphic, hence $M_\Gamma,
    M_{\Gamma'}$ have the same characteristic triple and likewise do
    $\Gamma, \Gamma'$.  By \Cref{thm:ADMQ} $\G$ and $\G'$ are proper
    homotopy equivalent, which implies \textit{(i)}.  

    Let $\rho : \Map(\Gamma) \to \Aut(\sph(M_\Gamma))$ be the
    action induced by $\xi$ in 
    \Cref{thm:GraphMCGAction}; note that $\rho$ and 
    $\xi$ have the same image in $\Aut(\sph(M_\Gamma))$, hence
    $\rho$ is surjective.
    To show \textit{(ii)}, suppose $\Gamma$ is a graph satisfying
    the hypotheses of \textit{(ii)}: by \Cref{thm:GraphMCGAction} 
    $\rho$ is injective, hence a group isomorphism. To obtain
    that $\rho$ is a homeomorphism, we observe that the pullback of
    the permutation topology on $\Aut(\sph(M_\Gamma))$ is compatible
    with the usual quotient topology on $\Map(\Gamma)$.  In
    particular, 
    by \cite[Prop.\ 7.1]{udall2024spherecomplexlocallyfinite}
    the quotient topology is identical to 
    the topology generated by the subbasis at identity 
    $\{U'_a\}_{a \in \sph(M_\Gamma)^{(0)}}$, where 
    \[
      U'_a =
    \rho^{-1}(U_a) = \{\phi \in \Map(\Gamma) : \rho(\phi)(a) = a\}
  \]
    and $\{U_a\}$ is the subbasis for $\Aut(\sph(M_{\Gamma}))$ given
    after \Cref{def:sphcpx}.
  \end{proof}

  Consequently, we obtain another proof of the following result,
  originally proven in Proposition 4.11 of \cite{AB2021}. 

\begin{corollary}
    For any locally finite connected graph $\G$, $\Map(\G)$ is
  Polish.  
\end{corollary}
\begin{proof}
    It is a standard result of descriptive set theory that the
  automorphism group of a countable graph equipped with the
permutation topology is Polish. Except for the finitely many cases
excluded in \Cref{thm:IvanovRigidity0}(ii), $\Map(\G)$ is
topologically isomorphic to the automorphism group of a countable
graph by \Cref{thm:IvanovRigidity0}(ii), so the result follows in
these cases. For the excluded cases, $\Map(\G)$ is countable and 
discrete, hence Polish.  
\end{proof}

\section{Rigidity of the sphere complex for finite-type doubled
handlebodies with $S^2$-boundaries}\label{sec:finite_rigidity} 

The goal of this section is to establish \Cref{thm:geom_rigidity}.
The setup and arguments closely follow \cite[Section
3]{bering2024finite}, where analogous results are proven for doubled
handlebodies with empty boundary. For completeness, we reintroduce
their setup and highlight the adjustments needed to prove the
results in the setting of doubled handlebodies with nonempty
boundary. For clarity, we will match the notation conventions in
\cite{bering2024finite} throughout this section.  We first construct
a set $X_0$ that can be extended to a finite strongly rigid set $X$
by adding a collection of ``good spheres.'' 

    Throughout this section, we will implicitly use 
    \Cref{prop:full_subcpx} to identify sphere 
    complexes of
    submanifolds with subcomplexes of the sphere complex of a parent
    manifold. In particular, links of simplices will be identified
    with sphere complexes of submanifolds.
    \par 
     Given a subcomplex $X$ of $\sph(M_{\G})$, we say that $X$ is
     \emph{geometrically rigid} if for every simplicial locally
     injective map $f:X\to \sph(M_{\G})$ there is a diffeomorphism
     $h$ of $M_{\G}$ such that the restriction of $h_*$ to $X$
     agrees with $f$. To prove \Cref{thm:geom_rigidity}, we first
     exhaust $\sph(M_{n,s})$ by a sequence of geometrically rigid
     sets in \Cref{prop:GeoRigidExhaustion}, and then show that
     these sets are strongly rigid.
    
    \subsection{Constructing a geometrically rigid set $X$.} Let $Y$ be a maximal collection of disjoint spheres $S_i \subset
    \sph(M_{n,s})$ whose union is non-separating.  Then, by removing
    a small regular neighborhood of $Y$ from $M_{n,s}$, we obtain a
    manifold with no genus. Specifically, 
    \begin{equation*}
      N \coloneq
      M_{n,s}\setminus \mathrm{nbd}(Y) \cong M_{0, 2n + s}
    \end{equation*}

Let $Z$ be the collection of all spheres in $\sph(M_{n,s})$ that are
disjoint from $Y$.  By construction of $N$, it follows $Z =
\sph(N)$.  

The spheres in $\partial N$ come in two types: 
\begin{itemize}
\item[1)] those coming from the original boundary of $M_{n,s}$; 
\item[2)] pairs of spheres $S_i^+, S_i^-$ coming from removing a sphere $S_i \in Y$.  
\end{itemize}
We define a labeling map $\delta : \partial N \to Y \cup \partial M_{n,s}$ recording where boundary components of $N$ came from.   Specifically, for spheres $S_i^{\pm} \in \partial N$ coming from removing a sphere $S_i \in Y$, $\delta(S_i^{\pm}) = S_i$. On the other hand, for spheres originally in the boundary $S_j \subset \partial M_{n,s}$, we define $\delta(S_j) = S_j$. See \Cref{fig:cutMns} for an illustration.  

 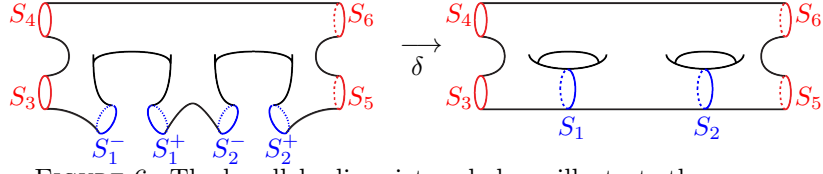
\begin{figure}[ht!]
    \centering
        \begin{overpic}[width=10cm]{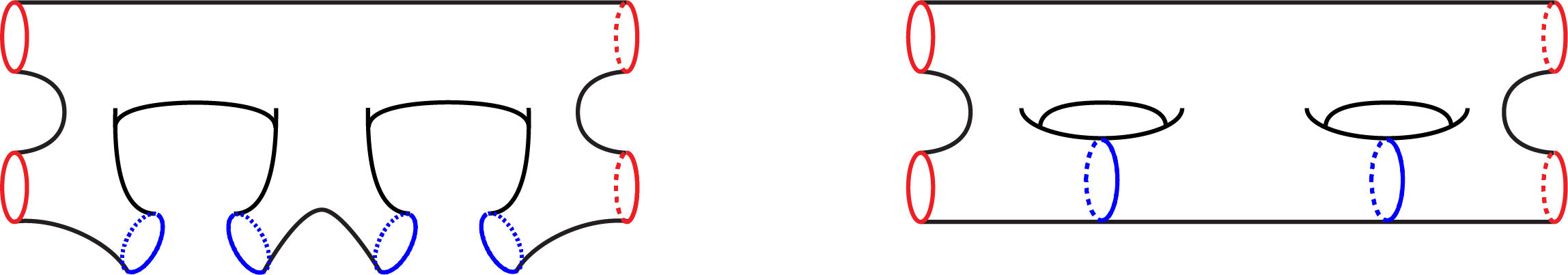}
            \put(-4,4){\color{red}{$S_3$}}
            \put(-4,15){\color{red}{$S_4$}}
            \put(41,4){\color{red}{$S_5$}}
            \put(41,15){\color{red}{$S_6$}}
            \put(7,-3){\color{blue}{$S_1^-$}}
            \put(15,-3){\color{blue}{$S_1^+$}}
            \put(23,-3){\color{blue}{$S_2^-$}}
            \put(30,-3){\color{blue}{$S_2^+$}}
            \put(54,4){\color{red}{$S_3$}}
            \put(54,15){\color{red}{$S_4$}}
            \put(100.5,4){\color{red}{$S_5$}}
            \put(100.5,15){\color{red}{$S_6$}}
            \put(69,0){\color{blue}{$S_1$}}
            \put(87,0){\color{blue}{$S_2$}}
            \put(48,11){$\longrightarrow$}
            \put(49.5,8){$\delta$}
        \end{overpic}
    \caption{The handlebodies pictured above illustrate the process of removing a neighborhood of $Y$ from $M_{n,s}$. The red spheres correspond to boundary components of the uncut manifold.  The blue spheres play the role of $Y$.}
    \label{fig:cutMns}
\end{figure} 

We are interested in the subgraph of the 1-skeleton
of $\sph(M_{n,s})$ 
spanned by $Y$ and $Z$, which we denote by 
\begin{equation*}
    X_0 := \langle Y \cup Z \rangle.  
\end{equation*}
\noindent
Because $X_0$ is a join of $Y$ and $Z$, and $Y$ is complete, $X_0$
is not geometrically rigid. See \Cref{fig:NotRigid}. 

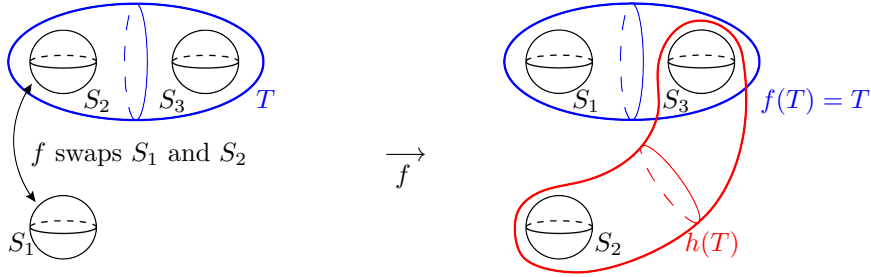
\begin{figure}[ht!]
    \centering
    \begin{overpic}[width=10cm]{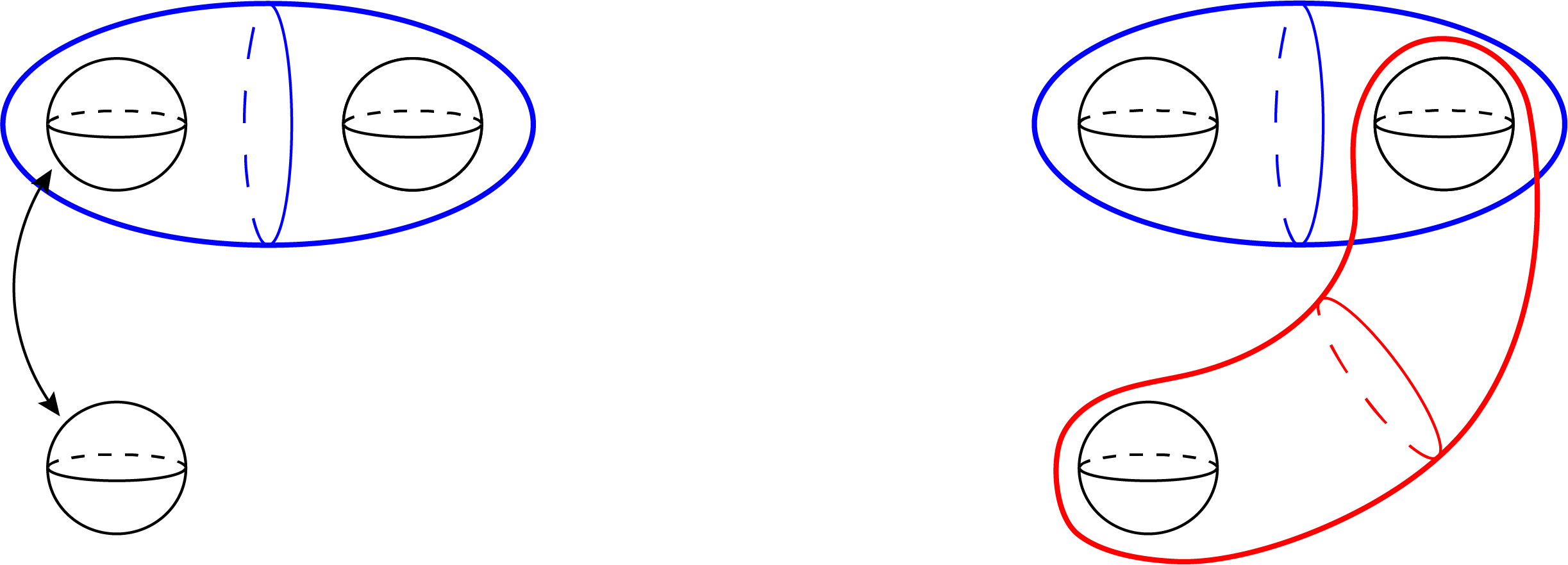}
        \put(0,3){$S_1$}
        \put(10,22){$S_2$}
        \put(20,22){$S_3$}
        \put(33,22){\color{blue}{$T$}}
        \put(3,15){$f$ swaps $S_1$ and $S_2$}
        \put(78,3){$S_2$}
        \put(75,22){$S_1$}
        \put(87,22){$S_3$}
        \put(100,22){\color{blue}{$f(T) = T$}}
        \put(90,3){\color{red}{$h(T)$}}
        \put(50,15){$\longrightarrow$}
        \put(51.5,12){$f$}
    \end{overpic}
    \caption{Let $S_1, S_2, S_3 \in Y$ and $T \in Z$ be as shown above.  Consider the simplicial isomorphism $f \colon X \to X$ with $f|_{Z} = \mathrm{Id}_Z$, and $f|_{Y}$ permuting $S_1$ and $S_2$.
    Any homeomorphism inducing the permutation of $S_1$ and $S_2$ cannot fix $T$. So, $f$ is not induced by any homeomorphism $h$.}  
    \label{fig:NotRigid}
\end{figure}

 Fix a locally injective simplicial map $f \colon X_0 \to
 \sph(M_{n,s})$. Then intersecting pairs of spheres in $Z$ have
 $X_0$-detectable intersection, so by \Cref{lem:XdetInt} $f(Z)$ must
 fill a connected submanifold. As $f(Z)$ is disjoint from $f(Y)$,
 the submanifold filled by $f(Z)$ must lie in a component of
 $M_{n,s}\setminus \nbd(f(Y))$. By a complexity argument, this is
 only possible if $M_{n,s}\setminus \nbd(f(Y))$ is connected, and
 thus homeomorphic to $M_{0, 2n+s}$. That is, if $M_{n,s}\setminus
 \nbd(f(Y))$ is disconnected, there is no way the manifold filled by
 $f(Z)$ could embed into any of the components of $M_{n,s}\setminus
 \nbd(f(Y))$, as can be seen by analyzing the possible components of
 $M_{n,s}\setminus \nbd(f(Y))$. By \Cref{lem:GenusZeroRigidity},
 there is a homeomorphism $h \colon N = M_{n,s} \setminus \nbd(Y)
 \to N' = M_{n,s} \setminus \nbd(f(Y))$  arising from $f|_{X_0}$,
 such that $h|_Z = f|_Z$.   

To build a finite rigid set, we will add spheres to the set $X_0$,
namely pairs of ``good spheres.'' The addition of these good spheres
will allow us to keep track of the pairs of boundary components of
$M_{0, 2n+s}$ that correspond to a component of $Y$.  In particular,
we will show that when $X_0$ is extended to a set $X$ containing
good pairs $h: N \to  N'$ ascends to a homeomorphism $\hat{h}
\colon M_{n,s} \to M_{n,s}$, which induces $f: X \to \sph(M_{n,s})$. 

We will now recall the definitions of a good sphere and a good pair.  

\begin{definition}[{\cite[Section 3]{bering2024finite}}]
Given $A \subset Y$, let $A^+, A^- \subset \partial N$ be the boundary spheres obtained from removing a neighborhood of $A$ from $M_{n,s}$; that is, $\delta(A^{\pm}) = A$.  
Let $a$ be an essential sphere in $M_{n,s}$ that essentially intersects $A$ in a single simple closed curve, and such that $a$ is disjoint from all other spheres of $Y$.  

When we descend to the cut manifold $N$, the sphere $a$ decomposes as the union of the two disjoint disks $a^+$ and $a^-$, with $\partial a^+ \subset A^+$ and $\partial a^- \subset A^-$. The boundary of regular neighborhoods of $A^+ \cup a^+$ and $A^- \cup a^-$ are disjoint pairs of pants. See \Cref{fig:GoodSphere} for an illustration of this setup. 

\begin{figure}[ht!]
    \centering
    \begin{overpic}[width=10cm]{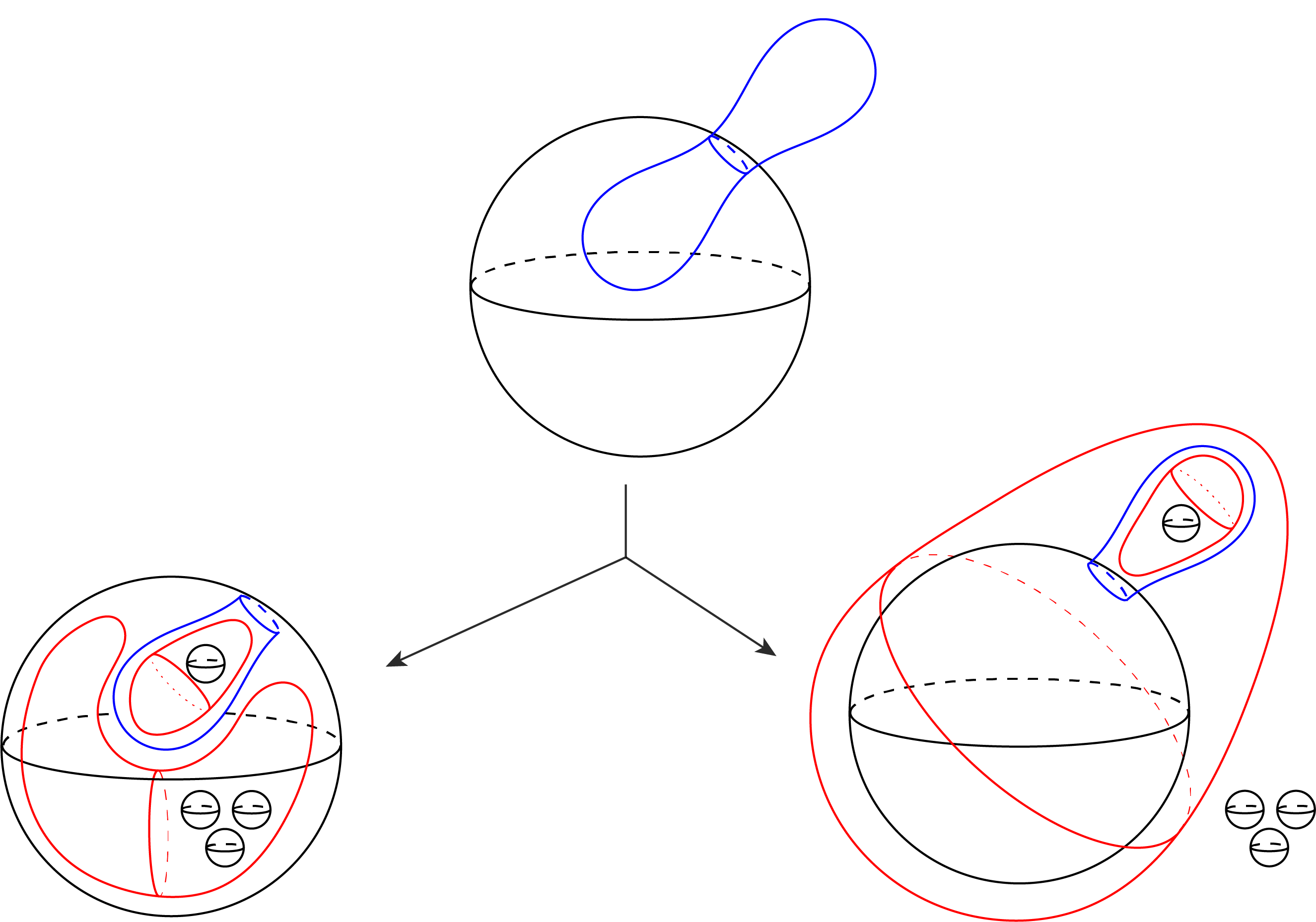}
        \put(63,48){$A$}
        \put(60,62){\color{blue}{$a$}}
        \put(90.5,18){$A^+$}
        \put(90.5,32){\color{red}{$S_3$}}
        \put(95,38){\color{red}{$S_4$}}
        \put(82,31){\color{blue}{$a^+$}}
        \put(-5,17){$A^-$}
        \put(12,23){\color{blue}{$a^-$}}
        \put(11,16){\color{red}{$S_1$}}
        \put(7,5){\color{red}{$S_2$}}
    \end{overpic}
    \caption{
   The sphere $a$ is good for $A$.  In the top picture, the spheres $a$ and $A$ are pictured in the uncut manifold $M_{n,s}$, and they intersect essentially.  After cutting along $Y$, we obtain the picture in $N$ below.  The six spheres pictured inside of $S_2$ and outside of $S_4$ are meant to illustrate that these spheres are essential and not necessarily peripheral, unlike $S_1$ and $S_3$, which are peripheral since $a$ is good.}
    \label{fig:GoodSphere}
\end{figure}
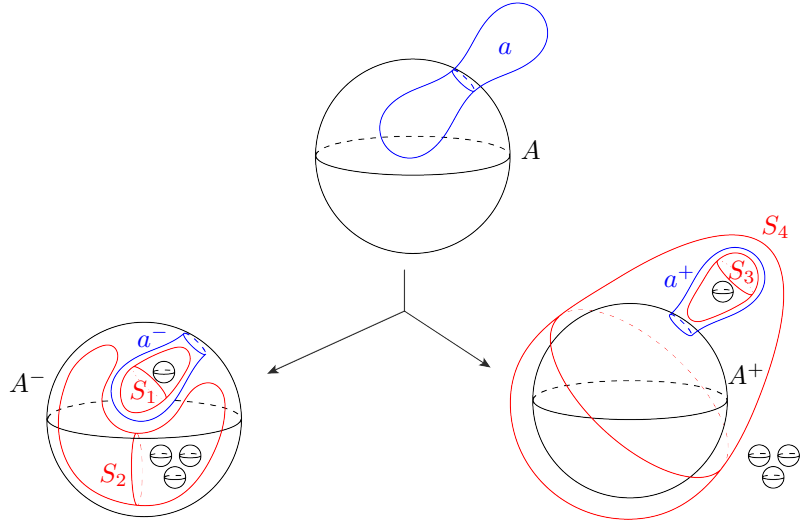

We denote $\partial (\mathrm{nbd}(A^- \cup a^-)) = A^- \cup S_1 \cup S_2$ and $\partial (\mathrm{nbd}(A^+ \cup a^+)) = A^+ \cup S_3 \cup S_4$. Let $\partial(A, a) := A^+ \cup A^- \cup S_1 \cup S_2 \cup S_3 \cup S_4$.  
If $S_1$ and $S_3$ are peripheral in the cut manifold $Y$, we say that $a$ is a \emph{good sphere} for $A$.  
Suppose that $a$ and $a'$ are disjoint good spheres for $A$.  Then if $\partial(A, a) \cap \partial(A, a') = A^+ \cup A^-$, then $a$ and $a'$ are said to be a \emph{good pair} for $A$.  

\end{definition}

Good pairs can always be found when $2n+s \ge 6$, i.e., when $\partial N$ has at least $6$ components. This is because each sphere in a good pair for $A$ requires the use of two boundary components of $N$ other than $A^{\pm}$, and by assumption, these pairs of boundary components must be distinct for each sphere in the pair. 
\par 
Let $X$ be the subcomplex of $\sph(M_{n,s})$ spanned by the vertices of  $X_0$, together with a choice of a good pair for each sphere $A \subset Y$.  

\subsection{Geometric rigidity of $X$ for $M_{n,s}$}
Throughout the rest of \Cref{sec:finite_rigidity}, we fix a choice of $M_{n,s}$, with $2n+s\geq 6$.
\par 
Up to this point, the setup has been the same as in \cite[Section 3]{bering2024finite}. A key difference in the setting of $M_{n,s}$ when $s \ne 0$ is that the good pairs in $X$ serve an additional role as discussed below in the proof of \Cref{prop:RigidityOnY}.

\begin{proposition}[{\cite[Lemma 13]{bering2024finite}}]
    Let $f \colon X \to \sph(M_{n,s})$ be a locally injective, simplicial map.  Let $h: N \to N'$ be the homeomorphism inducing $f|_Z$ as in \Cref{lem:GenusZeroRigidity}.  
    Then, for every sphere $A \in Y$ we have $\delta(h(A^\pm)) = f(\delta(A^\pm))$ where $A^+$ and $A^-$ denote the corresponding boundary spheres in $N$, i.e., $\delta(A^\pm) = A$. 
    \label{prop:RigidityOnY}
\end{proposition}

\begin{proof}
  For all $A \in Y$, the spheres $A^\pm$ are not mapped by $h$ to spheres $S \in \partial N'$ coming from the original boundary of $M_{n,s}$.  This is true because the original boundary spheres of $M_{n,s}$ do not intersect any essential spheres in $M_{n,s}$. In particular, the boundary spheres do not admit any good spheres or good pairs. 
  Thus, the good spheres in $X$ allow us to differentiate between the boundary components of $N$ that result from removing $Y$ and the original boundary components of $M_{n,s}$. A priori, $h$ could map a sphere coming from $Y$ to an original boundary sphere of $M_{n,s}$.  However, following the argument in {\cite[Lemma 13]{bering2024finite}} this cannot happen.   
Given this additional role of the good pairs, the proof of {\cite[Lemma 13]{bering2024finite}} applies in the more general setting of $M_{n, s}$ when $s\neq0$. 
\end{proof}

\begin{proposition}[{\cite[Proposition 14]{bering2024finite}}] The set $X$ is geometrically rigid. 

\end{proposition}

\begin{proof}
\Cref{prop:RigidityOnY} shows that $h: N \to N'$ ascends to a map $\hat{h}: M_{n,s} \to M_{n,s}$ such that $\hat{h}_\ast$ and $f$ agree on $X_0$.  It remains to verify that $\hat{h}_\ast$ and $f$ agree on the good spheres in $X$.  The argument proceeds exactly as in {\cite[Proposition 14]{bering2024finite}}. 
\end{proof}

\subsection{Exhaustion by geometrically rigid sets}
In this section, we generalize Proposition 22 of \cite{bering2024finite} and prove that we can find a nested family of geometrically rigid sets that exhaust $\sph(M_{n,s})$ (see \Cref{prop:GeoRigidExhaustion}). The proof will primarily follow the argument used in \cite{bering2024finite}, with two differences that will be explicitly stated.

\begin{definition}
    Suppose $P$ is a pants decomposition of $M_{n,s}$ and $a\in P$. A sphere $b\in\sph(M_{n,s})$ is a \emph{split sphere for $(a,P)$} if $a$ is the unique sphere in $P$ intersecting $b$. 
\end{definition}

\begin{definition}
    If $X\subseteq\sph(M_{n,s})$ is a subcomplex, $P\subseteq X^{(0)}$ and $b\in X^{(0)}$ is a split sphere for $(a,P)$, then we say that $P$ is $X-$split at $a$ by $b$. We say that $P$ is $X-$split if it is $X-$split at $a$ for some $a\in P$. If $X$ contains every split sphere for $P$, then we say it is fully $X-$split.
\end{definition}

\begin{definition}
    Suppose $X\subseteq\sph(M_{n,s})$ is a subcomplex and $a\in X^{(0)}$. A pair of distinct, disjoint spheres $(b_1.b_2)$ in $\sph(M_{n,s})^{(0)}$ is a \emph{split pair} for $a$ relative to $X$ if there exists pants decompositions $P_1,P_2\subseteq X^{(0)}$, both containing $a$ such that $b_i$ is a split sphere for $(a,P_i)$ for $i=1,2$. See \Cref{lab:SplitPair}.
\end{definition}

\begin{figure}[ht!]
    \centering
    \begin{tabular}{>{\centering\arraybackslash}m{0.473\textwidth} >{\centering\arraybackslash}m{0.473\textwidth}}
       \centering
       \begin{overpic}[scale=0.3]{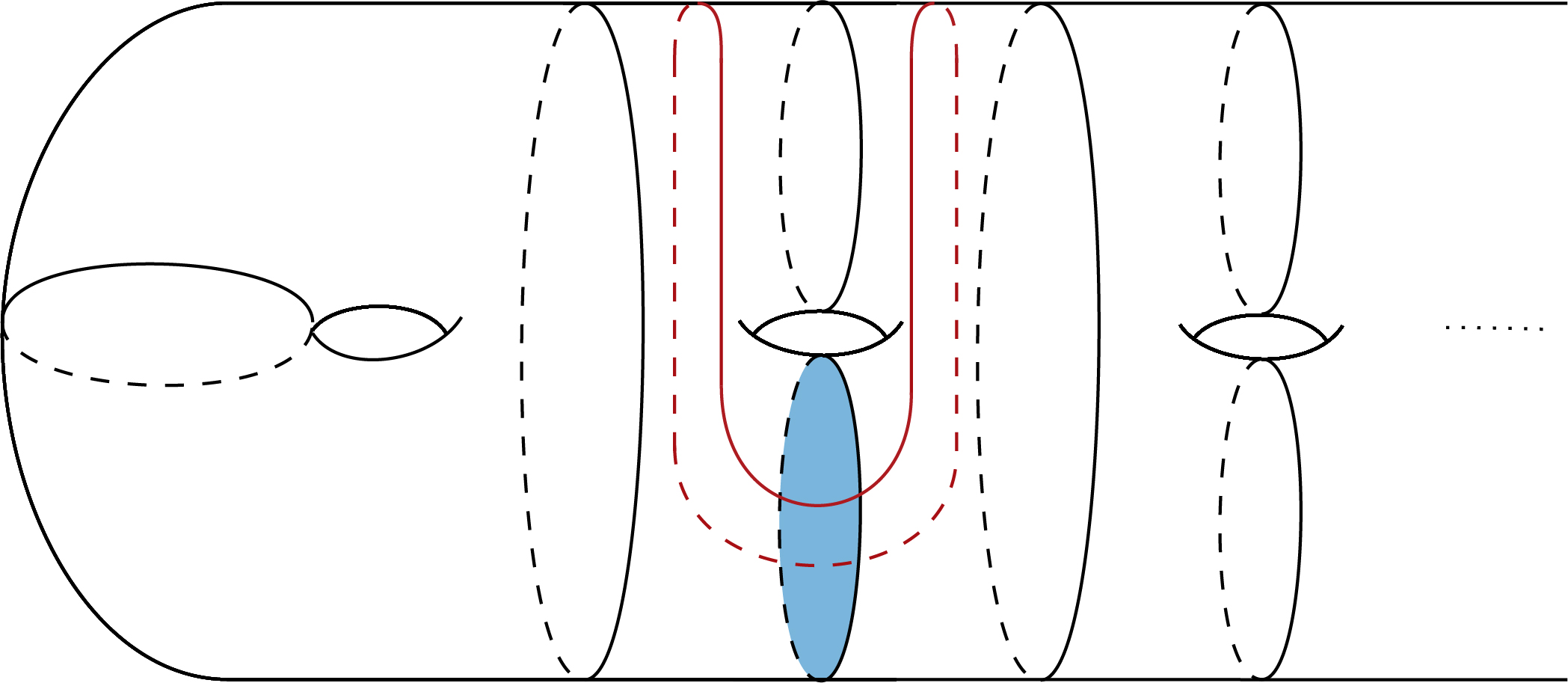}
        \put(43,46){$\color{Maroon}b_1$}
        \put(51,-5){$a$}
        \put(-7,5){$P_1$}
        \end{overpic}
         &
       \begin{overpic}[scale=0.3]{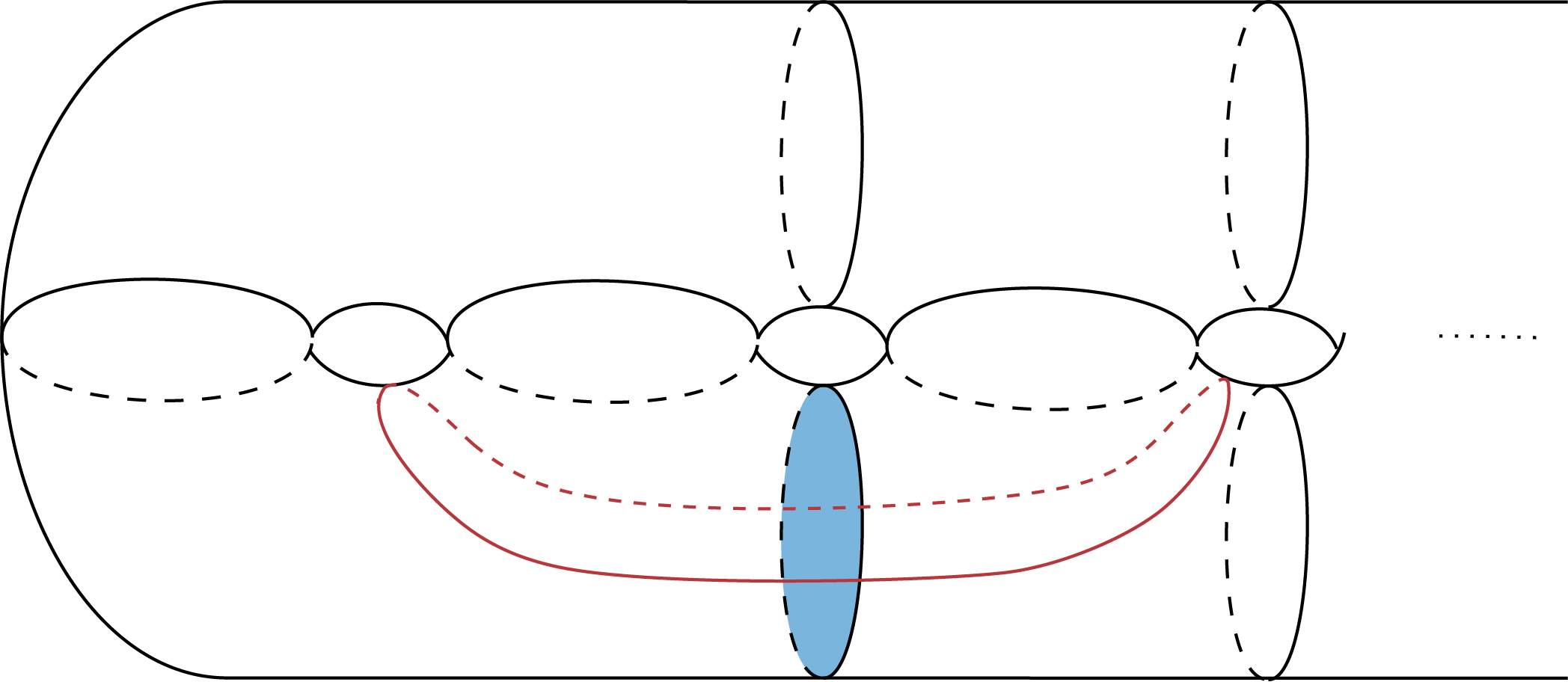}
        \put(23,5){$\color{Maroon}b_2$}
        \put(51,-5){$a$}
        \put(-7,5){$P_2$}
        \end{overpic}  
    \end{tabular}
    \vspace{5pt}
    \caption{The spheres $b_1$ and $b_2$ give a split pair for $a$ with respect to the pants decompositions $P_1$ and $P_2$.}
    \label{lab:SplitPair}
\end{figure}
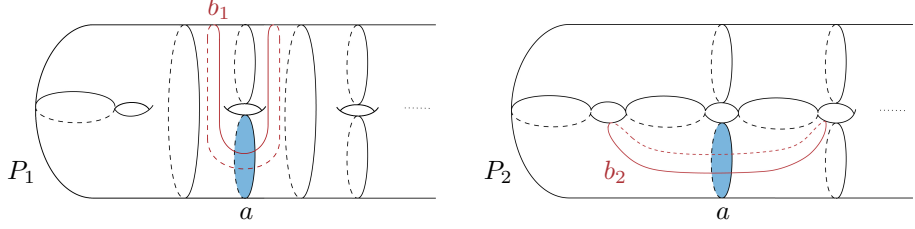

\begin{lemma}[{\cite[Lemma 18]{bering2024finite}}]
\label{lem:PantsInductiveStep}
    Suppose $X\subseteq\sph(M_{n,s})$ is a geometrically rigid subcomplex and $a\in X^{(0)}$ has a split pair $(b,c)$. Then the subcomplex $X_{b,c}$ induced by $X\cup\{b,c\}$ is geometrically rigid.
\end{lemma}

The following is Lemma 20 of \cite{bering2024finite}. The lemma as stated in \cite{bering2024finite} claims that if $a$ and $c$ are two distinct adjacent spheres in a pants decomposition $P$ of $M_{n,s}$, then the component of $M_{n,s}\setminus (P\setminus\{a,c\})$ is homeomorphic to $M_{0,5}$. However, consider Case 2 of \Cref{fig:pants-exhaustion}, where when $a$ and $c$ are the two nonperipheral spheres on the left, and the given component is homeomorphic to $M_{1,2}$.
\par
Instead, we note the following lemma, which follows directly from their proof. We simply assume that $a$ and $c$ have the property that $M_{n,s}\setminus (P\setminus\{a,c\})$ is actually homeomorphic to $M_{0,5}$. Afterward, we note how to modify the proof of Lemma 21 of \cite{bering2024finite}, which is the only place where Lemma 20 is used in that paper.

\begin{lemma}[{\cite[Lemma 20]{bering2024finite}}]
    Suppose $X\subseteq\sph(M_{n,s})$ is a subcomplex and $P\subseteq X^{(0)}$ is a pants decomposition that is $X-$split at $a$ by $b\in X^{(0)}$. For every sphere $c\in P$ that is adjacent to $a$ so that the component of $M_{n,s}\setminus (P\setminus\{a,c\})$ containing $a$ and $c$ is homeomorphic to $M_{0,5}$, there are spheres $d$ and $e$ such that $(d,e)$ is a split pair for $c$.
\end{lemma}

We now discuss Lemma 21 of \cite{bering2024finite} and note the small difference that has to be made in the proof.
\begin{lemma}[{\cite[Lemma 21]{bering2024finite}}]
    Suppose $X\subseteq\sph(M_{n,s})$ is a finite geometrically rigid set and $P\subseteq X^{(0)}$ a pants decomposition is $X-$split. Then there is a finite geometrically rigid set $X^P\supseteq X$ so that $P$ is fully $X^P-$split, that is, $X^P$ contains every split sphere for $P$.
\end{lemma}
\begin{proof}
    We inductively define the sets $P_i\subset P$ as follows. Suppose $a_0\in P$ is $X$-split. Define $P_0=\{a_0\}$, and for $i\geq 1$ let
    $$
  P_i = \left\{ s\in P\ \middle\vert \begin{array}{l}
    s \text{ is adjacent to } a\in P_{i-1} \text{ and the component of } M_{n,s}\setminus (P\setminus \{a,s\}) \\
    \text{containing } a,s \text{ is homeomorphic to } M_{0,5}
  \end{array}\right\}.
$$
These sets differ from the sets denoted $P_i$ in the proof of Lemma 21 of \cite{bering2024finite}, as there it is not assumed that the component of $M_{n,s}\setminus (P\setminus \{a,s\})$ is homeomorphic to $M_{0,5}$ (as they implicitly assume this). Even so, just as in their proof, there is a $k$ so that $\cup_{i=1}^k P_i$ contains all the spheres in $P$ which have a split sphere (i.e., are not self-adjacent). To see this, consider the dual graph to $P$, which has a vertex for every component of $M_{n,s}\setminus P$, and an edge connecting two vertices if they share a common component of $P$. This graph is connected, and the subgraph spanned by the edges whose corresponding spheres are not self-adjacent is connected (as such spheres correspond to the non-loop edges of the graph).
\par 
Then, starting from the edge corresponding to $a_0$, one can reach any other sphere by only going along edges corresponding to spheres $a$ and $s$ so that the component of $M_{n,s}\setminus (P\setminus \{a,s\})$ containing $a$ and $s$ is homeomorphic to $M_{0,5}$. To see this, note that every edge which is not a loop contains a vertex of valence $2$, or a vertex of valence $3$ that is not incident to a loop (potentially one of both). If a vertex $v$ is valence $2$ and $s_1$ and $s_2$ are the two spheres in $P$ that correspond to the edges containing $v$, then it is easy to see that the component of $M_{n,s}\setminus (P\setminus\{s_1,s_2\})$ containing $s_1$ and $s_2$ is homeomorphic to $M_{0,5}$. Thus if $s_1\in P_{i-1}$, then $s_2\in P_i$. If $v$ has valence $3$ and is incident to no loops, with edges corresponding to spheres $s_1, s_2$, and $s_3$, then one can see that one of these spheres, say $s_3$, is such that, for $i=1,2$, the component of $M_{n,s}\setminus(P\setminus\{s_i,s_3\})$ containing $s_i$ and $s_3$ is homeomorphic to $M_{0,5}$. In this case, one needs to explicitly use the fact that $M_{2,0}$ has been excluded, see Case 5 in \Cref{fig:pants-exhaustion} below. Thus if $s_1$ is in $P_{i-1}$, then $s_3\in P_i$ and $s_2\in P_{i+1}$, and similarly if $s_2\in P_{i-1}$. If $s_3\in P_{i-1}$, both $s_1$ and $s_2$ are in $P_i$.
\par 
\Cref{fig:pants-exhaustion} depicts all the cases for $v$ and helps illustrate the above argument. The important cases for the choice of vertex $v$ are Case 4, which illustrates the valence $2$ case, and Cases 10 and 11, which illustrate the valence $3$ case. 
\par 
The proof then proceeds identically as in \cite{bering2024finite}, as the connectivity argument above shows that every non-self-adjacent sphere in $P$ will eventually be contained in some $P_i$.
\end{proof}

\begin{figure}
    \centering
    \includegraphics[width=\textwidth]{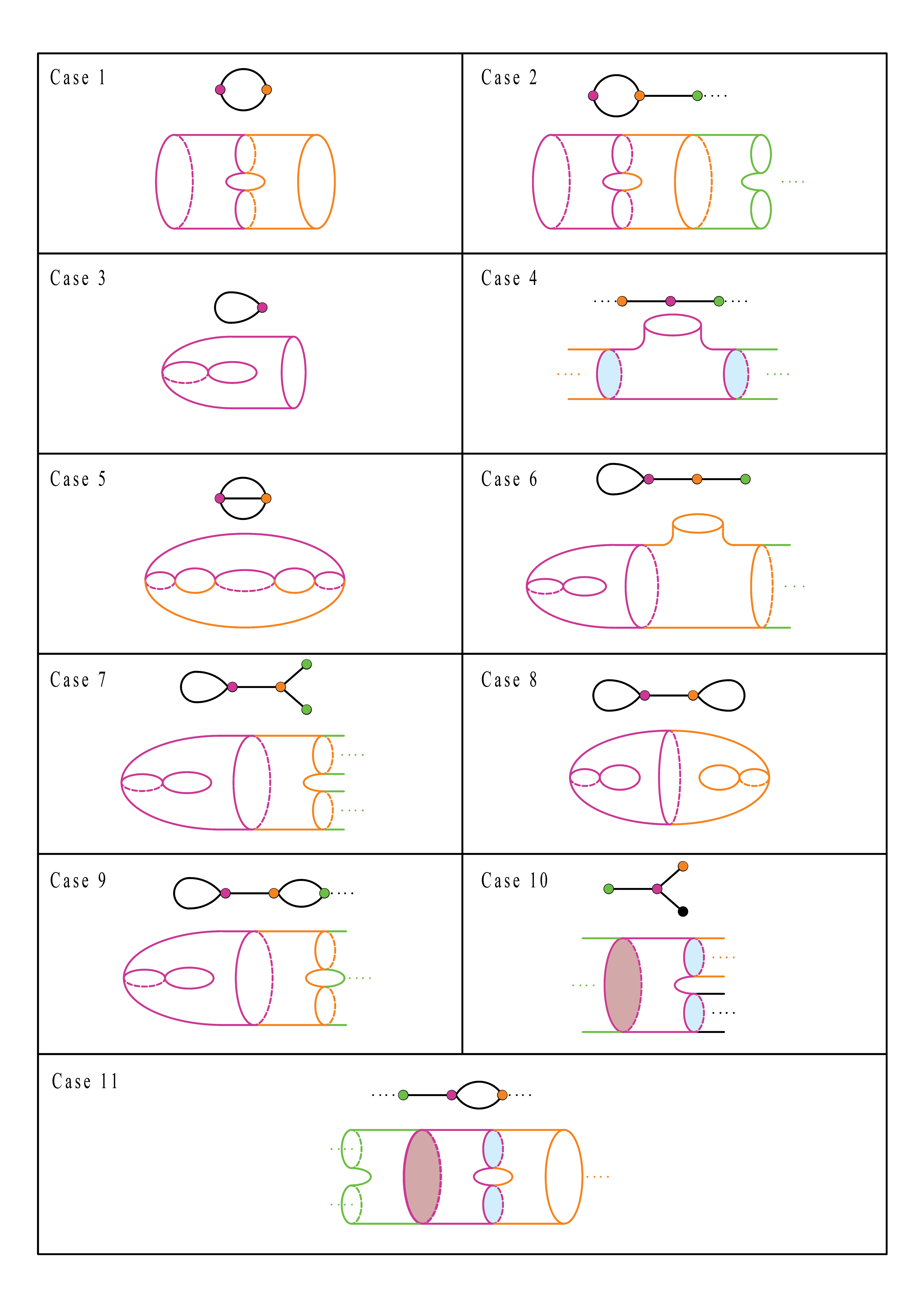}
    \caption{In all the cases, the vertex $v$ is the pink vertex. In cases 1-4, $v$ has valence 2, and the blue spheres denote $s_1$ and $s_2$. In cases 5-11, $v$ has valence three, the blue spheres denote $s_1$ and $s_2$ while the red sphere denotes $s_3$. The cases without these spheres colored in don't apply, either because the pink pair of pants has a self-adjacent sphere (Cases 3, 6, 7, 8 and 9)  or because we may assume the complexity is high enough (Cases 1, 3, 5 and 8).}
    \label{fig:pants-exhaustion}
\end{figure}

We now note the main proposition, which is the key ingredient for proving \Cref{thm:geom_rigidity}. The proof is nearly identical, but we will modify the start of the argument slightly so that it makes sense in the generality we are working in. 
\begin{proposition}[{\cite[Proposition 22]{bering2024finite}}]\label{prop:GeoRigidExhaustion}
    There exists a nested family of finite geometrically rigid sets $X_j\subseteq\sph(M_{n,s})$ such that
    $$
    \sph(M_{n,s})=\bigcup_jX_j
    $$
\end{proposition}
\begin{proof}
    Let $X$ be the strongly rigid set constructed before \Cref{prop:RigidityOnY}. By construction, $X$ contains a pants decomposition $P_0$ which is $X$-split.
    \par 
    We define a sequence of collections of pants decompositions of $M_{n,s}$ as follows. Begin with $\mathcal{P}_0=\{P_0\}$, and inductively define
    $$\mathcal{P}_i=\{P \text{ a pants decomposition } | \text{ there is a } P' \in \mathcal{P}_{i-1} \text{ such that } |P \Delta P'|=2\}.$$
    For every $P\in \mathcal{P}_i$, there is a $P'\in \mathcal{P}_{i-1}$ so that $P$ is obtained from $P'$ by exchanging split spheres. By \Cref{prop:FlipMoveConnectivity}, every pants decomposition can be reached from $P_0$ by applying these exchanges. Thus, every pants decomposition of $M_{n,s}$ is contained in some $\mathcal{P}_i$. 
    The proof then proceeds in exactly the same way as in Proposition 22 of \cite{bering2024finite}.
\end{proof}

To finish the proof of \Cref{thm:geom_rigidity}, it suffices to show that each $X_j$ in \Cref{prop:GeoRigidExhaustion} is strongly rigid. To do this, we summarize the corresponding argument in \cite{bering2024finite}. The proofs of the next lemma and its corollary follow exactly as in \cite{bering2024finite}.

\break 

\begin{lemma}[{\cite[Lemma 23]{bering2024finite}}]
    Suppose $h\in \Map(M_{n,s})$ is the point-wise stabilizer of a pants decomposition $P$ so that every sphere is the boundary of two distinct complementary components of $P$. Then $h$ induces the identity map on $\sph(M_{n,s})$.
\end{lemma}

\begin{corollary}[{\cite[Corollary 24]{bering2024finite}}]\label{cor:uniqueGeoRigid}
    If $X\subset \sph(M_{n,s})$ is geometrically rigid and contains a pants decomposition $P$ so that every sphere is the boundary of two distinct complementary components of $P$, then it is uniquely geometrically rigid in the sense that any mapping class fixing $X$ pointwise induces the identity on $\sph(M_{n,s})$.
\end{corollary}

\begin{proof}[Proof of \Cref{thm:geom_rigidity}]
The proof of Theorem 25 of \cite{bering2024finite} generalizes to imply \ref{thm:IvanovRigidity0}(ii) a graph of rank $n$ with $s$ ends, where $2n+s\geq 6$ (using \Cref{thm:GraphMCGAction} in place of Laudenbach's result in \cite{laudenbach_sur_1973}).  
\par 
In particular, if $X_j$ is a set as in \Cref{prop:GeoRigidExhaustion} and $f: X_j \to \sph(M_{n,s})$ is a simplicial locally injective map, then by \Cref{prop:GeoRigidExhaustion} there is a diffeomorphism $h$ of $M_{n,s}$ inducing $f$. By \Cref{cor:uniqueGeoRigid}, as each $X_j$ contains a pants decomposition as required, any other diffeomorphism inducing $f$ agrees with $h$ on $\sph(M_{n,s})$. In particular, $h_*$ is an automorphism of $\sph(M_{n,s})$ that agrees with $f$ on $X_j$. It is the only such automorphism since there is some diffeomorphism inducing every automorphism, and we just showed that any other diffeomorphism inducing $f$ agrees with $h$ on $\sph(M_{n,s})$. Thus, $X_j$ is strongly rigid.
\end{proof}

\section{Another proof of \Cref{thm:IvanovRigidity0}}\label{sec:mainthm} 
 In this section, we utilize the results in \Cref{sec:finite_rigidity} along with an argument analogous to one that appears in \cite{bavard2020isomorphisms} to give another proof of \Cref{thm:IvanovRigidity0}.
\par
We first show that the sphere complex uniquely determines finite-rank
doubled handlebodies. This result follows from those in \Cref{sec:geom_rigid}, but we give another proof of it here to show the proof of \Cref{thm:IvanovRigidity0} that appears here is independent of the proofs in \Cref{sec:geom_rigid}.

\begin{proposition}\label{prop:FiniteTypeIso}
    Suppose $\phi:\sph(M_{n,s})\to \sph(M_{m,r})$ is an isomorphism between the sphere graphs of two handlebodies of finite type. Then $n=m$ and $s=r$.
\end{proposition}
\begin{proof}
    Note that the dimension of the largest simplex of the two sphere complexes must be the same, which puts a dimensional restriction on when the two graphs can be isomorphic. Thus, if $s, r\in \{0,1,2 \}$, the result follows, as the maximal simplex dimension of $\sph(M_{n,s})$ is $3n+s-3$, and similarly it is $3m+r-3$ for $\sph(M_{m,r})$. For these to be equal, $s-r$ must be a multiple of $3$, which means $s=r$ in this case. Thus, $n=m$ as well.
    \par 
    Now suppose we have an isomorphism when $m>n$ and $s\geq 3$. Then it follows that $s>r$. In $\sph(M_{n,s})$ there is a sphere $S$ cutting off a homeomorphic copy of $M_{0, s+1}$. In particular, by \Cref{prop:equivequiv}, as the equivalence classes of the link must be preserved by the isomorphism, $\phi(S)$ must also be a separating sphere, and one of its complementary components is homeomorphic to $M_{0, s+1}$. This is because the only doubled handlebodies with boundary that have a finite sphere graph are $M_{1,1}$ and $M_{0, p}$, and the only time when they have the same number of spheres as $M_{0,s+1}$ is when $p=s+1$ ($\sph(M_{1,1})$ only has a single sphere). But this is impossible as $s>r$, so no such spheres exist in $M_{m,r}$. Hence, $m=n$, and thus $s=r$, finishing the proof.
\end{proof}
 
\subsection{Links of sphere systems} 
In analogy with \cite{bavard2020isomorphisms}, we first define an
equivalence relation on $\lk(\sigma)$ for $\sigma$ a sphere system in
$\sph(M_\Gamma)$.
\begin{definition}\label{def:lnequiv} 
  Let $\a,\b \in \lk (\sigma)$.  Then $\a \sim \b$ if and only if 
  there exists $\c \in \lk(\sigma)$ non-adjacent 
  to both $\a, \b$. 
\end{definition}

\begin{remark}
  Since $\lk(\sigma)$ is loop-free, in particular $\a \sim \a$.
  Likewise, if $\a, \b$ are non-adjacent, then $\a$ is non-adjacent to
  both and $\a \sim \b$.
\end{remark}

While defined combinatorially, $\sim$ may be characterized
topologically.
In particular, the following shows that $\sim$ is an equivalence
relation.

\begin{proposition}\label{prop:equivequiv}
  Let $\a, \b \in \lk(\sigma)$.  Then $\a \sim \b$ if and only if $\a, \b$
  lie in the same complementary component of $\sigma$.
\end{proposition}

\begin{proof}
To prove the forward direction, note that if $\c \in \lk(\sigma)$
is not adjacent to $\a,\b$, then $\a \cup \b \cup \c$ is connected
and disjoint from $\sigma$, hence is contained in a single connected
component.

Conversely, suppose that $\a, \b \subset M$ for $M \in \pi_0(M_\Gamma
\setminus \sigma)$.
We assume $\a, \b$
are disjoint and distinct, else $\a\sim \b$ by the remark above; by
\Cref{rmk:full_in_link}, $\a,\b$ are essential in $M$.
  Suppose first that $M$ has at least two distinct punctures
  $x,y$.  By \Cref{lem:int_pop}, it suffices to find a simple arc
  $\gamma \subset M$ between $x,y$ that intersects $\a, \b$
  essentially: the sphere $\c$ bounding a regular neighborhood of
  $\gamma$ intersects $\a, \b$ essentially in $M$, thus by
  \Cref{prop:full_subcpx} essentially in $M_\Gamma$. 
  Thus $\c$ is essential in
  $M_\Gamma$ and disjoint from $\sigma$, hence $\c \in \lk(\sigma)$
  and $\a\sim \b$.  We consider four cases:

  \begin{enumerate}[label=\textit{(\roman*)}]
    \item \textit{$\a \cup \b$ is non-separating.}  We may choose
      $\gamma$ to intersect $\a, \b$ each exactly once.
    \item \textit{$\a \cup \b$ is separating but $\a, \b$ are
      non-separating.} We may choose $\gamma$ intersecting $\a, \b$
      each exactly once if $x,y$ are not separated by $\a \cup \b$;
      else, choose $\gamma$ to intersect $\b$ once and $\a$ with
      signed intersection $\pm 2$. 
    \item \textit{Only $\a$ is separating.} If $\a$ separates some
      punctures of $M$, then replace $x,y$ such that they are
      separated by $\a$ and fix $\gamma$ to intersect $\b$ exactly
      once.  Else, the complementary component $M'$ 
      of $\a$ not containing
      $x,y$ must have genus; if $\b \not \subset M'$, 
      fix a non-separating sphere $d \subset
      M'$ and choose $\gamma$ to intersect $\b$ and $d$ (if defined) 
      each exactly once.
    \item \textit{$\a, \b$ are separating.}   Let $U,V,W$ be
      components of $M \setminus (\a \cup \b)$, where $\a$ separates
      $U$ and $V \cup W$ and $\b$ separates $U \cup V$ and $W$. 
      Let $P$ denote the set of punctures of $M$.
      If $U$ has no punctures in $P$, then it has genus, hence fix a
      non-separating sphere $X \subset U$, and likewise if $W$
      has no punctures in $P$ fix a non-separating sphere $Y
      \subset W$.  Replace $x \in P$ to be a puncture in $U$, if one
      exists, and $y \in P$ to be a puncture in $W$, if one exists,
      or else some puncture distinct from $x$.
      If defined, let $\gamma$ intersect $X, Y$ each exactly once.  
  \end{enumerate}

  $M$ is obtained by removing $\sigma$ from $M_\Gamma$, hence
  must have at least one puncture.  If $M$ has exactly one puncture,
then we may assume $M$ has genus at least $2$; otherwise, $M$ has at most
one distinct essential sphere, and the statement follows.  Hence
there exists an essential sphere $q$ distinct from $\a,\b$
and non-separating in $M$; we note that $M' = M \setminus q$ has
three punctures.  Replacing $\sigma$ with $\sigma' = \sigma
\cup \{q\}$ and $M$ with $M'$, the argument above
obtains $\c \in \lk(\sigma') \subset \lk(\sigma)$ non-adjacent
to $\a,\b$.
\end{proof}

Let $\lk(\sigma)|_{[\a]}$ denote the full subcomplex of
$\lk(\sigma)$
(or equivalently of 
$\sph(M_\Gamma)$, since it is flag) 
induced by the equivalence class of $\a$ in
$\lk(\sigma)$. By \Cref{prop:equivequiv}, $[\a]$ is exactly 
the spheres in $\lk(\sigma)$ in the same component $M \subset
M_\Gamma \setminus \sigma$ as
$\a$.  From \Cref{prop:full_subcpx} and \Cref{rmk:full_in_link}
we then have the following:

\begin{corollary}\label{cor:link_equiv_char}
  Let $\sigma \subset \sph(M_\Gamma)$ and $\a \in \lk(\sigma).$  Let $M$
  be the component of $M_\Gamma \setminus \sigma$ containing $\a$.  
  Then
  the equivalence class $[\a]$ is the set of all essential embedded
  2-spheres in $M$, and $\lk(\sigma)|_{[\a]} \cong \sph(M)$. \qed
\end{corollary}

\subsection{Diffeomorphisms from automorphisms}
The following theorem will be the main result to proving
\Cref{thm:IvanovRigidity0} in the general case. It is analogous to
the work in \Cref{sec:geom_rigid} to construct a diffeomorphism
inducing an isomorphism between sphere graphs, but the method in
which we build the diffeomorphism differs from what is in
\Cref{sec:geom_rigid}.  The proof below is inspired by the proof of
\cite[Theorem 1.3]{bavard2020isomorphisms}, and for the most part,
follows it very closely. 

\begin{thm}\label{thm:IvanovRigidity}
    Let $\Gamma$ and $\Gamma'$ be two locally finite, infinite graphs, with associated 3-manifolds $M_\Gamma$ and $M_{\Gamma'}$. Every simplicial isomorphism of the sphere complexes $\sph(M_\Gamma) \to \sph(M_\Gamma')$ is induced by a homeomorphism $M_\Gamma \to M_{\Gamma'}$.  
\end{thm}

\begin{proof}

Fix a simplicial isomorphism $\Psi \colon \sph(M_\Gamma) \to \sph(M_\Gamma')$. Fix a compact exhaustion $K_1 \subset K_2 \subset \dots$ of $M_{\G}$ such that all components $M_{\Gamma} \setminus K_i$ are infinite-type and $K_i$ is homeomorphic to $M_{n_i, s_i}$ for $2n_i+s_i\geq 6$. By enlarging each $K_i$ if necessary, we may assume that the boundary spheres $\a_i$ of $K_i$ are essential embedded spheres of $K_{i + 1}$.
\par 
Let $E_i$ denote the equivalence class of $\lk(\a_i)$ that contains essential embedded 2-spheres in $K_i$.  The equivalence class $E_i$ is distinguishable combinatorially from the other equivalence classes of $\lk(\a_i)$ since it is the unique class with a finite clique number.  This is because $K_i$ is finite-type and, therefore, the spheres in a pants decomposition of $Y_i$ for a finite maximal clique.  Since connected components of $M_\Gamma \setminus Y_i$ are of infinite type,  their associated equivalence classes in $\lk(\a_i)$ have unbounded clique numbers. 
\par 
In particular, as the restriction of $\Psi$ from $\lk(\a_i)$ to $\lk(\Psi(\a_i))$ sends equivalence classes to equivalence classes (as the classes are defined combinatorially), it follows from the discussion in the previous paragraph that:
\begin{itemize}
    \item $\lk(\Psi(\a_i))$ contains exactly one equivalence class $E_i'$ which corresponds to a compact doubled handlebody $K_i' \subset M_{\Gamma'}$ whose boundary is $\Psi(\a_i)$, and 
    \item $\Psi$ restricts to an isomorphism $\Psi_i: \sph(K_i)\to \sph(K_i')$.
\end{itemize}

By \Cref{prop:FiniteTypeIso}, $K_i \cong K_i'$. Recall that the proof of \Cref{thm:geom_rigidity} showed that every element of $\Aut(\sph(M_{n_i, s_i}))$ is induced by a diffeomorphism. It follows that $\Psi_i$ is induced by a diffeomorphism $\phi_i: K_i \to K_i'$. The ambiguity on the choice of $\phi_i$ is only up to a product of sphere twists. By possibly modifying each $\phi_i$ by sphere twists, we may thus assume that these diffeomorphisms are compatible in the sense that $\phi_{i+1}(K_i)=K_i'$ and the restriction of $\phi_{i+1}$ to $K_i$ agrees with $\phi_i$. In particular, the direct limit of this sequence of diffeomorphisms induces a diffeomorphism $\phi: M_{\G} \to M_{\G'}$ which, by construction, induces the isomorphism $\Psi$.
\end{proof}

To finish the proof of \Cref{thm:IvanovRigidity0}, one can follow the proof in \Cref{sec:geom_rigid}, utilizing \Cref{thm:IvanovRigidity} in place of \Cref{lem:compatibility_sigma} and \Cref{cor:ManifoldMCGActionSurjective}.

\section{Locally finite strongly rigid
sets}\label{sec:loc_strongly_rigid}

Let $M_\Gamma$ be the doubled handlebody associated to a 
locally
finite graph $\Gamma$.  We would like to extend to
$\sph(M_\Gamma)$ the results of \Cref{sec:finite_rigidity} and 
construct an exhaustion by (the appropriate generalization of)
finite strongly rigid sets. 

If $\Gamma$ is infinite-type, then
$M_\Gamma$ does not admit a finite strongly rigid set. Indeed, a finite set $X_0 \subset \sph(M_\Gamma)^{(0)}$ has compact union 
$K= 
\bigcup_{\a\in X_0}\a \subset M_\Gamma$.
There exists a diffeomorphism $h : M_\Gamma \to M_\Gamma$
acting non-trivially on 
$\sph(M_\Gamma)$ but supported disjointly from $K$, 
hence fixing $X_0$. The identity automorphism and 
$h_* \neq \mathrm{id}$ both restrict to the inclusion map of $X_0$
into $\sph(M_{\G})$, hence $X_0
\hookrightarrow \sph(M_\Gamma)$ does not extend to a unique
automorphism of $\sph(M_{\G})$.  We instead
consider the rigidity of subcomplexes that satisfy 
a local version of finiteness.

\begin{definition}\label{def:loc_finite}
  A subcomplex $X \subset \sph(M_\Gamma)$ is 
\emph{topologically locally finite} if every
compact $K \subset M_\Gamma$ essentially intersects finitely many
components of $X^{(0)}$. 
\end{definition}

In fact, for all $\Gamma$ of infinite-type there exist
\textit{no} strongly 
rigid sets in $\sph(M_{\Gamma})$, as we describe in
\Cref{sec:no_rigid}. 
Nonetheless, we will construct subcomplexes 
which always admit unique isomorphism extensions of
locally injective maps, provided those maps also preserve maximal sphere systems. 

\begin{definition}\label{def:maximal}
  A simplicial map $f : X \subset \sph(M_\Gamma) \to
  \sph(M_{\Gamma'})$ is \emph{maximal} if for any sphere
  system $\sigma \subset X$ that is 
  maximal in $M_\Gamma$, $f\sigma$ is maximal in $M_{\Gamma'}$.
\end{definition}

\noindent
This condition is necessary in the following sense: 
if $\sigma \subset \sph(M_\Gamma)$
is a maximal sphere system $\sigma$ and a map
$f : \sigma \to \sph(M_\Gamma)$ extends to an automorphism, then
$f\sigma$ must also be maximal. 
If $X \subset \sph(M_\Gamma)$ 
is rigid, then any locally injective map $X \to
\sph(M_\Gamma)$ extends to an automorphism, hence is maximal.  

A subcomplex $X \subset \sph(M_\Gamma)$ for which any maximal
locally injective map $f : X \to \sph(M_\Gamma)$ extends to a unique
automorphism is \emph{strongly rigid over maximal maps}.
For brevity, by \emph{locally finite strongly rigid set}, we will
mean a topologically locally finite subcomplex that is strongly
rigid over maximal maps.
Note that when $\Gamma$ has infinite type, 
such subcomplexes are not locally finite in the usual
simplicial sense. In \Cref{sec:loc_finite_infinite}, we prove the
following, which implies \Cref{thm:sr_exhaustion}:

\begin{theorem}\label{thm:loc_finite_rigidity}
  Suppose that $\Gamma$ is connected and 
  $\dim(\sph(M_{\Gamma})) \geq 4$, and in particular if
  $\Gamma$ is infinite-type.  
  Then $\sph(M_\Gamma)$ is covered by 
  nested locally finite strongly rigid sets $Z_i$. 
\end{theorem}

\subsection{Constructing locally finite strongly rigid
sets}\label{sec:loc_finite_infinite}

We will need the following simplicial non-embeddings: 

\begin{lemma}\label{lem:non_embed} \phantom{.}
  \begin{enumerate}[label=\textit{(\roman*)}]  
  \item \textit{$\sph(M_{0,4})*\sph(M_{0,4})$ does not embed in 
    $\sph(M_{0,5})$ and vice versa.}
  \item \textit{$\sph(M_{0,4})*\sph(M_{0,4})$ and $\sph(M_{0,5})$
       do not embed in $\sph(M_{1,2})$.}
  \item \textit{$\sph(M_{0,6})$ does not embed in $\sph(M_{1,3})$}.
\end{enumerate}
\end{lemma}
\begin{proof}\phantom{.}
  \begin{enumerate}[label=\textit{(\roman*)}]  
  \item Suppose $\sph(M_{0,4})*\sph(M_{0,4})\cong K_{3,3}$ can be
    embedded via a map $f$ into $\sph(M_{0,5})$, which is isomorphic
    to the Petersen graph (see \Cref{fig:JoinNotEmbed}). 
    Consider the vertex
    $v=(1,2)$ in $K_{3,3}$. Without loss in generality, we may
    assume $f(v)=(1,4)$. Since both $v$ and $f(v)$ have valence 3,
    we know that $$\{(1',2'), (1',3'),
    (2',3')\}\xmapsto{f}\{(3,5),(3,2),(2,5)\}$$ Again, without loss
    in generality, we may assume that $f(1',2')=(3,5)$. Since
    $(1,3)$ and $(2,3)$ are adjacent to $(1',2')$ in $K_{3,3}$,
  \[
  \{(1,3),(2,3)\}\xmapsto{f}\{(2,4),(1,2)\}
\]
This is not possible since neither $(2,4)$ nor $(1,2)$ are adjacent
to $(3,2)$ in the Petersen graph.

  On the other hand, $\sph(M_{0,5})$ has ${5 \choose 2}=10$
  vertices, while $\sph(M_{0,4})*\sph(M_{0,4})$ has only $6$
  vertices, so $\sph(M_{0,5})$ cannot embed into
  $\sph(M_{0,4})*\sph(M_{0,4})$.

  \item The graph $\sph(M_{1,2})$ a tree (see
    \Cref{fig:M12SphereGraph}). As seen in \Cref{fig:JoinNotEmbed},
    both  $\sph(M_{0,4})*\sph(M_{0,4})$ and $\sph(M_{0,5})$ admit
    embedded loops, and thus cannot be embedded into a tree.  

  \item
    The complex $\sph(M_{1,3})$
    is contractible
\cite{hatcher_homology_2004}. On the other hand, $\sph(M_{0,6})$ is
a two-dimensional simplicial complex with 25 vertices, 105 edges,
and 105 faces, which means its Euler characteristic is 25. Since the
Euler characteristic is greater than two, $H_2(\sph(M_{0,6}))\neq
0$.  Suppose that $\sph(M_{0,6})$ embeds into $\sph(M_{1,3})$.  Then
the pair $(\sph(M_{1,3}),\sph(M_{0,6}))$ gives the exact sequence
\[
H_3(\sph(M_{1,3}),\sph(M_{0,6})) \to H_2(\sph(M_{0,6})) \to
H_2(\sph(M_{1,3})) \:.
\]
Since $\sph(M_{1,3})$ is $2$-dimensional the first group vanishes
and $H_2(\sph(M_{0,6})) \hookrightarrow H_2(\sph(M_{1,3})) = 0$, a
contradiction.\qedhere
\end{enumerate}
\end{proof}

\begin{figure}[ht!]
        \centering
        \begin{overpic}[width=12cm]{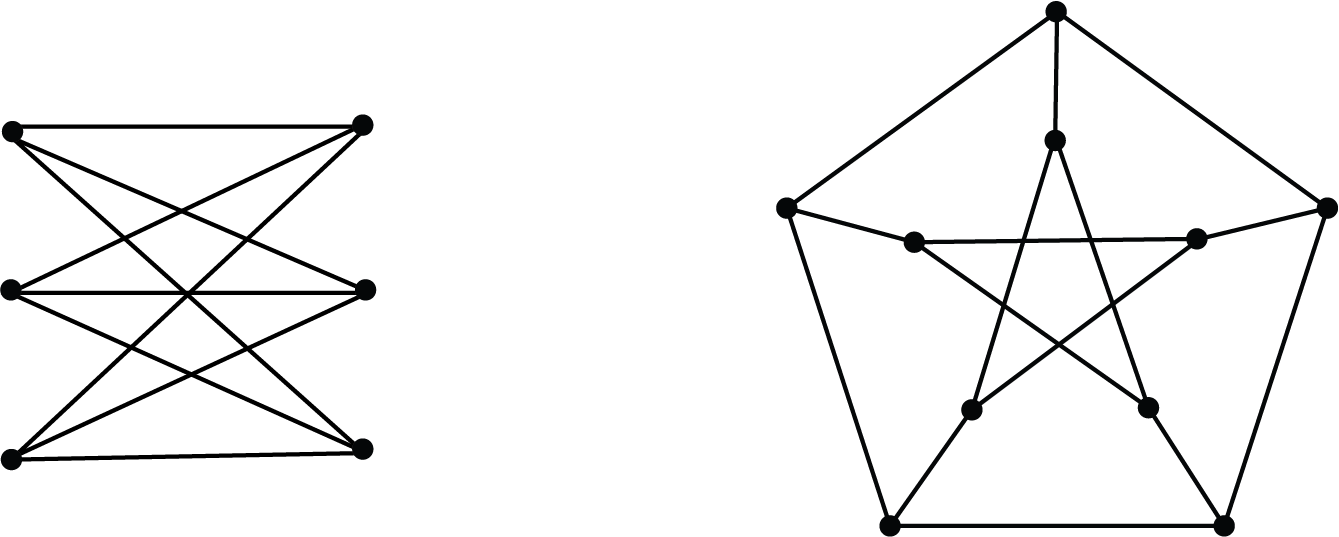}
            \put(-7,30){$(1,2)$}
            \put(-7,18){$(1,3)$}
            \put(-7,5){$(2,3)$}
            \put(28,30){$(1',2')$}
            \put(28,18){$(1',3')$}
            \put(28,5){$(2',3')$}
            \put(51,25){$(3,5)$}
            \put(100,25){$(3,4)$}
            \put(93,0){$(2,5)$}
            \put(59,0){$(1,4)$}
            \put(76,41){$(1,2)$}
            \put(80,29){$(4,5)$}
            \put(65,24){$(2,4)$}
            \put(66,11){$(3,2)$}
            \put(87,10){$(1,3)$}
            \put(88,18.5){$(1,5)$}
        \end{overpic}
        \caption{On the left is $\sph(M_{0,4}) \ast \sph(M_{0,4}) \cong K_{3,3}$.  With the boundary components of each $\partial M_{0,4}$ labeled 1--4 and $1'$--$4'$, the pair $(i,j)$ determines the sphere by \Cref{lem:GenusZeroSphere}.  Similarly, for $\sph(M_{0,5})$, which is pictured on the right and is isomorphic to the Petersen graph.}
        \label{fig:JoinNotEmbed}
    \end{figure}
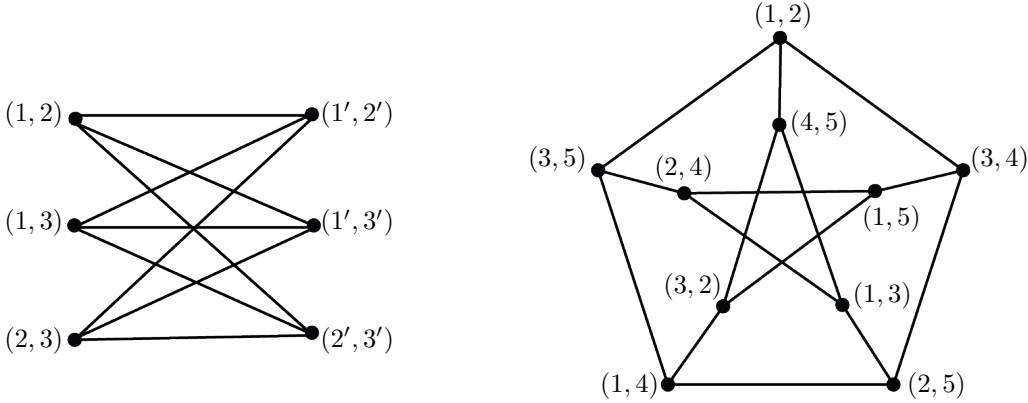
Recall that $X_\eta \coloneq \bigcup_{\a \in \eta}\lk(\eta \setminus \a)$ for $\eta \subset
\sph(M_\Gamma)$ a maximal sphere system.

\begin{lemma}\label{lem:local_diffeo}
  Let $\Gamma$ be connected and let
  $\eta \subset \sph(M_\Gamma)$ be a maximal sphere system such
  that $|\eta| \geq 4$. Suppose $Y_\eta \subset \sph(M_\Gamma)$
  is a full subcomplex such that
  \begin{enumerate}
    \item $X_\eta \subset Y_\eta$, 
    \item 
      for $\a \neq \b \subset \eta$, $\lk(\eta \setminus
      \{\a,\b\})\subset Y_\eta$, unless $e_\a \cup e_\b$
      has non-zero rank in $\Delta_\eta$, in which case $Y_\eta$
      contains a subset of $\lk(\eta \setminus \{\a,\b\})$ 
      of size at least $11$, and
    \item for $\a,\b,\c \subset \eta$, if
      $e_\a\cup e_\b\cup e_\c \cong K_{1,3}$ then $\lk(\eta
      \setminus \{\a,\b,\c\}) \subset Y_\eta$; if
      $e_\a \cup e_\b\cup e_\c \cong K_3$ then $Y_\eta$
      contains a subset of $\lk(\eta
      \setminus \{\a,\b,\c\})$ of size at least $26$.
  \end{enumerate}
  If $f : Y_\eta \to \sph(M_{\Gamma'})$ is a locally injective
  simplicial map for which $f\eta$ is maximal, then $M_\Gamma \cong
  M_{\Gamma'}$.
\end{lemma}

\begin{proof}
 It suffices to find a maximal sphere system $\eta' \subset
  \sph(M_{\Gamma'})$ such that $\Delta_\eta \cong \Delta_{\eta'}$.
   Since $|\eta| \geq 4$, $f$ is injective on the sets 
  $Y_\eta \cap \lk(\eta \setminus \a) = \lk(\eta \setminus
  \a)$, 
  $Y_\eta \cap \lk(\eta \setminus \{\a,\b\})$, and $Y_\eta
  \cap \lk(\eta \setminus \{\a,\b,\c\})$, as each of these sets is contained in the star of some element of $\eta$.

  First, assume that $\Delta_\eta$ is loop-free and $a\in\eta$.  If $e_{f\a}$ is a loop in $\Delta_{f\eta}$ then $\lk(\eta \setminus \a) \cong \sph(M_{0,4})$ does not embed into $\lk(f\eta
  \setminus f\a) \cong \sph(M_{1,1})$.  Thus, $\Delta_{f\eta}$ is also loop-free. 
  Now let $\a,\b \in \eta.$  The edges $e_\a$ and $e_\b$ are
  disjoint in $\Delta_\eta$ if and only if $\lk(\eta \setminus
  \{\a,\b\})$ is isomorphic to $\sph(M_{0,4})*\sph(M_{0,4})$ and
  $e_\a$ and $e_\b$ are incident on one common vertex or
  two common vertices if and only if $\lk(\eta \setminus \{ a, b\})$
  is isomorphic to $\sph(M_{0,5})$ or $\sph(M_{1,2})$, respectively.
  An analogous statement holds for $e_{f\a}$ and $e_{f\b}$. 

  The map $f$ restricts to an embedding $Y_\eta \cap \lk(\eta \setminus
  \{\a,\b\}) \hookrightarrow \lk(f\eta \setminus
  \{f\a,f\b\})$.  
  There are three possibilities for the domain of this embedding:
  $Y_\eta \cap \sph(M_{0,4})*\sph(M_{0,4})
  =\sph(M_{0,4})*\sph(M_{0,4}) $, $Y_\eta \cap \sph(M_{0,5}) =
  \sph(M_{0,5})$, or $Y_\eta \cap \sph(M_{1,2})$. We observe the 
  following embeddings are not possible: 

  \bigskip

\adjustbox{width=\textwidth}{
$\begin{tikzcd}[cramped,sep=tiny]
	& {\mathcal{S}(M_{0,5})} && {\mathcal{S}(M_{0,4}) \ast \mathcal{S}(M_{0,4}) } && {\mathcal{S}(M_{0,4}) \ast \mathcal{S}(M_{0,4})} \\
	{\mathcal{S}(M_{0,4}) \ast \mathcal{S}(M_{0,4}) } & {(1)} & {\mathcal{S}(M_{0,5})} & {(2)} & {Y_\eta \cap \mathcal{S}(M_{1,2})} & {(3)} \\
	& {\mathcal{S}(M_{1,2})} && {\mathcal{S}(M_{1,2})} && {\mathcal{S}(M_{0,5})} \\
	&& {}
	\arrow["\shortmid"{marking}, hook, from=2-1, to=1-2]
	\arrow["\shortmid"{marking}, hook, from=2-1, to=3-2]
	\arrow["\shortmid"{marking}, hook, from=2-3, to=1-4]
	\arrow["\shortmid"{marking}, hook, from=2-3, to=3-4]
	\arrow["\shortmid"{marking}, hook, from=2-5, to=1-6]
	\arrow["\shortmid"{marking}, hook, from=2-5, to=3-6]
\end{tikzcd}$
}
where $(1)$ and $(2)$ follow from \Cref{lem:non_embed}, and (3) 
is true since $|Y_\eta\cap\mathcal{S}(M_{1,2})|\geq11$, but
$|\sph(M_{0,4})*\sph(M_{0,4})|,|\sph(M_{0,5})| < 11$. 
  The only remaining possibility is that all links (intersected with
  $Y_\eta$) embed into
  isomorphic copies of themselves. Hence it follows that $\lk(\eta
  \setminus \{\a,\b\}) \cong \lk(f\eta \setminus \{f\a,f\b\})$ and
  $e_\a \cup e_\b \cong e_{f\a} \cup e_{f\b}$, without loss of
  generality preserving order: \textit{i.e.}\ 
  via a graph isomorphism with $e_\a \mapsto e_{f\a}$ and $e_\b
  \mapsto e_{f\b}$.

  Thus, $f|_\eta \colon \Delta_\eta \to \Delta_{f\eta}$ is an edge
  isomorphism. To obtain a graph isomorphism, we need only
  show that $f|_\eta$ does not have a $K_3,K_{1,3}$-pair.  Now,
  $e_\a\cup e_\b\cup e_\c \cong K_{1,3}$
  if and only if $\lk(\eta \setminus \{\a,\b,\c\}) \cong
  \sph(M_{0,6})$ and 
  $e_\a \cup e_\b \cup e_\c \cong K_3$ if and only if
  $\lk(\eta\setminus \{\a,\b,\c\}) \cong \sph(M_{1,3})$ 
 (see \Cref{fig:K13K3})  
 and likewise for $e_{f\a} \cup e_{f\b} \cup e_{f\c}$ as well. 
  The restriction $f : Y_\eta \cap \lk(\eta
  \setminus \{\a,\b,\c\}) 
  \hookrightarrow \lk(f\eta \setminus \{f\a,f\b,f\c\})$, along with
  \Cref{lem:non_embed} and that $|\sph(M_{0,6})^{(0)}| = 25$,
  implies $f|_\eta$ has no $K_3,K_{1,3}$-pairs.  Therefore,
  $\Delta_\eta \cong \Delta_{f\eta}$ by \Cref{thm:whitney_multi} and
  \Cref{cor:whitney_infinite}.

   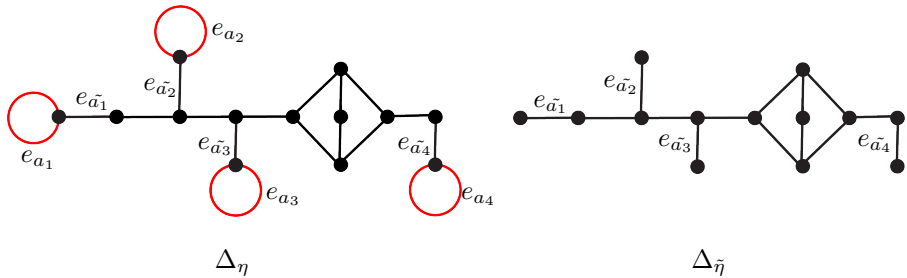
\begin{figure}[ht!]
    \centering
    \begin{tabular}{>{\centering\arraybackslash}m{0.473\textwidth} >{\centering\arraybackslash}m{0.473\textwidth}}
       \centering
       \begin{overpic}[scale=0.5]{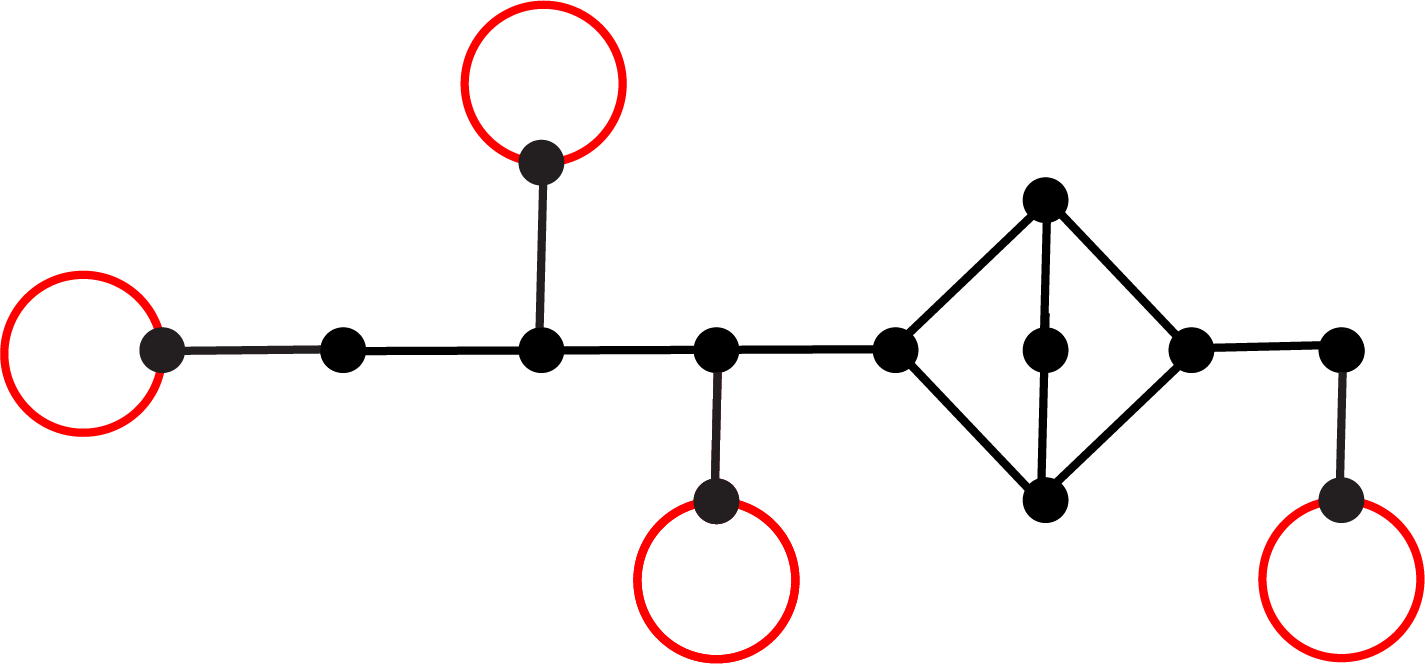}
            \put(3,12){$e_{a_1}$}
            \put(15,25){$e_{\tilde{a_1}}$}
            \put(30,28){$e_{\tilde{a_2}}$}
            \put(45,40){$e_{a_2}$}
            \put(42,15){$e_{\tilde{a_3}}$}
            \put(57,4){$e_{a_3}$}
            \put(100,4){$e_{a_4}$}
            \put(86,15){$e_{\tilde{a_4}}$}
        \end{overpic}
         &
       \begin{overpic}[scale=0.5]{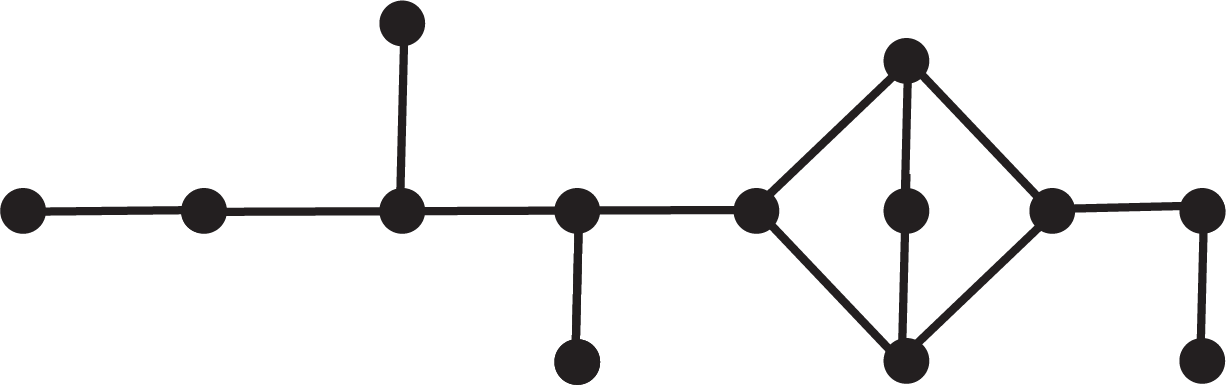}
            \put(5,17){$e_{\tilde{a_1}}$}
            \put(23,22){$e_{\tilde{a_2}}$}
            \put(37,7){$e_{\tilde{a_3}}$}
            \put(88,7){$e_{\tilde{a_4}}$}
        \end{overpic} \\[1.5cm] 
        $\Delta_\eta$ & $\Delta_{\tilde{\eta}}$
    \end{tabular}
    \vspace{5pt}
    \caption{On the left the graph $\Delta_\eta$ has loops $e_a$
    incident to a unique edge $e_{\tilde{a}}$. On the right,
  $\Delta_{\tilde{\eta}}$ is the loop-free subgraph.} 
  \label{fig:Eta}
\end{figure}

  Now consider the case where $\Delta_\eta$ has loops. Since
  $|\eta|\geq 4$, no edge in $\Delta_\eta$ can be adjacent to two
  loops.  Hence, if $a\in\eta$ and $e_a$ is a loop, then we denote
  by $\tilde{a}\in\eta$ the unique sphere such that $e_a$ is
  adjacent to $e_{\tilde{a}}$.  Define
  $\tilde{\eta}=\eta\setminus\eta_0$, where $\eta_0=\{a\in\eta\ |\
  e_a\text{ is a loop}\}$.  The subgraph
  $\Delta_{\tilde{\eta}}\subset\Delta_\eta$ is a graph without loops
  (see \Cref{fig:Eta}) and $\tilde{\eta},\ f\tilde{\eta}$ are
  maximal in $M_\Gamma\setminus\eta_0$ and $M_{\Gamma'}\setminus
  f\eta_0$ respectively.  By the previous case, 
  $\Delta_{\tilde{\eta}}\cong\Delta_{f\tilde{\eta}}$ via an
  isomorphism that induces $f|_{\tilde{\eta}}$ on the edges.
  Focusing on $a \in \eta_0$ (where $e_a$ is a loop) we see that
  $e_{fa}\cup e_{f\tilde{a}}$ is isomorphic to either two disjoint
  edges, a loop disjoint from an edge, two edges incident on one
  common vertex, a bigon or $e_a\cup e_{\tilde{a}}$.  In the first
  three cases, $\lk(f\eta\setminus\{fa,f\tilde{a}\})$ is isomorphic
  to $\sph(M_{0,4})\ast\sph(M_{0,4})$,
  $\sph(M_{1,1})\ast\sph(M_{0,4})$ and $\sph(M_{0,5})$ respectively.
  Since $|Y_\eta\cap\lk(\eta\setminus\{a,\tilde{a}\})|\geq11$, the
  map $f:Y_\eta\cap\lk(\eta\setminus\{a,\tilde{a}\})
  \hookrightarrow\lk(f\eta\setminus\{fa,f\tilde{a}\})$
  cannot be an embedding: $e_{fa}\cup
  e_{f\tilde{a}}$ is isomorphic to a bigon or $e_a\cup
  e_{\tilde{a}}$. 

  Let $\eta_1 \subset \eta_0$ be the subset of spheres 
  $a \in \eta_0$ for which $e_{fa}\cup e_{f\tilde a}$ is a bigon. 
  For each $a \in \eta_1$, we may perform a flip move $f\tilde{a}
  \mapsto \tilde a'$ such that $e_{fa}\cup e_{\tilde{a}'}\cong
  e_a\cup e_{\tilde{a}}$ (see \Cref{fig:Flip}). Vertices in 
  $\Delta_{f\eta}$ have valence at most $3$, hence no two bigons are
  adjacent: we may perform these flips disjointly
   for each $a \in \eta_1$ to obtain a new pants decomposition
   $\eta'$ for $M_{\Gamma'}$ with $\Delta_{\eta'}\cong
   \Delta_\eta$, as desired. 
\end{proof}

\smallskip
\begin{figure}[ht!]
    \centering
    \begin{tabular}{>{\centering\arraybackslash}m{0.25\textwidth} >{\centering\arraybackslash}m{0.02\textwidth} >{\centering\arraybackslash}m{0.25\textwidth} >{\centering\arraybackslash}m{0.03\textwidth} >{\centering\arraybackslash}m{0.25\textwidth}}
       \centering
       \begin{overpic}[scale=0.28]{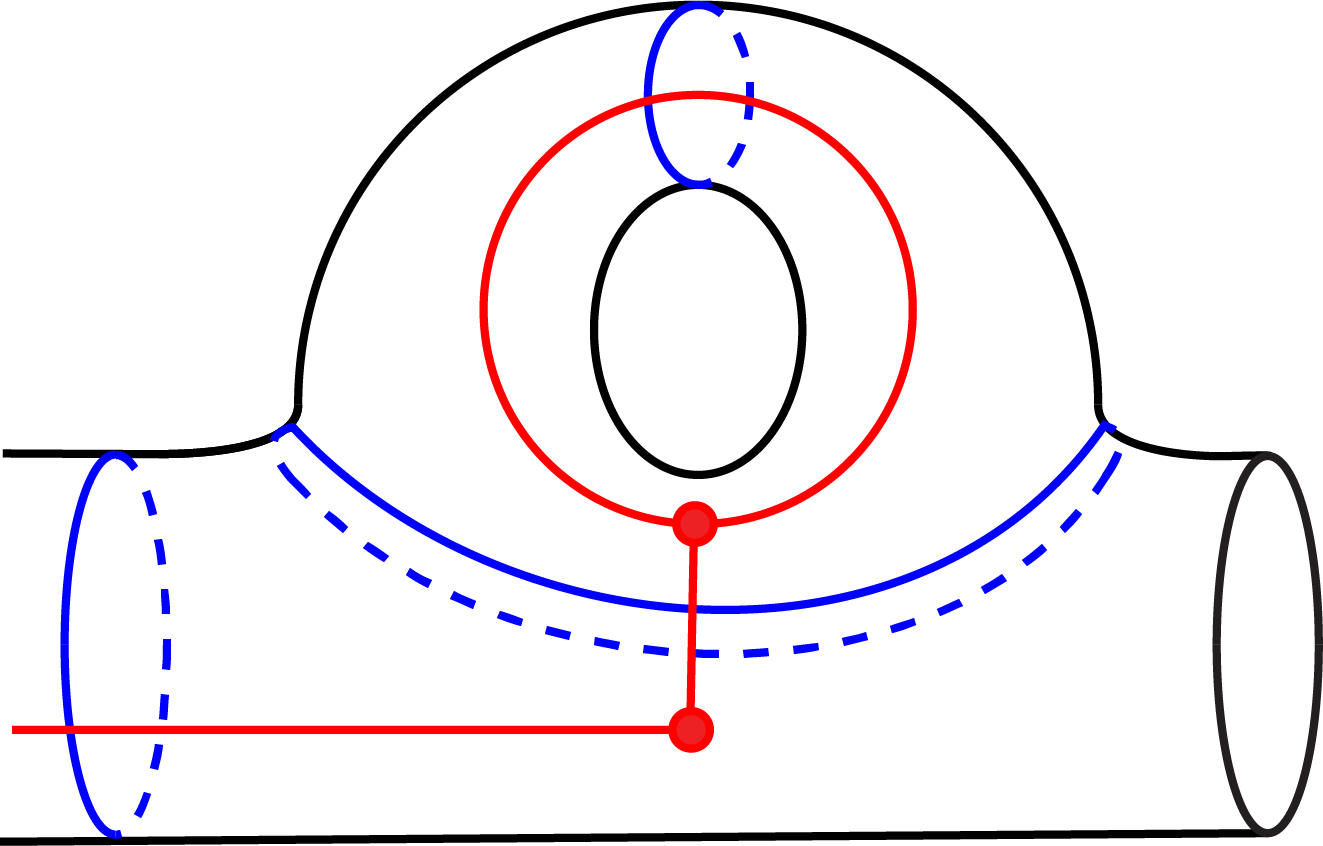}
       \put(50,66){$a$}
       \put(70,25){$\tilde{a}$}
        \end{overpic}
         &
        $\xrightarrow{f}$
         &
       \begin{overpic}[scale=0.28]{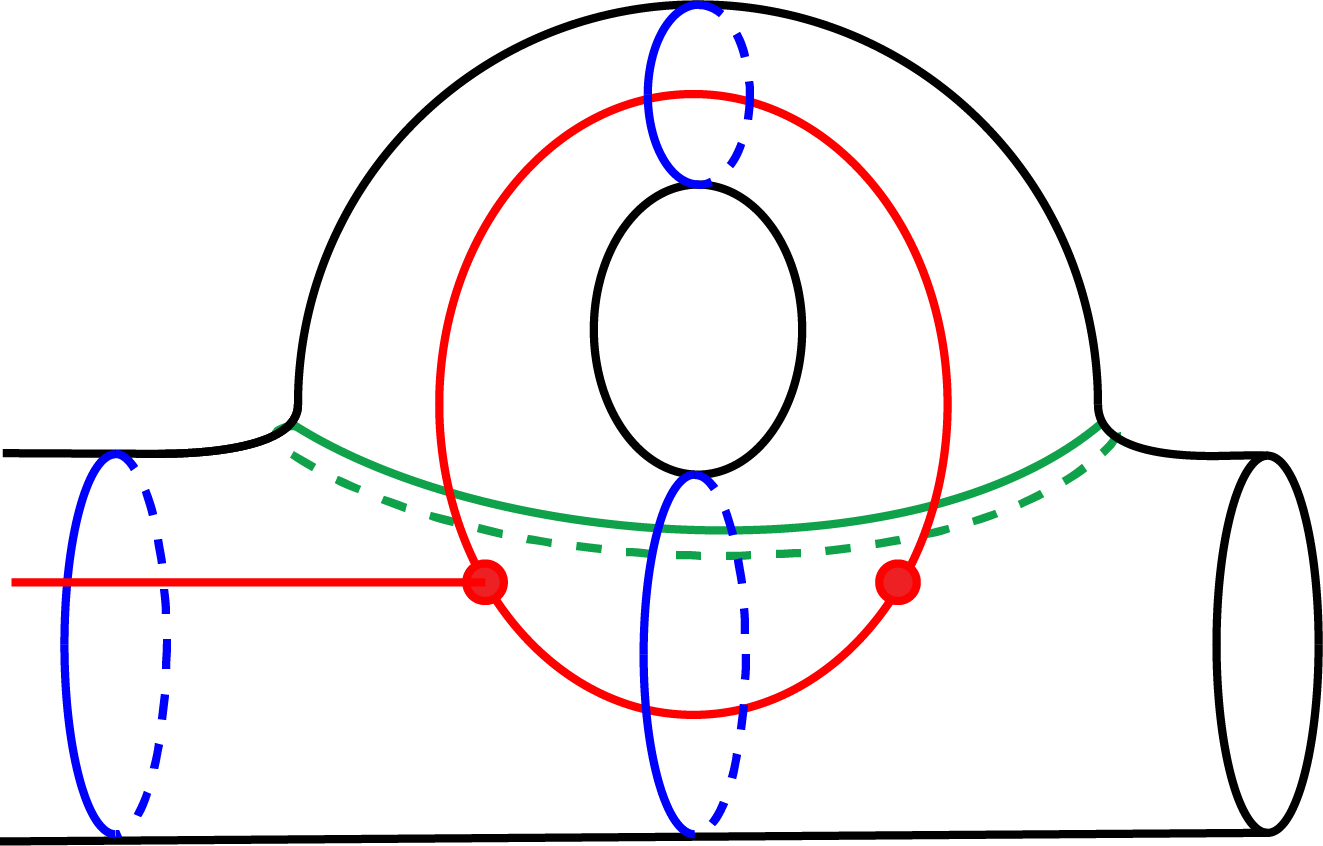}
       \put(45,67){$fa$}
       \put(77,17){$\tilde{a}'$}
       \put(45,-10){$f\tilde{a}$}
        \end{overpic} 
        &
        $\xrightarrow{\text{flip}}$
        &
        \begin{overpic}[scale=0.28]{figs/LoopCase1.png}
        \put(45,67){$fa$}
       \put(70,25){$\tilde{a}'$}
        \end{overpic} \\[.5cm] 
    \end{tabular}
    \caption{The flip move in \Cref{lem:local_diffeo}.}  
\label{fig:Flip}
\end{figure}

\vspace{-.3cm}
\begin{remark*}\label{rmk:h_eta} 
  If $\eta$ is a finite maximal system and $|\eta| \geq 5$,
then $M_\Gamma \cong M_{n,s}
\setminus \partial M_{n,s}$ for some $n,s$ and one verifies that 
$2n + s \geq 6$.  
By \Cref{thm:geom_rigidity}, there exists a
finite strongly rigid set $Y_\eta \subset \sph(M_\Gamma)$ 
satisfying the hypotheses of \Cref{lem:local_diffeo}, in which case
a locally injective map $Y_\eta \to \sph(M_{\Gamma'})$ extends
uniquely to an automorphism $\sph(M_\Gamma) \to \sph(M_{\Gamma'})$;
by \Cref{thm:IvanovRigidity0}, this automorphism is induced by a
diffeomorphism $M_\Gamma \to M_{\Gamma'}$. 
\end{remark*}

Suppose that $\Gamma$ is connected and 
of infinite or finite type such that
$\dim(\sph(M_\Gamma)) \geq 5$. 
Let $\sigma \subset \sph(M_\Gamma)$ be a maximal sphere system 
and let $\Omega$ be the collection of
sets of distinct spheres $\eta = \{\a,\b,\c,\d,e\} \subset
\sigma$ such that $\bigcup_{t\in \eta} e_t \subset \Delta_\sigma$
is connected.  For each 
$\eta \in \Omega$, denote by
$M_\eta$ the complementary component of $\sigma \setminus \eta$
that contains $\eta$
(and in particular is not a pair of pants), and let $Y_\eta \subset
\lk(\sigma \setminus \eta) \cong \sph(M_\eta)$ as
in \Cref{rmk:h_eta}.  
Define $Z_\sigma \subset
\sph(M_\Gamma)$ to be the full subcomplex induced by the $Y_\eta$:
\[
  Z_\sigma \coloneq \braket{Y_\eta}_{\eta \in \Omega}\:.
\]
Note that $\sigma \subset X_\sigma \subset Z_\sigma$, and that
$Z_\sigma$ is topologically locally finite by construction.

\begin{proposition}\label{prop:Zsigma_rigid}
  Suppose that $f : Z_\sigma \to \sph(M_{\Gamma'})$ is a locally
  injective simplicial map such that $f\sigma$ is maximal.  Then $f$
  extends to a unique automorphism $\sph(M_\Gamma) \to
  \sph(M_{\Gamma'})$.
\end{proposition}

\begin{proof}
  For $\eta \in \Omega$, let $M_{f\eta}\subset M_{\Gamma'}$ denote
  the complementary component of $f\sigma \setminus f\eta$
  containing $f\eta$.  By \Cref{rmk:h_eta}, $f|_{Y_\eta}$
  extends to a unique isomorphism $\lk(\sigma \setminus \eta) \to
  \lk(f\sigma \setminus f\eta)$, which is 
  induced by a diffeomorphism $h_\eta : M_\eta \to M_{f\eta}$ 
  agreeing with $f$ over $X_\eta\subset Y_\eta$. 

  We first prove that $f|_\sigma$ 
  is an edge isomorphism $E(\Delta_\sigma)
  \to E(\Delta_{f\sigma})$ without a $K_3,K_{1,3}$-pair.  
  If $\a,\b \in \sigma$ are such that $e_a$ and $e_b$ are
  adjacent in $\Delta_\sigma$, 
  fix $\eta \in \Omega$ such that $\a,\b
  \in \eta$.  Then from the above we have 
  $e_\a \cup e_\b \cong e_{h_\eta \a} \cup e_{h_\eta \b} 
  = e_{f\a} \cup e_{f\b}$, preserving ordering.  A similar argument
  implies that $e_\a$ is a loop if and only if $e_{f\a}$ is a
  loop, and prohibits a $K_3,K_{1,3}$-pair (\textit{e.g.}\ fix $\eta
  \supset K_3$).  It remains to show that when
  $e_\a,e_\b$ are non-adjacent, so are $e_{f\a},e_{f\b}$.
  Suppose that $e_{f\a}, e_{f\b}$
  are adjacent and  $e_{\a}$ is a loop.  
  As $e_{f\a}$ is a loop, 
  $e_{f\b}$ is the unique sphere adjacent to $e_{f\a}$.  
  Let $\c \in
  \sigma$ be a sphere so that $e_\c$ is adjacent to $e_\a$ in
  $\Delta_\sigma$. 
Since $e_\c$ and $e_\a$ are adjacent, $e_\c \cup e_\a \cong
e_{f\c} \cup e_{f\a}$ and $e_{f\c}, e_{f\a}$ are adjacent, hence
$f\c = f\b$;  by injectivity 
  $\b = \c$ and $e_\a,e_\b$ are adjacent.  
  Finally, if neither $e_\a,e_\b$ are loops
  and $e_\a,e_\b$ are non-adjacent, then, since $X_\sigma\subset Z_\sigma$, 
  both $\lk(\sigma\setminus a)\subset Z_\sigma$ and $\lk(\sigma\setminus b)\subset Z_\sigma$.
  Consequently $\lk(\sigma\setminus a)\ast\lk(\sigma\setminus b)=\lk(\sigma\setminus\{a,b\})\subset Z_\sigma$.
  Now, as in the proof of
  \Cref{lem:local_diffeo}, we know that $\lk(\sigma\setminus b)\cong\sph(M_{0,4})\ast\sph(M_{0,4})$ 
  does not embed into $\lk(f\sigma\setminus\{fa,fb\})$ if $e_{fa}$ and $e_{fb}$ are adjacent. 
  Hence $e_{f\a}$ and $e_{f\b}$ must be non-adjacent. 

  By \Cref{thm:whitney_multi} and \Cref{cor:whitney_infinite},
  $f|_\sigma$ is induced by an isomorphism $\Delta_\sigma \to
  \Delta_{f\sigma}$.  
  Let $h$ be a diffeomorphism constructed as in \Cref{sec:diffeo};
  we show that $h_*$ extends $f$.
   For any  
  $\a,\b,\c \subset \sigma$ such that $e_\a \cup e_\b
  \cup e_\c \subset \Delta_\sigma$ is connected, fix $\rho \in
  \Omega$ containing $\a,\b,\c$. 
  Then $f$ extends to an isomorphism $\lk(\sigma \setminus \rho) \to
  \lk(f\sigma \setminus f\rho)$, which suffices to apply 
  \Cref{rmk:compatibility_sigma}: 
  $h_*$ agrees with $f$ over $X_\sigma$.
  For $\eta \in \Omega$,
  $h_*$ agrees with $f$ on $X_\eta = X_\sigma \cap \lk(\sigma
  \setminus \eta)$, hence by \Cref{rmk:uniqueness} 
  $(h|_{M_\eta})_* =
  (h_\eta)_*$. 
  Since $(h_\eta)_*$ agrees with $f$ over $Y_\eta$, we conclude.
  Finally, to show
  uniqueness, by \Cref{thm:IvanovRigidity0} 
  any isomorphism extending $f$ is induced by a diffeomorphism $g$
  that agrees with $h_*$ over $X_\sigma$: 
  $g_* = h_*$ again by \Cref{rmk:uniqueness}.
\end{proof}

If $\sigma,\sigma'$ are maximal sphere systems that differ by a flip
move, then $Z_\sigma \cup Z_{\sigma'}$ exhibits the same rigidity.
Suppose that $f : Z_\sigma \cup Z_{\sigma'} \to \sph(M_{\Gamma'})$
is a locally injective simplicial map and $f\sigma, f\sigma'$ are
maximal.  Apply \Cref{lem:compatibility_flip}: the isomorphism 
extensions with
respect to $\sigma,\sigma'$ are identical.  More generally, let 
$\mathcal P$ be a finite collection of maximal sphere systems
for which any two
$\rho,\rho' \in \mathcal{P}$ 
differ by a sequence of successive flip moves
in $\mathcal{P}$.  Then $Z_{\mathcal{P}} \coloneq \bigcup_{\rho \in
\mathcal{P}} Z_{\rho}$ is strongly rigid over maps $f$ which
preserve the maximality of all $\rho \in \mathcal{P}$.
Let $\mathcal{P}_i$ be a nested family of such sets such that every
sphere in $\sph(M_\Gamma)$ is contained in some $\rho \in
\mathcal{P}_i$ for some $i$.  Then $Z_{\mathcal{P}_i}$ is a (nested)
exhaustion of $\sph(M_\Gamma)$ by topologically 
locally finite subcomplexes 
that are strongly rigid over maximal maps,
proving \Cref{thm:loc_finite_rigidity}. 

\subsection{Non-existence of rigid sets}\label{sec:no_rigid} 

We now show that for all infinite-type graphs $\Gamma$, there is no
way to strengthen \Cref{thm:loc_finite_rigidity} to remove the
assumption that the locally injective maps involved send maximal
sphere systems to maximal sphere systems. We do this by producing nonsurjective embeddings of $\sph(M_{\G})$
into itself so that the image does not contain any sphere systems
which are maximal in $M_{\G}$. It follows immediately from this that no subcomplex of $\sph(M_{\G})$ can be strongly rigid.
\par 
The construction below is inspired by
a similar construction for surfaces, showing that the curve graph of a
hyperbolic surface embeds into the curve graph of the same surface
but with one extra puncture: see Theorem 2.3 of \cite{Rafi_2009}.
We start by discussing two technical lemmas about pairs of spheres in normal form. The first is a consequence of Proposition 1.1 in \cite{hatcher2002algebraic}.

\begin{lemma}\label{lem:disjointSpheresNormalForm}
    Suppose $\a$ and $\b$ are essential spheres in $M_{\G}$ whose homotopy classes can be realized disjointly. Fix a maximal sphere system $\sigma$. Then $\a$ and $\b$ contain disjoint homotopy representatives which are in normal form with respect to $\sigma$.
\end{lemma}

\begin{lemma}\label{lem:homotopyMinusPoint}
    Suppose $\a$ and $\a'$ are homotopic spheres in normal form with respect to a maximal sphere system $\sigma$. Let $p$ be a point in a component of $\sigma$ not contained in either $\a$ or $a'$. Then there is a homotopy from $a$ to $a'$ whose image does not intersect $p$.
\end{lemma}
\begin{proof}
    This follows directly from the proof of Proposition 1.2 in \cite{hatcher2002algebraic}. Namely, we first take two homotopic lifts $\tilde{a}$ and $\tilde{a}'$ of $a$ and $a'$, respectively, as well as a lift $\widetilde{\sigma}$ of $\sigma$. Following the result in \cite{hatcher2002algebraic}, one can first homotope $\tilde{a}$ via a homotopy whose image is disjoint from all the lifts of $p$ so that intersection of each piece of $\tilde{a}$ with $\widetilde{\sigma}$ agrees with the intersection of $\tilde{a}'$ with $\widetilde{\sigma}$ in the component of $\widetilde{M_{\G}}\setminus \widetilde{\sigma}$ that they both lie in. This homotopy can be built by choosing a neighborhood of each element of $\widetilde{\sigma}$ so that the intersections of $\tilde{a}$ and $\tilde{a}'$ with each neighborhood are either both empty or a cylinder. Note that  Proposition 1.2 in \cite{hatcher2002algebraic} implies that such neighborhoods exist, as the intersections of $\tilde{a}$ and $\tilde{a}'$ with the components of $\widetilde{M_{\G}}\setminus \widetilde{\sigma}$ agree combinatorially (that is, the pieces that show up in each component of $\widetilde{M_{\G}}\setminus \widetilde{\sigma}$ are the same for both spheres). Now, in each such neighborhood with nonempty intersection with $\tilde{a}$ and $\tilde{a}'$, one can slide $\tilde{a}$ so that the desired intersection agreement holds. In each neighborhood containing a lift of $p$, such a homotopy can be chosen to avoid this lift as the complement of a point in $S^2$ is simply connected. 
    \par 
    Now that the intersection circles of $\tilde{a}$ and
    $\tilde{a}'$ with $\widetilde{\sigma}$ agree, one can follow the
    argument of Proposition 1.2 in \cite{hatcher2002algebraic}
    verbatim to homotope $\tilde{a}$ to $\tilde{a}'$ via a homotopy supported in the complement of $\widetilde{\sigma}$, finishing the proof.
\end{proof}

The desired embedding from $\sph(M_{\G})$ into itself will be factored into two maps, the first given by the following lemma. Given a locally finite graph $\G$, let $M_{\G}^*$ denote the doubled handlebody $M_{\G}$ with an extra puncture.

\begin{lemma}\label{lem:puncturedManifoldSphereEmbedding}
    For any locally finite graph $\G$, there is a simplicial embedding of $\sph(M_{\G})$ into $\sph(M_{\G}^*)$.
\end{lemma}
\begin{proof}
     Put every element of $\sph(M_{\G})$ in normal form with respect to some maximal system $\sigma$. Place the extra puncture of $M_{\G}^*$ in a component $s$ of $\sigma$ so that the puncture is not contained in any of the fixed normal form representatives of the elements of $\sph(M_{\G})$. Such a choice for the puncture can be made since the intersections of all the spheres with $s$ is a countable collection of embedded circles, which is measure $0$ in $s$ (one needs to isotope the normal form representative of $s$ off itself by a small isotopy as well, so the puncture is disjoint from it too).
     \par 
     Then \Cref{lem:disjointSpheresNormalForm} implies that given any two spheres in normal form $\a$ and $\b$ which can be realized disjointly in $M_{\G}$, there are spheres $\a'$, $\b'$ also in normal form homotopic to $\a$, $\b$, respectively, which disjoint from each other. From \Cref{lem:homotopyMinusPoint} we obtain homotopies (and thus isotopies by \cite{laudenbach_sur_1973}) in $M_{\G}^*$ from $\a$ to $\a'$ and from $\b$ to $\b'$. In particular, it follows that two spheres which can be realized disjointly in $M_{\G}$ can also be realized disjointly in $M_{\G}^*$, if the elements of $\sph(M_{\G})$ are realized in normal form first, and then included into $M_{\G}^*$ as above. In particular, this induces a simplicial embedding of $\sph(M_{\G})$ into $\sph(M_{\G}^*)$, as desired.
\end{proof}
\begin{prop}\label{prop:nonMaximalEmbedding}
    Suppose $\Gamma$ is an infinite-type graph. Then there is a simplicial embedding of $\sph(M_{\G})$ into itself so that the image of every maximal system is not maximal. 
\end{prop}
\begin{proof}
    We can classify infinite-type graphs into three types.
    \begin{enumerate}
        \item Graphs with infinite rank.
        \item Finite rank graphs with infinitely many isolated ends..
        \item Finite rank graphs whose space of ends is a Cantor set along with possibly finitely many isolated ends.
    \end{enumerate}
     To see this, suppose $\Gamma$ is not of the first two types. Then as we are assuming $\Gamma$ is infinite-type, it must have an infinite space of ends. It follows by a theorem of Brouwer that, if $E'$ is the space of ends of $\Gamma$ without its isolated ends, then $E'$ is homeomorphic to a Cantor  set\cite{brouwer1910Cantor}. This is because $E'$ is compact, perfect, totally disconnected, and metrizable.
    \par  
    We choose a separating sphere $\a$ in each case. For type (1) graphs, let $\a$ cut off a copy of $M_{1,1}$. For type (2) graphs, let $\a$ cut off $2$ isolated ends. Finally, for type (3) graphs, let $\a$ cut off a Cantor set of ends on one side. Then there is an embedding $\sph(M_{\G}^*)$ into $\sph(M_{\G})$ which is given by removing $\a$ from $M_{\G}$ and choosing a diffeomorphism from $M_{\G}^*$ to a component of $M_{\G}\setminus \a$, sending the extra puncture of $M_{\G}^*$ to the puncture corresponding to $\a$. Such a diffeomorphism exists, as in any case the characteristic triple of $M_{\G}^*$ is the same as one of the two components of $M_{\G}\setminus \a$.
    \par 
    The composition of the embedding from \Cref{lem:puncturedManifoldSphereEmbedding} and that from the previous paragraph sends $\sph(M_{\G})$ to a subcomplex of itself so that in either case, $\a$ is not in the image of this map and the image of every sphere can be realized disjointly from $\a$. In particular, every maximal sphere system is sent to something not maximal.
\end{proof}

\section{Geometric rigidity in low complexity cases} \label{sec:sporadic}
In this section we consider the existence of finite rigid sets of $\sph(M_{n,s})$ for $n,s$ not covered by \Cref{thm:geom_rigidity} The cases are $M_{0,4}, M_{0,5}, M_{1,0}, M_{1,1}, M_{1,2}, M_{1,3}, M_{2,0}$, and $M_{2,1}$. In the first four cases the complex is finite so the existence of finite rigid sets is trivial. It follows from Proposition 26 of \cite{bering2024finite} that $\mathcal{S}(M_{2,0})$ has no finite rigid sets.
\par 
The remaining cases are then $M_{1,2}, M_{1,3}$, and $M_{2,1}$. We give a direct argument in the first case. 

\begin{lemma}
    The graph $\sph(M_{1,2})$ is geometrically rigid. On the other hand, $\sph(M_{1,2})$ has no finite rigid sets.
\end{lemma}
\begin{proof}
    That $\sph(M_{1,2})$ is geometrically rigid follows directly from \Cref{thm:IvanovRigidity0}. To show that it has no finite rigid sets, we will show that $\sph(M_{1,2})$ is isomorphic to the graph in \Cref{fig:M12SphereGraph}, and deduce that this graph has no finite rigid sets.
    \par 
    Given a non-separating sphere $\a$ of $M_{1,2}$, there are exactly three spheres disjoint from $\a$. 
To see why, note that $M_{1,2} \setminus \mathrm{nbd}(\a) \cong M_{0,4}$, which has exactly three distinct spheres, say $\a_1, \a_2$, and $\b$.  See \Cref{fig:M12Rigid.png}.
 Observe that $\a_1, \a_2$ and $b$ are disjoint from $a$, so in $\sph(M_{1,2})$, $a$ is a trivalent vertex. 
    The spheres $\a_1$ and $\a_2$ are non-separating,
    and $b$ is separating.  The complementary components of $b$ are a pair of pants and a one-holed torus.  A pair of pants contains no essential spheres, and a one-holed torus contains a single homotopy class of spheres.  Consequently, any essential sphere in $M_{1,2}$ disjoint from $a$ must intersect $b$.  Therefore, $b$ has valence one in the sphere complex.
    Further, this separating sphere is only disjoint from $S$.
    
    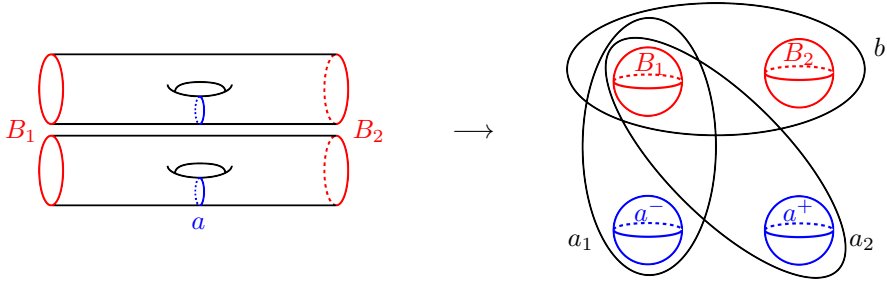
\begin{figure}[ht!]
        \centering
        \begin{overpic}[width=11cm]{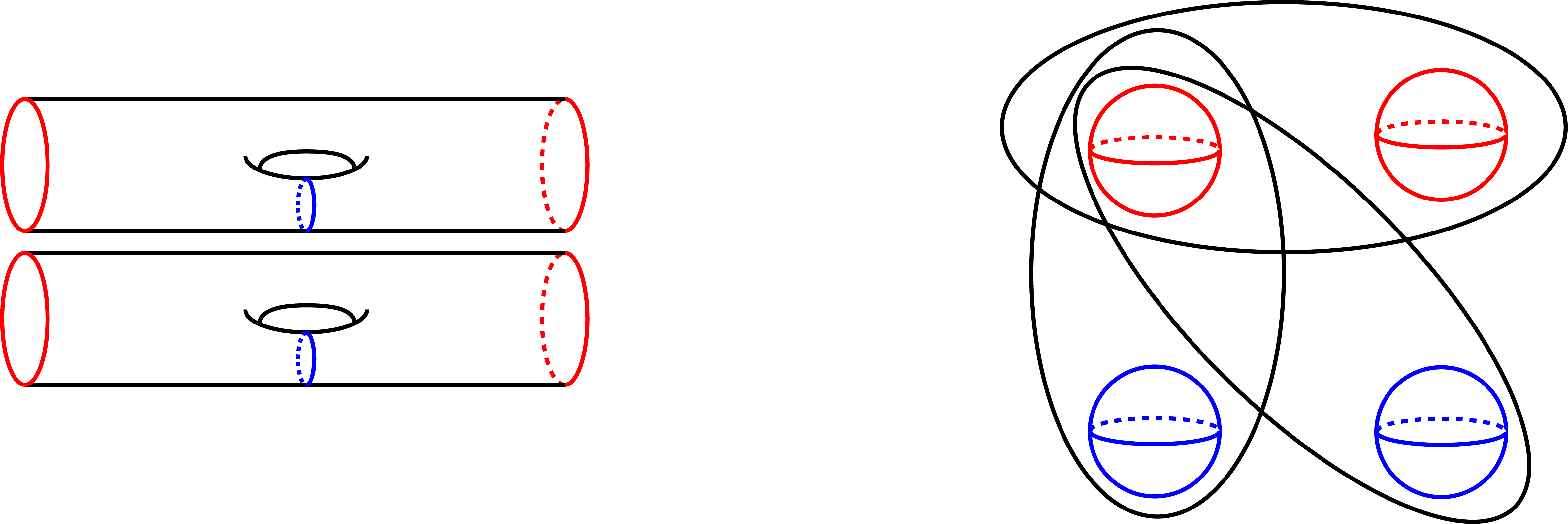}
            \put(-4,17){\color{red}{$B_1$}}
            \put(38,17){\color{red}{$B_2$}}
            \put(18.5,6){\color{blue}{$\a$}}
            \put(50,17){$\longrightarrow$}
            \put(101,27){$\b$}
            \put(98,4){$\a_2$}
            \put(64,4){$\a_1$}
            \put(72,7){\color{blue}{$\a^-$}}
            \put(90,7){\color{blue}{$\a^+$}}
            \put(72,25.1){\color{red}{$B_1$}}
            \put(90,26.4){\color{red}{$B_2$}}
        \end{overpic}
        \caption{On the left, $M_{1,2}$ is pictured with boundary components $B_1$ and $B_2$ and a non-separating sphere $a$. To visualize the spheres $\a_1, \a_2$, and $\b$, it is useful to cut along $\a$, as seen on the right.  }
        \label{fig:M12Rigid.png}
    \end{figure}

    \par 
    For completeness, we show explicitly that $\sph(M_{1,2})$ is an infinite graph. Let $h$ be a homeomorphism which pushes one of the boundary components of $M_{1,2}$ around the manifold once, and let $\a$ be a fixed non-separating sphere. Then $h^n(\a)$ is not homotopic to $\a$ for any $n\geq 1$. If it were, they would have lifts in the universal cover $\widetilde{M}_{1,2}$ which are homotopic. But it is clear that the decompositions of the boundary components of $\widetilde{M}_{1,2}$ induced by any lifts of $\a$ and $h^n(\a)$ must differ, and thus they cannot be homotopic. See \Cref{fig:M12UnivCover}.
    
    \begin{figure}[ht!]
        \centering
        \begin{overpic}[width=8cm]{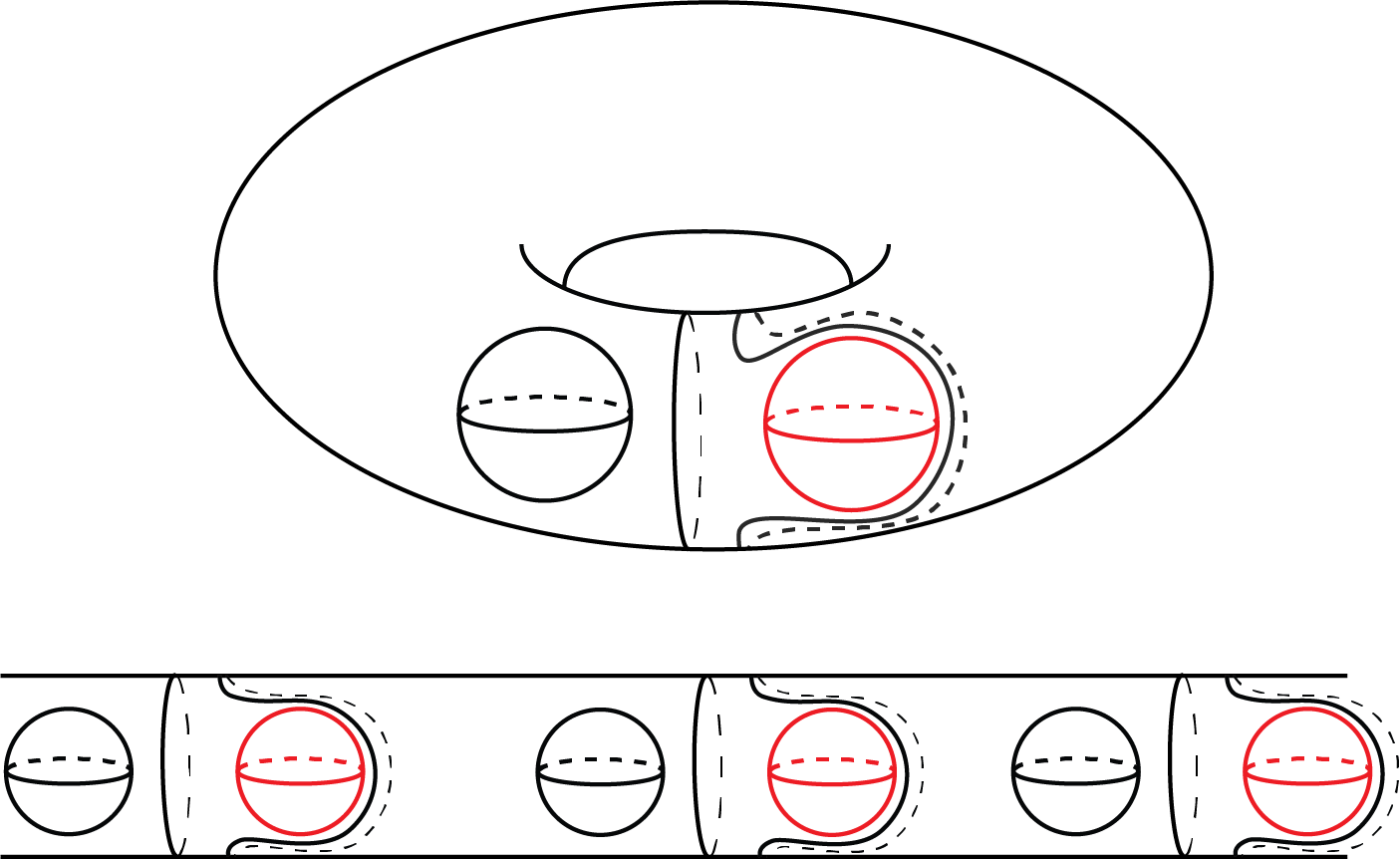}
        \put(46,18){$\a$}
        \put(55,18){$h(\a)$}
        \end{overpic}
        \caption{The spheres $\a$ and $h(\a)$ cannot be homotopic as their none of the lifts in the universal cover are. A similar picture works for $\a$ and $h^n(\a)$ for all $n$.}
        \label{fig:M12UnivCover}
    \end{figure}

    \par
    Thus, $\sph(M_{1,2})$ is an infinite graph consisting of a collection of trivalent vertices, each of which is connected to two other trivalent vertices and one valence $1$ vertex. Such a graph is isomorphic to the real line with a vertex at each integer point with a another edge attached at each vertex connecting to a valence $1$ vertex, as in \Cref{fig:M12SphereGraph}.   
    \par 
    Suppose $X$ is a finite rigid subgraph for $\sph(M_{1,2})$. Then $X$ must be larger than a point, or else one could send a non-separating vertex to a separating vertex or vice versa. It must also be connected, or else one could embed $X$ in many ways by fixing the image of one component and letting another vary, and all but finitely many of these maps cannot be induced by an automorphism of $\sph(M_{1,2})$. Thus assume $X$ is connected. By the structure of $\sph(M_{1,2})$ there must be some vertex $v\in X^{(0)}$ which is a non-separating sphere which is connected by an edge to a vertex $w\in X^{(0)}$ which is valence $1$ in $X$. Then we can embed $X$ into $\sph(M_{1,2})$ so that $v$ is sent to itself, and $w$ is sent to a separating sphere if it is non-separating, and to a non-separating sphere if it is separating. This is a contradiction, so $\sph(M_{1,2})$ has no finite rigid sets.
\end{proof}
    \begin{figure}[ht!]
        \centering
        \begin{overpic}[width=6cm]{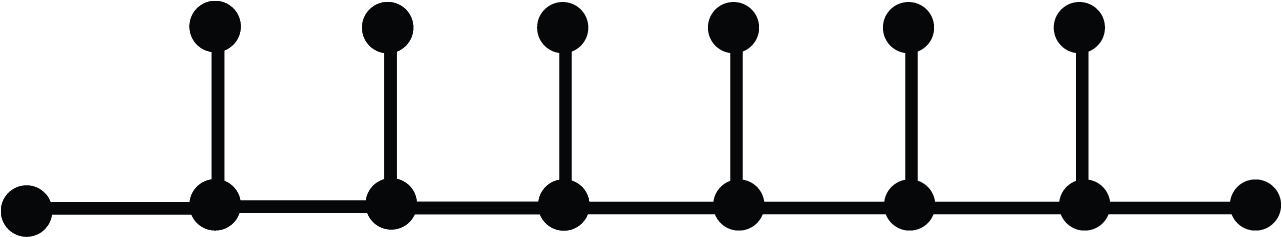}
            \put(-8,2){\dots}
            \put(102,2){\dots}
        \end{overpic}
        \caption{The graph $\sph(M_{1,2})$.}
        \label{fig:M12SphereGraph}
    \end{figure}

 We ask the following question for the final remaining cases.
\begin{question}
    Do $\sph(M_{1,3})$ and $\sph(M_{2,1})$ have finite rigid sets?
\end{question}

%\nocite{*}
\begingroup
\DeclareEmphSequence{\itshape} 
\bibliographystyle{alpha}
\bibliography{main}
\endgroup
\end{document}